\documentclass[10pt]{amsart}
\usepackage{graphicx}
\usepackage{mathrsfs}
\usepackage{epsfig}
\usepackage{amsmath}
\usepackage{amsfonts}
\usepackage{amssymb}
\usepackage{amssymb,amscd}
\usepackage{tikz}
\usepackage{verbatim}
\usepackage{enumerate}
\usepackage{booktabs}
\usepackage[all]{xy}
\usepackage{subfigure}
\usepackage{url}

\newtheorem{thm}{Theorem}[section]
\newtheorem{cor}[thm]{Corollary}
\newtheorem{lem}[thm]{Lemma}
\newtheorem{prop}[thm]{Proposition}
\newtheorem{conj}[thm]{Conjecture}
\theoremstyle{definition}
\newtheorem{defn}[thm]{Definition}

\newtheorem*{ack}{Acknowledgements}
\numberwithin{equation}{section}

\newcommand{\Z}{\mathbb{Z}}

\newcommand{\hc}{\mathbf{H}^2_{\mathbb{C}}}

\begin{document}

\title[]{ Spherical CR uniformization of   the   magic 3-manifold }
\author{Jiming Ma}
\address{School of Mathematical Sciences, Fudan University, Shanghai, 200433, P. R. China}
\email{majiming@fudan.edu.cn}
\author{Baohua Xie}
\address{School of Mathematics, Hunan University, Changsha, 410082, China
}
\email{xiexbh@hnu.edu.cn}

\keywords{Complex hyperbolic geometry, Spherical CR uniformization, Triangle groups, Cusped hyperbolic 3-manifolds.}

\subjclass[2010]{20H10, 57M50, 22E40, 51M10.}

\date{\today}

\thanks{The first author was supported by   NSFC (No.12171092).  The
second author  was supported by NSFC (No.11871202, No.12271148).}
\maketitle

\begin{abstract} 	
We show the 3-manifold at infinity of the complex hyperbolic triangle group  $\Delta_{3,\infty,\infty;\infty}$    is  the three-cusped   magic 3-manifold $6_1^3$. We also show the 3-manifold at infinity of the complex hyperbolic triangle group  $\Delta_{3,4,\infty;\infty}$ is  the two-cusped   3-manifold  $m295$ in the Snappy  Census, which is a 3-manifold obtained by Dehn filling on one cusp of $6_1^3$. In particular,  hyperbolic 3-manifolds $6_1^3$ and  $m295$  admit  spherical CR uniformizations.

These results support our conjecture that  the 3-manifold at infinity of the complex hyperbolic triangle group $\Delta_{3,n,m;\infty}$ is the one-cusped hyperbolic 3-manifold obtained from the    magic manifold  $6_1^3$  via Dehn fillings   with filling slopes $(n-2)$ and $(m-2)$ on the first two cusps of  it.
\end{abstract}



\section{Introduction}

\subsection{Motivation}
Thurston's work on 3-manifolds has shown that geometry has an important role to play in the study of topology of 3-manifolds.
There is a very close relationship between the topological properties of 3-manifolds and the existence of geometric structures.
A {\it spherical CR-structure} on a smooth 3-manifold $M$ is a maximal collection of distinguished charts modeled on the boundary $\partial \mathbf{H}^2_{\mathbb C}$
of the complex hyperbolic space $\mathbf{H}^2_{\mathbb C}$, where coordinates changes  are restrictions of elements of  $\mathbf{PU}(2,1)$.
In other words, a {\it spherical CR-structure} is a $(G,X)$-structure with $G=\mathbf{PU}(2,1)$ and $X=S^3$. In contrast to the results on other geometric structures carried on 3-manifolds, there are relatively few examples known about spherical CR-structures.

In general, it is very difficult to determine whether a 3-manifold admits a spherical CR-structure or not. Some of the first examples were given by
Burns-Shnider \cite{BS:1976}. Three-manifolds with $Nil^{3}$-geometry naturally admit such structures, but by Goldman \cite{Goldman:1983}, any closed 3-manifold with Euclidean or $Sol^3$-geometry
does not admit such structures.

We are interested in an important class of spherical CR-structures, called uniformizable spherical CR-structures. A spherical CR-structure on a 3-manifold $M$  is {\it uniformizable} if it is
obtained as $M=\Gamma\backslash \Omega_{\Gamma}$, where  $\Omega_{\Gamma}\subset \partial \mathbf{H}^2_{\mathbb C}$ is discontinuity region of a discrete subgroup  $\Gamma$ acting on $\partial \mathbf{H}^2_{\mathbb C}=S^3$.
Constructing discrete subgroups of $\mathbf{PU}(2,1)$  can be used to constructed spherical CR-structures on 3-manifolds.
Thus, the study of the geometry of discrete subgroups of  $\mathbf{PU}(2,1)$ is crucial to the understanding of uniformizable spherical CR-structures.
Complex hyperbolic triangle groups provide  rich examples of such discrete subgroups. As far as we know, almost all known examples of uniformizable spherical CR-structures are relate to complex hyperbolic triangle groups.

Let $\Delta_{p,q,r}$ be the abstract $(p,q,r)$ reflection triangle group with the presentation
$$\langle \sigma_1, \sigma_2, \sigma_3 | \sigma^2_1=\sigma^2_2=\sigma^2_3=(\sigma_2 \sigma_3)^p=(\sigma_3 \sigma_1)^q=(\sigma_1 \sigma_2)^r=id \rangle,$$
where $p,q,r$ are positive integers or $\infty$ satisfying $1/p+1/q+1/r<1$. One can choose the integers so that $p \leq q \leq r$. If $p,q$ or $r$ equals $\infty$, then
the corresponding relation does not appear.
A \emph{complex hyperbolic $(p,q,r)$ triangle group} is a representation $\rho$ of $\Delta_{p,q,r}$ into $\mathbf{PU}(2,1)$
where the generators fix complex lines, we denote $\rho(\sigma_{i})$ by $I_{i}$. It is well known  that the space of $(p,q,r)$-complex reflection triangle groups has real dimension one if $3 \leq p \leq q \leq r$. Sometimes, we denote the  representation of the triangle group $\Delta_{p,q,r}$ into $\mathbf{PU}(2,1)$ such that $I_1 I_3I_2 I_3$ of order $n$ by $\Delta_{p,q,r;n}$.

Richard Schwartz has conjectured the necessary and sufficient condition for a complex hyperbolic $(p,q,r)$ triangle group $\langle I_1,I_2,I_3\rangle < \mathbf{PU}(2,1)$ to be a discrete and faithful  representation of $\Delta_{p,q,r}$ \cite{schwartz-icm}. Schwartz's conjecture has been proved in a few cases.

We now provide a brief historical overview, before discussing our results.
Schwartz proved the following theorem, first conjectured by Goldman and Parker \cite{GoPa}.
\begin{thm}[Schwartz \cite{Schwartz:2001ann, schwartz:2006}]
	Let $\Gamma=\langle I_1, I_2, I_3 \rangle < \mathbf{PU}(2,1)$ be a complex hyperbolic ideal triangle group.
	If $I_1 I_2 I_3$ is not elliptic, then $\Gamma$ is discrete and faithful.
	Moreover, if $I_1 I_2 I_3$ is elliptic, then $\Gamma$ is not discrete.
\end{thm}

Furthermore, he analyzed the group when $I_1 I_2 I_3$ is parabolic.
\begin{thm}[Schwartz \cite{Schwartz:2001acta}]
	Let $\Gamma=\langle I_1, I_2, I_3 \rangle$ be the complex hyperbolic ideal triangle group with $I_1 I_2 I_3$ being parabolic.
	Let $\Gamma'$ be the even subgroup of $\Gamma$. Then the manifold at infinity of the quotient ${\bf H}^2_{\mathbb C}/{\Gamma'}$ is commensurable with the
	Whitehead link complement in the 3-sphere.
\end{thm}

Recently, Schwartz's conjecture was shown for complex hyperbolic $(3,3,n)$ triangle groups with positive integer $n\geq 4$ in \cite{ParkerWX:2016} and $n=\infty$ in \cite{ParkerWill:2016}.
\begin{thm}[Parker, Wang and Xie \cite{ParkerWX:2016}, Parker and Will \cite{ParkerWill:2016}]
	Let $n \geq 4$, and let $\Gamma=\langle I_1, I_2, I_3 \rangle$ be a complex hyperbolic $(3,3,n)$ triangle group.
	Then $\Gamma$ is a discrete and faithful representation of  the  $(3,3,n)$ triangle group if and only if $I_1 I_3 I_2 I_3$ is not elliptic.
\end{thm}

There are some interesting results on complex hyperbolic $(3,3,n)$ triangle groups with $I_1I_3I_2I_3$ being parabolic.
\begin{thm} [Deraux and Falbel \cite{DerauxF:2015}, Deraux \cite{Deraux:2015} and Acosta \cite{Acosta:2019}]\label{thm:Acosta}
	Let $4 \leq n \leq +\infty $, and let $\Gamma=\langle I_1, I_2, I_3 \rangle$ be a complex hyperbolic $(3,3,n)$ triangle group with $I_1I_3I_2I_3$ being parabolic.
	Let $\Gamma'$ be the even subgroup of $\Gamma$. Then the manifolds at infinity of  the quotient ${\bf H}^2_{\mathbb C}/{\Gamma'}$ is obtained from  Dehn surgery on  one of the cusps of the Whitehead link complement with slope $n-2$.
\end{thm}

Note that the choice of the  meridian-longitude systems of the  Whitehead link complement in Theorem  \ref{thm:Acosta} is different from that in \cite{Acosta:2019}. The  meridian-longitude systems chosen here is from \cite{MartelliP:2006}, which  seems more coherent for the manifolds at infinity of the complex hyperbolic triangle group  $\Delta_{3,m,n; \infty}$ below.

These deformations furnish some of the simplest interesting examples in the still mysterious subject of complex hyperbolic deformations. While some progress has been made in understanding these examples, there is still a lot unknown about them.

\subsection{Our result}

The main purpose of this paper is to study the geometry of triangle groups $\Delta_{3,\infty,\infty;\infty}$ and   $\Delta_{3,4,\infty;\infty}$.
Thompson showed  \cite{Thompson:2010} that $\Delta_{3,\infty,\infty;\infty}$ and $\Delta_{3,4,\infty;\infty}$ are arithmetic subgroups of  $\mathbf{PU}(2,1)$, thus they are discrete. We will  identify the manifolds at infinity for them via Ford domains. 
Our main results are the following Theorems \ref{thm:3pp} and \ref{thm:34p}. Theorem \ref {thm:3pp} is  also step one  of a possible approach  to  Conjecture \ref{conj:3mninfty} later.

\begin{thm} \label{thm:3pp}
	Let $\Gamma=\langle I_1, I_2, I_3 \rangle$ be the complex hyperbolic  triangle group $\Delta_{3,\infty,\infty;\infty}$. Then the manifold at infinity of the even subgroup $\langle I_1I_2,I_2I_3\rangle$ of $\Gamma$ is the magic 3-manifold $6_1^3$ in the Snappy  Census.	
\end{thm}

\begin{thm} \label{thm:34p}
Let $\Gamma=\langle I_1, I_2, I_3 \rangle$ be the complex hyperbolic  triangle group $\Delta_{3,4,\infty;\infty}$. Then the manifold at infinity of the even subgroup $\langle I_1I_2,I_2I_3\rangle$ of $\Gamma$ is the  two-cusped hyperbolic 3-manifold  $m295$ in the Snappy  Census.
	
\end{thm}

The  magic 3-manifold is the complement of the  simplest chain link in $\mathbb{S}^3$ with three components
\cite{MartelliP:2006}, which  appears as $6_1^3$ in  Rolfsen's list \cite{Rolfsen}, and it  is a hyperbolic 3-manifold \cite{CullerDunfield:2014}.
 Note also that $m295$ is the two-components link  $9^2_{50}$ in Rolfsen's list \cite{Rolfsen}.  See Figure 	\ref{fig:631m295} for the diagrams of these two links in the 3-sphere.


\begin{figure}
	\begin{minipage}[t]{0.4 \linewidth}
		\centering
		\includegraphics[height=5.cm,width=5.cm]{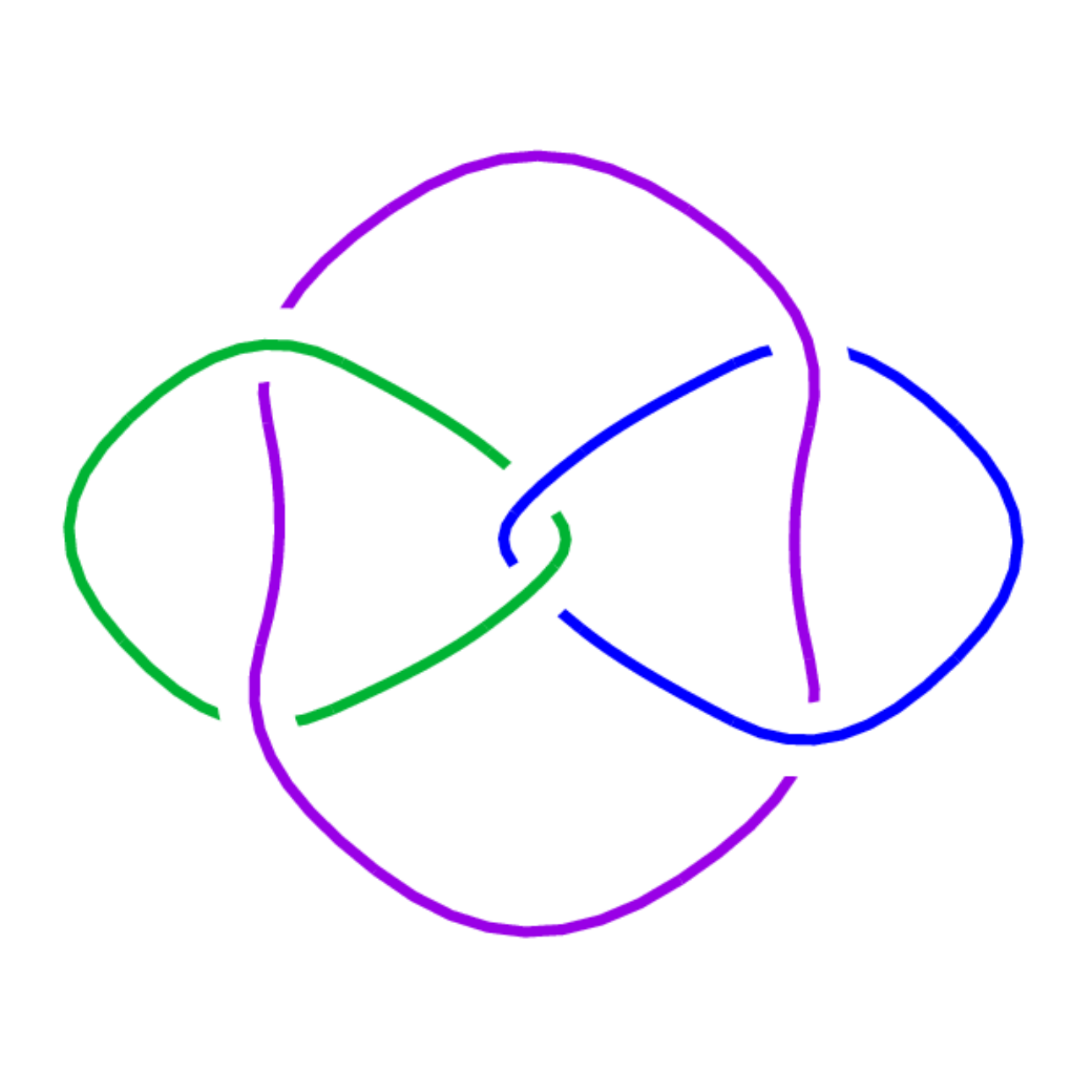}
	\end{minipage}
	\begin{minipage}[t]{0.4 \linewidth}
		\centering
		\includegraphics[height=4.5cm,width=4.5cm]{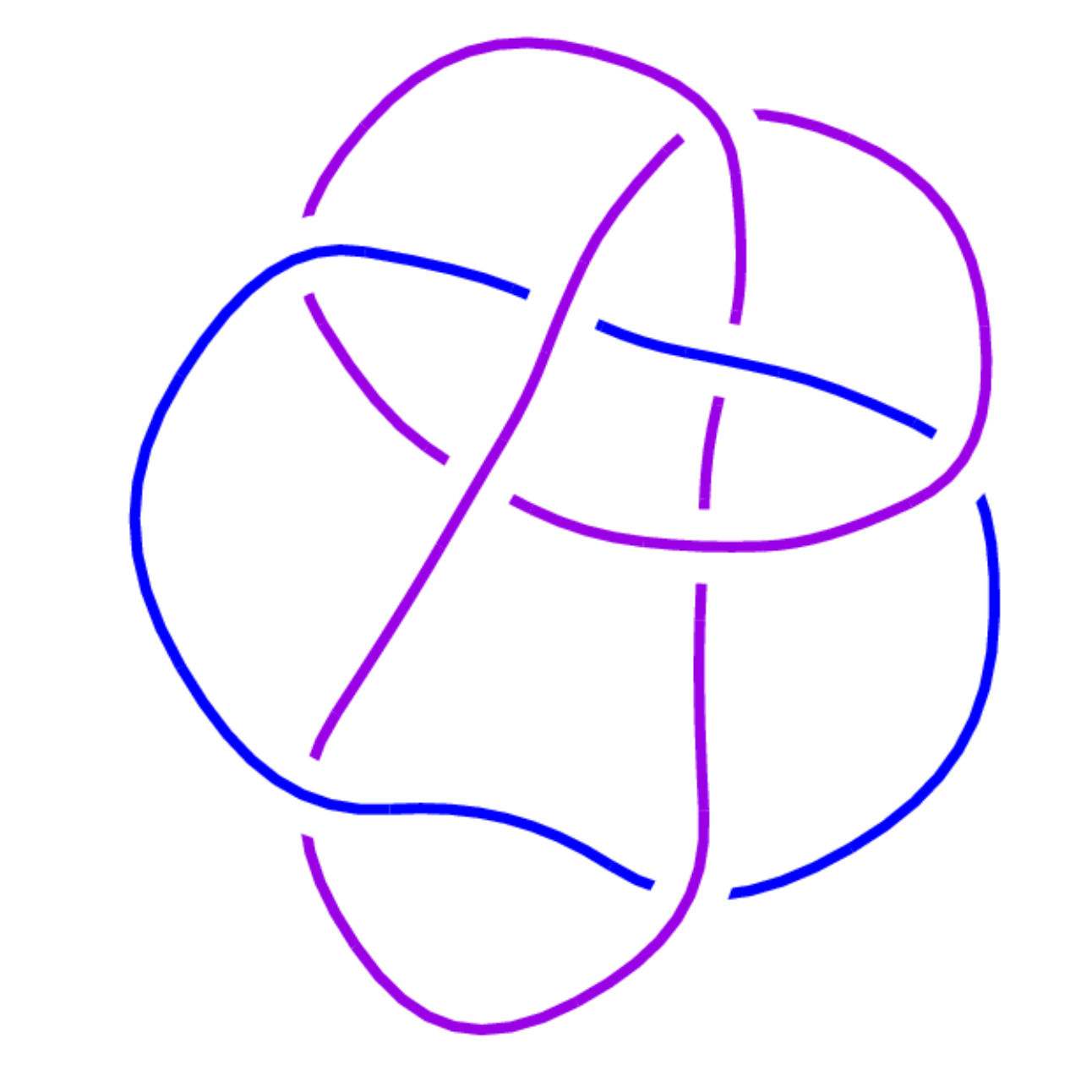}
	\end{minipage}
	\caption{ Left: the magic 3-manifold $6_1^3$. Right: the two-cusped 3-manifold  $m295$. The  Link diagrams are obtained from https://knotinfo.math.indiana.edu.}
	\label{fig:631m295}
\end{figure}



The proofs of   Theorems  \ref{thm:3pp} and  \ref{thm:34p} are  via Ford domains. But they are much more  involved than previous results \cite{Deraux:2016, ParkerWill:2016, jwx, mx}. The main reason is that there are infinitely  many handles in the ideal boundaries of the  Ford domains of our groups. We need to cut  the ideal boundary of the  Ford domain into a simply connected region via hypersurfaces with explicit defining  functions.   We use new kinds of surfaces, say  "crooked-like surfaces" and ruled surfaces,  in the study of complex hyperbolic geometry. We  use  the  computer algebra system to do  some tedious and elementary calculations or to draw some geometric objects; however our paper is independent of the computer and  our proofs are geometric.




First consider the group $\Gamma=\Delta_{3,\infty,\infty;\infty}$. We construct a Ford domain $D_{\Gamma}$ for  $\Gamma$. The ideal boundary  $\partial_{\infty}D_{\Gamma}=D_{\Gamma} \cap \partial \mathbf{H}^2_{\mathbb C}$ is crucial to identify the topology of the manifolds at infinity of $\Gamma$:
 \begin{itemize} \item We construct  "crooked-like surfaces" to cut out a fundamental domain of $Stab_{\Gamma}(\infty)=\mathbb{Z}$ on $\mathbb{S}^{3}\backslash\{\infty\}$. See planes $E_1$ and $E_2$ in Subsection \ref{subsection:3ppfundamentaldomain};

\item   We also need more cutting disks to  cut the fundamental domain of $Stab_{\Gamma}(\infty)$ into a 3-ball. One of these    complicated disks is a union of several ruled surfaces, see the disk $D_2$  in Subsection \ref{subsection:3pphandlebody};

\item   From the combinatorial description of $\partial_{\infty}D_{\Gamma}$ and the cutting disks, we can calculate the fundamental group of the 3-manifold at infinity of $\Gamma$.
  The end result is that the manifold at infinity will be identified with the hyperbolic 3-manifold $6_1^3$.

\end{itemize}

 The group  $\Gamma=\Delta_{3,4,\infty;\infty}$ is even more difficult:
 \begin{itemize} \item  
 	We have to  take a  global combinatorial model of the ideal boundary $\partial _{\infty}D_{\Gamma}$ of  the Ford domain $D_{\Gamma}$ for $\Gamma$ in Subsection
\ref{subsection:globalmodel34p};
\item Then we use geometric argument to show the  geometrical realization of our  combinatorial model is the ideal boundary  of the Ford domain of  $\Gamma$;
\item  We use the combinatorial model to study the 3-manifold at infinity of   $\Gamma$.
  We cut $\partial _{\infty}D_{\Gamma}$ in a geometrical way in the boundary of $\partial _{\infty}D_{\Gamma}$, but in a topological way far from the boundary of  $\partial _{\infty}D_{\Gamma}$, to get a handlebody $H'$;
 \item 
 From $H'$, we can use  the similar argument in the case of  $\Delta_{3,3,\infty;\infty}$ to show the 3-manifold at infinity of  $\Delta_{3,4,\infty;\infty}$  is the 3-manifold   $m295$ in Snappy  Census \cite{CullerDunfield:2014}.
\end{itemize} 

\subsection{A conjectured picture  on  the group $\Delta_{3,m,n;\infty}$}

Our Theorems \ref{thm:3pp} and \ref{thm:34p} are partial results toward a proof of the following conjecture, announced in \cite{mx}:

\begin{conj}\label{conj:3mninfty}	
	The 3-manifold at infinity of the even subgroup of the complex triangle group $\Delta_{3,m,\infty;\infty}$ is the hyperbolic 3-manifold obtained via the Dehn surgery of
	$6_1^3$ on the first  cusp with slope $m-2$. Moreover,
	the manifold at infinity of the even subgroup of the complex triangle group $\Delta_{3,m,n;\infty}$ is the hyperbolic 3-manifold obtained via the Dehn fillings of	$6_1^3$ on the first two cusps with slopes $m-2$ and  $n-2$ respectively.
\end{conj}

We use the meridian-longitude systems of the cusps of $6_1^3$ as in \cite{MartelliP:2006}, which is different from  the meridian-longitude systems  in Snappy.
Theorems \ref{thm:3pp} and  \ref{thm:34p}, and results in  \cite{Acosta:2019,DerauxF:2015,  mx, ParkerWill:2016} can be viewed as evidences of Conjecture \ref{conj:3mninfty}. See  \cite{mx} for more  explanations why this conjecture should be true.

A possible  approach to Conjecture \ref{conj:3mninfty} is based on  the Ford domain of  $\Delta_{3,\infty,\infty;\infty}$ studied in this paper, and then using the method in \cite{Acosta:2019}. Even through the rigorous proof seems highly non-trivial.


\textbf{Outline of the paper}: In Section \ref{sec:background}  we give well known background
material on complex hyperbolic geometry. Section \ref{sec-gens} contains the matrix representations of the complex hyperbolic triangle groups $\Delta_{3,\infty, \infty;\infty}$ and   $\Delta_{3,4,\infty;\infty}$ in $\mathbf{SU}(2,1)$. Section \ref{sec-Ford-3pp}  is devoted to the description of the isometric spheres that bound the Ford domains for the complex hyperbolic triangle group  $\Delta_{3,\infty, \infty;\infty}$.
In Section  \ref{section:3ppmanifold}, we study combinatorial structure of the ideal boundary of the Ford domain for  the  group $\Delta_{3,\infty,\infty;\infty}$ and get the 3-manifold at infinity is the hyperbolic magic 3-manifold $6_1^3$.
 Section \ref{sec-ford-34p}  is devoted to the description of the isometric spheres that bound the Ford domain for the group  $\Delta_{3,4, \infty;\infty}$. In Section \ref{section:3-mfd-34p}, we study combinatorial structure of the ideal boundary of the Ford domain for  the group $\Delta_{3,4,\infty;\infty}$ and show the  3-manifold at infinity is the hyperbolic manifold $m295$.

\begin{ack} The second named author is grateful to LMNS(Laboratory of Mathematics for Nonlinear Science) of Fudan University for its hospitality and financial support, where part of this work was carried out. We are  very grateful to the referees for their insightful comments and helpful suggestions.
\end{ack}

\section{Background}\label{sec:background}
The purpose of this section is to introduce briefly complex hyperbolic geometry. One can refer to Goldman's book \cite{Go} for more details.

\subsection{Complex hyperbolic plane}

Let $\langle {\bf{z}}, {\bf{w}} \rangle={\bf{w}^{*}}H{\bf{z}}$ be the Hermitian form on ${\mathbb{C}}^3$ associated to $H$, where $H$ is the Hermitian matrix
$$
H=\left[
  \begin{array}{ccc}
    0 & 0 & 1 \\
    0 & 1 & 0 \\
    1 & 0 & 0 \\
  \end{array}
\right].
$$
Then ${\mathbb{C}}^3$ is the union of negative cone $V_{-}$, null cone $V_{0}$ and positive cone $V_{+}$, where
\begin{eqnarray*}
  V_{-} &=& \left\{ {\bf{z}}\in {\mathbb{C}}^3\backslash\{0\} : \langle {\bf{z}}, {\bf{z}} \rangle <0 \right\}, \\
  V_{0} &=& \left\{ {\bf{z}}\in {\mathbb{C}}^3\backslash\{0\} : \langle {\bf{z}}, {\bf{z}} \rangle =0 \right\}, \\
  V_{+} &=& \left\{ {\bf{z}}\in {\mathbb{C}}^3\backslash\{0\} : \langle {\bf{z}}, {\bf{z}} \rangle >0 \right\}.
\end{eqnarray*}

\begin{defn}
Let $P: {\mathbb{C}}^3\backslash\{0\} \rightarrow {\mathbb{C}P^2}$ be the projectivization map.
Then the \emph{complex hyperbolic plane} $\hc$ is defined to be $P(V_{-})$ and its {boundary} $\partial \hc$ is defined to be $P(V_{0})$.
Let $d(u,v)$ be the distance between two points $u,v \in \hc$.
Then the \emph{Bergman metric} on complex hyperbolic plane is given by the distance formula
\begin{equation}  \label{eq:bergman-metric}
\cosh^2\left(\frac{d(u,v)}{2}\right)=\frac{\langle {\bf{u}}, {\bf{v}} \rangle\langle {\bf{v}}, {\bf{u}} \rangle}{\langle {\bf{u}}, {\bf{u}} \rangle \langle {\bf{v}}, {\bf{v}} \rangle},
\end{equation}
where ${\bf{u}}, {\bf{v}} \in {\mathbb{C}}^3$ are lifts of $u,v$.
\end{defn}

The standard lift $(z_1,z_2,1)^T$ of $z=(z_1,z_2)\in \mathbb {C}^{2}$ is negative if and only if
$$z_1+|z_2|^2+\overline{z}_1=2{\rm Re}(z_1)+|z_2|^2<0.$$  Thus $\mathbb{P}(V_{-})$ is a paraboloid in ${\mathbb C}^{2}$, called the {\it Siegel domain}.
In these coodinates, the boundary $\mathbb{P}(V_{0})$ is given by
$$2{\rm Re}(z_1)+|z_2|^2=0.$$

Let $\mathcal{N}=\mathbb{C}\times \mathbb{R}$ be the Heisenberg group with product
$$
[z,t]\cdot [\zeta,\nu]=[z+\zeta,t+\nu-2{\rm{Im}}(\bar{z}\zeta)].
$$
Then the boundary of complex hyperbolic plane $\partial \hc$ can be identified to the union $\mathcal{N}\cup \{q_{\infty}\}$, where $q_{\infty}$ is the point at infinity.
The \emph{standard lift} of $q_{\infty}$ and $q=[z,t]\in\mathcal{N}$ in $\mathbb{C}^3$ are
\begin{equation}\label{eq:lift}
  {\bf{q}_{\infty}}=\left[
                    \begin{array}{c}
                      1 \\
                      0 \\
                      0 \\
                    \end{array}
                  \right],  \quad
{\bf{q}}=\left[
           \begin{array}{c}
             \frac{-|z|^2+it}{2} \\
             z \\
             1 \\
           \end{array}
         \right].
\end{equation}
We write $q=[z,t]\in\mathcal{N}$ for $z \in \mathbb{C}$ and $t \in \mathbb{R}$ or $q=[x,y,t]\in\mathcal{N}$ for $x,y,t\in \mathbb{R} $  and $z=x+yi$.

Complex hyperbolic plane and its boundary $\hc \cup \partial \hc$ can be identified to ${\mathcal{N}}\times{\mathbb{R}_{\geq 0}}\cup q_{\infty}$.
Any point $q=(z,t,u)\in{\mathcal{N}}\times{\mathbb{R}_{\geq0}}$ has the standard lift
$$
{\bf{q}}=\left[
           \begin{array}{c}
             \frac{-|z|^2-u+it}{2} \\
             z \\
             1 \\
           \end{array}
         \right].
$$
Here $(z,t,u)$ is called the \emph{horospherical coordinates} of $\overline {\bf H}^2_{\mathbb{C}}=\hc \cup \partial \hc$. The natural projection $\mathcal{N}=\mathbb{C} \times \mathbb{R} \rightarrow \mathbb{C}$ is called the {\it vertical projection}.

\begin{defn}
The \emph{Cygan metric} $d_{\textrm{Cyg}}$ on $\partial \hc \backslash\{q_{\infty}\}$ is defined to be
\begin{equation}\label{eq:cygan-metric}
  d_{\textrm{Cyg}}(p,q)=|2\langle {\bf{p}}, {\bf{q}} \rangle|^{1/2}=\left| |z-w|^2-i(t-s+2{\rm{Im}}(z\bar{w})) \right|^{1/2},
\end{equation}
where $p=[z,t]$ and $q=[w,s]$.

The {\it Cygan sphere} with center $(z_0,t_0)$  and radius $r$ has equation
$$d_{\textrm{Cyg}}\left((z,t),(z_0,t_0)\right)=\left||z-z_0|^{2}+i(t-t_0+2{\rm Im}(z\overline{z}_0))\right|=r^2.$$

The \emph{extended Cygan metric} on $\hc$ is given by the formula
\begin{equation}\label{eq:cygan-metric-extend}
  d_{\textrm{Cyg}}(p,q)=\left| |z-w|^2+|u-v|-i(t-s+2{\rm{Im}}(z\bar{w})) \right|^{1/2},
\end{equation}
where $p=(z,t,u)$ and $q=(w,s,v)$.
\end{defn}

If
$$\mathbf{p}=\left(\begin{matrix}
p_1\\ p_2\\p_3\end{matrix}\right),\quad \mathbf{q}=\left(\begin{matrix}
q_1\\ q_2\\q_3\end{matrix}\right)$$ are lifts of $p,q$ in $\mathbf{H}^2_{\mathbb C}$, then the {\it Hermitian cross product} of $p$ and $q$  is defined by

$$\mathbf{p}\boxtimes \mathbf{q}=\left(\begin{matrix}
\overline{p_1q_2-p_2q_1}\\ \overline{p_3 q_1-p_1q_3}\\ \overline{p_2 q_3-p_3q_2}\end{matrix}\right).$$
This vector is orthogonal to $p$ and $q$  with  respect to the Hermitian form  $\langle \cdot,\cdot\rangle$.
It is a Hermitian version of the Euclidean cross product.

\subsection{The isometries }The complex hyperbolic plane is a K\"{a}hler manifold of constant holomorphic sectional curvature $-1$.
We denote by $\mathbf{U}(2,1)$ the Lie group of $\langle \cdot,\cdot\rangle$ preserving complex linear
transformations and by $\mathbf{PU}(2,1)$ the group modulo scalar matrices. The group of holomorphic
isometries of ${\bf H}^2_{\mathbb C}$ is exactly $\mathbf{PU}(2,1)$. It is sometimes convenient to work with
$\mathbf{SU}(2,1)$, which is a 3-fold cover of  $\mathbf{PU}(2,1)$.

The full isometry group of ${\bf H}^2_{\mathbb C}$ is given by
$$\widehat{\mathbf{PU}(2,1)}=\langle \mathbf{PU}(2,1),\iota\rangle,$$
where $\iota$ is given on the level of homogeneous coordinates by complex conjugate
$$
\iota:\left[\begin{matrix} z_1 \\ z_2 \\ z_3 \end{matrix}\right]
\longmapsto \left[\begin{matrix} \overline{z}_1 \\
\overline{z}_2 \\ \overline{z}_3 \end{matrix}\right].
$$

Elements of $\mathbf{SU}(2,1)$ fall into three types, according to the number and types of the fixed points of the corresponding
isometry. Namely, an isometry is {\it loxodromic} (resp. {\it parabolic}) if it has exactly two fixed points (resp. one fixed point)
on $\partial {\bf H}^2_{\mathbb C}$. It is called {\it elliptic}  when it has (at least) one fixed point inside ${\bf H}^2_{\mathbb C}$.
An elliptic $A\in \mathbf{SU}(2,1)$ is called {\it regular elliptic} whenever it has three distinct eigenvalues, and {\it special elliptic} if
it has a repeated eigenvalue.

The types of isometries can be determined by the traces of their matrix realizations, see Theorem 6.2.4 of Goldman \cite{Go}. Assume $A\in{\rm{\mathbf{SU}}}(2,1)$  is non-trivial and has real trace. Then $A$ is elliptic if $-1\leq{\rm{tr}(A)}<3$.
Moreover, $A$ is unipotent if ${\rm{tr}(A)}=3$. In particular, if ${\rm{tr}(A)}=-1,0,1$, $A$ is elliptic of order 2, 3, 4 respectively.

Unipotent elements of $\mathbf{SU}(2,1)$ are conjugate in $\mathbf{SU}(2,1)$
to one fixing $q_{\infty}$ given by:
$$T_{[z,t]}=\left(\begin{matrix}
	1 & -\overline{z}& \frac{-|z|^{2}+it}{2} \\ 0 & 1 & z \\ 0 & 0 & 1 \end{matrix}\right).$$
Note that applying  $T_{[z,t]}$ to $[w,s]$ amounts to doing the Heisenberg
left multiplication by $[z,t]$. For that reason $T_{[z,t]}$ is called a 
{\it Heisenberg translation}.  A Heisenberg translation by $[0,t]$ is called a {\it vertical translation} by $t$.

The full stabilizer of $q_{\infty}$ is generated by the above unipotent group, together with the isometries of the forms
\begin{equation}
\left(\begin{matrix}
1 & 0& 0 \\ 0 & e^{i\theta} & 0 \\ 0 & 0 & 1 \end{matrix}\right) \quad  and   \quad
\left(\begin{matrix}
\lambda & 0 & 0 \\ 0 & 1 & 0 \\ 0 & 0 & 1/\lambda \end{matrix}\right),
\end{equation}
where $\theta,\lambda\in \mathbb{R}$ and $\lambda \neq 0$.  The first acts on $\partial {\bf H}^2_{\mathbb C}\backslash\{q_{\infty}\}=\mathbb{C}\times \mathbb{R}$ as a rotation
with vertical axis:
$$(z,t)\mapsto (e^{i\theta}z,t),$$  whereas the second one acts as $$(z,t)\mapsto (\lambda z,\lambda^2 t).$$
Note that the parabolic isometries fixing $q_{\infty}$ are Cygan isometries, see \cite{Go}.

\subsection{Totally geodesic submanifolds and complex reflections}
There are two kinds of totally geodesic submanifolds of real dimension 2 in ${\bf H}^2_{\mathbb C}$: {\it complex lines} in ${\bf H}^2_{\mathbb C}$ are
complex geodesics (represented by ${\bf H}^1_{\mathbb C}\subset {\bf H}^2_{\mathbb C}$) and {\it Lagrangian planes} in ${\bf H}^2_{\mathbb C}$ are totally
real geodesic 2-planes (represented by ${\bf H}^2_{\mathbb R}\subset {\bf H}^2_{\mathbb C}$). Since the Riemannian sectional curvature of the  complex hyperbolic plane is nonconstant, there are no totally geodesic hypersurfaces.

The ideal boundary of a  Lagrangian plane  on $\partial\hc$ is called a $\mathbb{R}$-circle.
The ideal boundary of a complex line on $\partial\hc$ is called a $\mathbb{C}$-circle.
Let $L$ be a complex line and $\partial L$ the associated $\mathbb{C}$-circle. A {\it polar vector} of $L$ (or $\partial L$) is the unique vector (up to scalar multiplication) perpendicular to this complex line with
respect to the Hermitian form. A polar vector  belongs to  $V_{+}$ and each vector in $V_{+}$ corresponds to a complex line or a $\mathbb{C}$-circle.

In the Heisenberg model, $\mathbb{C}$-circles are either vertical lines, or ellipses whose projection on the $z$-plane are circles.
A finite $\mathbb{C}$-circle is determined by a center and a radius. They may also be described using polar vectors. A finite $\mathbb{C}$-circle
with center $(x+yi,t) \in \mathbb{C} \times \mathbb{R}$ and radius $r$ has polar vector
$$\left(\begin{matrix} \frac{r^2-x^2-y^2+it}{2}\\ x+y i \\ 1 \end{matrix}\right).$$

There is a special class of elliptic elements of order two in $\mathbf{PU}(2,1)$.
\begin{defn}
	The \emph{complex involution} on complex line $C$ with polar vector ${\bf{n}}$ is given by the following formula:
	\begin{equation}\label{eq:involution}
	I_{C}({\bf{z}})=-{\bf{z}}+\frac{2\langle {\bf{z}}, {\bf{n}} \rangle}{\langle {\bf{n}}, {\bf{n}} \rangle} {\bf{n}}.
	\end{equation}
	Then  $I_{C}$ is a holomorphic isometry fixing the complex line $C$.
\end{defn}

\subsection{Bisectors and spinal coordinates}
In order to analyze  2-faces of a Ford polyhedron, we must study the intersections of isometric spheres. Isometric spheres are special examples of bisectors. In this subsection, we will describe a convenient set of
coordinates for bisector intersections, deduced from  the slice decomposition of a bisector.

\begin{defn} The  \emph{bisector} $\mathcal{B}(p_0,p_1)$ between two points
 $p_0$  and $p_1$ in ${\bf H}^2_{\mathbb C}$ is defined by
 $$\mathcal{B}(p_0,p_1)=\{x\in{\bf H}^2_{\mathbb C}: d(x,p_0)=d(x,p_1)\}.$$	
\end{defn}

From the distance formula given in Equation (\ref{eq:bergman-metric}), we have
\begin{prop}Let  $\bf{p}_0$  and  $\bf{p}_1$ be the lifts of  $p_0$  and $p_1$ to
$\mathbb{C}^{2,1}$ with the same norm. Then the bisector $\mathcal{B}(p_0,p_1)$ is simply the projectivization of the set of negative vectors $\bf{x}$ in $\mathbb{C}^{2,1}$ that satisfy
$$|\langle \bf{x},\bf{p}_0\rangle|=|\langle \bf{x},\bf{p}_1\rangle|.$$	
\end{prop}

The  {\it spinal sphere} of the bisector $\mathcal{B}(p_0,p_1)$ is the intersection of $\partial {\bf H}^2_{\mathbb C}$ with the closure of $\mathcal{B}(p_0,p_1)$ in $\overline{{\bf H}^2_{\mathbb C}}= {\bf H}^2_{\mathbb C}\cup \partial { {\bf H}^2_{\mathbb C}}$. The bisector $\mathcal{B}(p_0,p_1)$ is a topological 3-ball, and its spinal sphere is a 2-sphere.
The  {\it complex spine} of $\mathcal{B}(p_0,p_1)$ is the complex line through the two points $p_0$ and $p_1$. The {\it real spine} of $\mathcal{B}(p_0,p_1)$
is the intersection of the complex spine with the bisector itself, which is a (real) geodesic; it is the locus of points inside the complex spine which are equidistant from $p_0$ and $p_1$. Bisectors are not totally geodesic, but they can be foliated by complex lines. Mostow \cite{Mostow:1980}  showed that a bisector is the preimage of the real
spine under the orthogonal projection onto the complex spine. The fibres of this projection are complex lines called the {\it complex slices} of the bisector. Goldman \cite{Go} showed that a bisector is the union of all  Lagrangian planes containing the real spine. Such Lagrangian planes are called the {\it real slices} or {\it meridians} of the bisector. The "foliation" of bisectors by real planes is actually a singular foliation, since all leaves intersect along the real spine.

From the detailed analysis in \cite{Go}, we know that the intersection of two bisectors is usually not totally geodesic and can be somewhat complicated. In  this paper, we shall only consider the intersection of
coequidistant bisectors, i.e. bisectors equidistant from a common point.   When $p,q$ and $r$ are not in a common complex line, that is,  their lifts are linearly independent in $\mathbb {C}^{2,1}$, then the locus $\mathcal{B}(p,q,r)$ of points in $ {\bf H}^2_{\mathbb C}$
equidistant to  $p,q$ and $r$ is a smooth disk that is not totally geodesic, and is often called a \emph{Giraud disk}.  The following property is crucial when studying fundamental domain.

\begin{prop}[Giraud] \label{prop:Giraud}
	If $p,q$ and $r$ are not in a common complex line, then the Giraud disk $\mathcal{B}(p,q,r)$ is contained in precisely three bisectors, namely $\mathcal{B}(p,q),  \mathcal{B}(q,r)$ and  $\mathcal{B}(p,r)$.
\end{prop}

Note that checking whether an isometry maps  a Giraud disk to another is equivalent to checking that the corresponding triples of  points are mapped to each other by an element of $\mathbf{PU}(2,1)$ .

In order to study Giraud  disks, we will use {\it spinal coordinates}. The complex slices of $\mathcal{B}(p,q)$ are given explicitly by choosing a lift $\mathbf{p}$ (resp. $\mathbf{q}$) of $p$ (resp. $q$).
When $p,q\in  {\bf H}^2_{\mathbb C}$, we simply choose lifts such that $\langle \mathbf{p},\mathbf{p}\rangle= \langle \mathbf{q},\mathbf{q}\rangle$.
In this paper, we will mainly use these parametrization when  $p,q\in  \partial{\bf H}^2_{\mathbb C}$. In that case, the condition
$\langle \mathbf{p},\mathbf{p}\rangle= \langle \mathbf{q},\mathbf{q}\rangle$  is vacuous, since all lifts are null vectors; we then choose some fixed lift $\mathbf{p}$
for  the center of the Ford domain, and we take $ \mathbf{q}=G( \mathbf{p})$ for some $G\in \mathbf{U}(2,1)$. For  a different matrix $G'=SG$, with $S$
is a scalar matrix, note  that the diagonal element of  $S$ is a unit complex number, so $\mathbf{q}$ is well defined up to a unit complex number.
The complex slices of $\mathcal{B}(p,q)$ are obtained as the set of negative lines $(\overline{z}\mathbf{p}-\mathbf{q})^{\bot}$ in ${\bf H}^2_{\mathbb C}$ for some arc of values of $z\in S^1$,  which is determined by requiring that $\langle \overline{z}\mathbf{p}-\mathbf{q},\overline{z}\mathbf{p}-\mathbf{q}\rangle>0$.

Since a point of the bisector is on precisely one complex slice, we can parameterize the {\it Giraud torus}  $\hat{\mathcal{B}}(p,q,r)$  in ${\bf P}^2_{\mathbb C}$   by $(z_1,z_2)=(e^{it_1},e^{it_2})\in S^1\times S^1 $ via
$$V(z_1,z_2)=(\overline{z}_1\mathbf{p}-\mathbf{q})\boxtimes (\overline{z}_2\mathbf{p}-\mathbf{r})=\mathbf{q}\boxtimes \mathbf{r}+z_1 \mathbf{r}\boxtimes \mathbf{p}+z_2 \mathbf{p}\boxtimes \mathbf{q}.$$

The Giraud disk $\mathcal{B}(p,q,r)$ corresponds to the $(z_1,z_2)\in S^1\times S^1 $  with  $$\langle V(z_1,z_2),V(z_1,z_2)\rangle<0.$$ 

The boundary at infinity $\partial\mathcal{B}(p,q,r)$ is a circle, given in spinal coordinates by the equation
$$\langle V(z_1,z_2),V(z_1,z_2)\rangle=0.$$

Note that the choices of two lifts of $q$ and $r$ affect the spinal coordinates by rotation on each of the $S^1$-factors.

A defining equation for the trace of another bisector $\mathcal{B}(u,v)$ on the Giraud disk $\mathcal{B}(p,q,r)$ can be written in the form
$$| \langle V(z_1,z_2),u\rangle|=| \langle V(z_1,z_2),v\rangle|,$$  provided that $u$ and $v$ are suitably chosen lifts.
The expressions $\langle V(z_1,z_2),u\rangle$ and $\langle V(z_1,z_2),v\rangle$ are affine in $z_1$ and $z_2$.

This triple bisector intersection can be parameterized  fairly explicitly, because one can solve the equation
$$|\langle V(z_1,z_2),u\rangle|^2=|\langle V(z_1,z_2),v\rangle|^2 $$ for one of the variables $z_1$ or $z_2$ simply by solving a quadratic equation.
A detailed explanation of how this works can be found in \cite{Deraux:2016,DerauxF:2015,dpp}.

\subsection{Isometric spheres and Ford polyhedron}

Suppose that $g=(g_{ij})^3_{i,j=1}\in\mathbf{PU}(2,1)$ does not fix $q_{\infty}$. This implies that $g_{31}\neq 0$, see Lemma 4.1 of \cite{Par}.
\begin{defn}
	The \emph{isometric sphere} of $g$, denoted by $ I(g)$, is the set
	\begin{equation}\label{eq:isom-sphere}
	I(g)=\{ p \in {\hc \cup \partial\hc} : |\langle {\bf{p}}, {\bf{q}}_{\infty} \rangle | = |\langle {\bf{p}}, g^{-1}({\bf{q}}_{\infty}) \rangle| \}.
	\end{equation}
\end{defn}
The isometric sphere $I(g)$ is the Cygan sphere with center
$$g^{-1}({\bf{q}}_{\infty})=\left[\frac{\overline{g_{32}}}{\overline{g_{31}}},2{\rm{Im}}(\frac{\overline{g_{33}}}{\overline{g_{31}}})\right]$$
and radius $r_g=\sqrt{\frac{2}{|g_{31}|}}$.

The \emph{interior} of $I(g)$ is the set
\begin{equation}\label{eq:exterior}
\{ p \in {\hc \cup \partial\hc} : |\langle {\bf{p}}, {\bf{q}}_{\infty} \rangle | > |\langle {\bf{p}}, g^{-1}({\bf{q}}_{\infty}) \rangle| \}.
\end{equation}
The \emph{exterior } of $I(g)$ is the set
\begin{equation}\label{eq:interior}
\{ p \in {\hc \cup \partial\hc} : |\langle {\bf{p}}, {\bf{q}}_{\infty} \rangle | < |\langle {\bf{p}}, g^{-1}({\bf{q}}_{\infty}) \rangle| \}.
\end{equation}

We summarize without proofs the relevant properties on isometric sphere in the following proposition.
\begin{prop}[\cite{Go}, Section 5.4.5] Let $g$ and $h$ be elements of $\mathbf{PU}(2,1)$ which do not fix $q_{\infty}$, such that $g^{-1}(q_{\infty})\neq h^{-1}(q_{\infty})$ and let $f \in\mathbf{PU}(2,1)$ be an unipotent transformation  fixing $q_{\infty}$.  Then the followings hold:
	\begin{itemize}
		\item  $g$ maps $I(g)$ to $I(g^{-1})$, and the exterior of $I(g)$ to the interior of $I(g^{-1})$.
		\item  $I(gf)=f^{-1}I(g)$, $I(fg)=I(g)$.
		\item $g(I(g)\cap I(h))=I(g^{-1})\cup I(hg^{-1})$, $h(I(g)\cap I(h))=I(gh^{-1})\cup I(h^{-1})$.
	\end{itemize}
\end{prop}

Since isometric spheres are Cygan spheres, we now recall some facts about Cygan spheres.
Let $S_{[0,0]}(r)$ be the Cygan sphere with center $[0,0]$ and radius $r>0$. Then
\begin{equation}\label{eq:cygan-sphere}
S_{[0,0]}(r)=\left\{ (z,t,u)\in\hc \cup \partial\hc: (|z|^2+u)^2+t^2=r^4 \right\}.
\end{equation}

We are interested in the intersection of Cygan spheres. Cygan spheres are examples of bisectors. Theorem 9.2.6 of \cite{Go} tells us  that two bisectors which  are respectively coequidistant or covertical(their complex spines respectively intersect or are asymptotic) have connected intersection.
The complex spine of the Cygan sphere $I(g)$ is the complex span of $g^{-1}({\bf{q}}_{\infty})$ and ${\bf{q}}_{\infty}$. So two Cygan spheres are covertical.
\begin{prop}[Goldman \cite{Go}, Parker and Will \cite{ParkerWill:2016}]
	The intersection of two Cygan spheres is connected.
\end{prop}

\begin{defn}The \emph{Ford domain} $D_{\Gamma}$ for a discrete group $\Gamma < \mathbf{PU}(2,1)$ centered at $q_{\infty}$ is
	the intersection of the (closures of the) exteriors of all isometric spheres of elements in $\Gamma$ not fixing $q_{\infty}$. That is,
	$$D_{\Gamma}=\{p\in {\bf H}^2_{\mathbb C} \cup \partial{\bf H}^2_{\mathbb C}: |\langle \mathbf{p},q_{\infty}\rangle|\leq|\langle \mathbf{p},G^{-1}(q_{\infty})\rangle|
	\ \forall G\in \Gamma \ \mbox{with} \ G(q_{\infty})\neq q_{\infty} \}.$$
\end{defn}

The boundary  at infinity of the Ford domain is made out of  pieces of isometric spheres. When $q_{\infty}$ is either in the domain  of discontinuity or is a parabolic fixed point, the Ford  domain  is preserved by $\Gamma_{\infty}$, the stabilizer of
$q_{\infty}$ in $\Gamma$.  In this case, $D_{\Gamma}$ is only a fundamental domain modulo the action of $\Gamma_{\infty}$.  In other words, the fundamental domain for
$\Gamma$ is the intersection of the Ford domain with a fundamental domain for $\Gamma_{\infty}$. Facets of codimension one, two, three and four in $D_{\Gamma}$ will be called {\it sides}, {\it ridges}, {\it edges} and {\it vertices}, respectively. Moreover, a  {\it bounded ridge} is a ridge which does not intersect  $\partial {\bf H}^2_{\mathbb C}$, and if the intersection of a ridge $r$ and $\partial {\bf H}^2_{\mathbb C}$ is non-empty, then
$r$ is an {\it infinite ridge}.

It is usually very hard to determine $D_{\Gamma}$ because one should check infinitely many inequalities.
Therefore a general method will be to guess the Ford polyhedron and check  it using the Poincar\'e polyhedron theorem.
The basic idea is that the sides of $D_{\Gamma}$ should be paired by isometries, and the images of $D_{\Gamma}$
under these so-called side-pairing maps should give a local tiling of ${\bf H}^2_{\mathbb C}$.  If they do (and if
the quotient of $D_{\Gamma}$ by the identification given by the side-pairing maps is complete), then the Poincar\'{e} polyhedron
theorem implies that the images of $D_{\Gamma}$ actually give a global tiling of ${\bf H}^2_{\mathbb C}$.

Once a fundamental domain is obtained, one gets an explicit presentation of $\Gamma$ in terms of the generators given by the side-pairing maps together
with a generating set for the stabilizer $\Gamma_{\infty}$, where the relations corresponding to so-called ridge cycles, which correspond to the local tilings bear each codimension-two face. For more on the Poincar\'e polyhedron theorem, see \cite{dpp, ParkerWill:2016}.

\section{The representations of the triangle groups $\Delta_{3,\infty,\infty;\infty}$   and $\Delta_{3,4,\infty; \infty}$}\label{sec-gens}

We will explicitly give  matrix representations in $\mathbf{SU}(2,1)$  of $\Delta_{3,\infty,\infty;\infty}$   and $\Delta_{3,4,\infty; \infty}$ in this section.

\subsection{The representation of the triangle group $\Delta_{3,\infty,\infty;\infty}$}\label{subsec-3pp}

Suppose that complex reflections $I_1$ and $I_2$
in $\mathbf{SU}(2,1)$ so that $I_1I_2$ is an unipotent  element fixing $q_{\infty}$. Conjugating by a Heisenberg rotation and a Heisenberg dilation if necessary, we may  suppose that the vertical $\mathbb{C}$-circles $\partial C_{i}$ fixed by $I_{i}$ for $i=1,2$ are
$$\partial C_{1}=\{-1-\frac{i}{\sqrt{3}}\}\times \mathbb{R}, \partial C_{2}=\{-\frac{i}{\sqrt{3}}\}\times \mathbb{R}.$$
The polar vectors of the $\mathbb{C}$-circles $\partial C_{i}$ for $i=1,2$ are
\begin{eqnarray*}
	{\bf n}_1 = \left[\begin{matrix} 1-\frac{i}{\sqrt{3}} \\ 1 \\ 0 \end{matrix}\right],\quad
	{\bf n}_2 = \left[\begin{matrix} -\frac{i}{\sqrt{3}} \\ 1 \\ 0 \end{matrix}\right].
\end{eqnarray*}
The choose of the  polar vector is not in its simplest form, but that it is more relevant for our purpose. The matrices of $I_1$ and $I_2$ are
\begin{equation}\label{eq-I1-I2}
	I_1=\left[\begin{matrix}
		-1 & 2-\frac{2\sqrt{3}i}{3}& \frac{8}{3} \\ 0 & 1 &  2+\frac{2\sqrt{3}i}{3} \\ 0 & 0 & -1 \end{matrix}\right], \quad
	I_2=\left[\begin{matrix}
		-1 & -\frac{2\sqrt{3}i}{3} & \frac{2}{3} \\ 0 & 1 & \frac{2\sqrt{3}i}{3} \\ 0 & 0 & -1 \end{matrix}\right].
\end{equation}

We want to find $I_3$ so that $I_1I_3$ and $I_1I_3I_2I_3$ are unipotent   and   $I_2I_3$ is an elliptic element of order 3. After some calculation,
 one can find that  the matrix of $I_3$ has the following form
$$I_3=\left[\begin{matrix}
0& 0& \frac{2}{3} \\ 0 &-1 &0 \\ \frac{3}{2} & 0 & 0 \end{matrix}\right].$$

\subsection{The representation of the triangle group $\Delta_{3,4,\infty;\infty}$}\label{subsec-34p}

A matrix  representation of the  complex hyperbolic triangle group $\Delta_{3,4,\infty;\infty}$ is given in \cite{Thompson:2010}:

Take
\begin{equation}
	I_1=\left[\begin{matrix}
		-1& 1-\sqrt{7}i & 4 \\ 0 & 1 &1+\sqrt{7}i \\0  & 0 & -1 \end{matrix}\right] \quad  and   \quad
	I_2=\left[\begin{matrix}
		0& 0& 1 \\ 0 &-1 &0 \\1 & 0 & 0 \end{matrix}\right],
\end{equation}
and take $$I_3=\left[\begin{matrix}
-1& 0&0 \\ \frac{1-\sqrt{7}i}{2} &1 &0 \\1 & \frac{1+\sqrt{7}i}{2} &-1 \end{matrix}\right].$$

Then it is easy to see $I_2I_3$ is elliptic of order 3,  $I_3I_1$  is elliptic of order 4, both  $I_1I_2$ and  $I_1I_3I_2I_3$ are unipotent.

\section{The Ford domain for the even subgroup of $\Delta_{3,\infty,\infty;\infty}$}\label{sec-Ford-3pp}

We will study the local combinatorial structure of  Ford domain for the even subgroup of $\Delta_{3,\infty,\infty;\infty}$ in this section, the main result is Theorem \ref{thm:fundamental-domain3pp}.

Let $\Gamma=\langle I_1, I_2, I_3\rangle$ be the  group $\Delta_{3,\infty,\infty;\infty}$ in Subsection \ref{subsec-3pp}.
Let $A=I_1I_2$, $B=I_2I_3$. So $\Gamma_1=\langle A,B\rangle$ is  the even  subgroup of $\langle I_1, I_2, I_3\rangle$ with index two. By directly calculation, we have
\begin{equation*}\label{eq-I1-I2}
	A=\left[\begin{matrix}
		1 & 2& -2+\frac{4i}{\sqrt{3}} \\ 0 & 1 & -2 \\ 0 & 0 & 1 \end{matrix}\right], \quad
	B=\left[\begin{matrix}
		1 &\frac{2i}{\sqrt{3}}& -\frac{2}{3} \\ \sqrt{3}i & -1&0 \\ \frac{-3}{2} &0 &0 \end{matrix}\right].
\end{equation*}
Note  that $A$ is parabolic and $B$ is elliptic of order 3. Moreover,  $\Gamma_1$ is conjugate to a subgroup of the Einsenstein-Picard  modular group $\mathbf{PU}(2,1;\mathbb{Z}[\omega])$, proving that $\Gamma_1$ is discrete.

We will construct a polyhedron $D_{\Gamma_1}$ whose sides are contained in the isometric spheres of some well-chosen group elements of $\Gamma_1$.  The details of these isometric spheres will be  given in this section.

Note that the elements $BA^{-1}$, $BA$, $AB$, $B^{-1}A$ are unipotent parabolic. For future reference, we provide the lifts of their fixed points, as vectors in $\mathbb{C}^{3}$:

\begin{eqnarray*}
\mathbf{ p}_{B^{-1}A} &=& \left[
                               \begin{array}{c}
                                 -\frac{2}{3}\\
                                1 -\frac{\sqrt{3}i}{3} \\
                                 1 \\
                               \end{array}
                             \right],\\
 \mathbf{ p}_{BA^{-1}} &=& \left[
                               \begin{array}{c}
                                 -\frac{2}{3}+ \frac{2\sqrt{3}}{3} i\\
                                 -1 -\frac{\sqrt{3}}{3} i \\
                                 1 \\
                               \end{array}
                             \right],\\
 \mathbf{ p}_{BA} &=& \left[
                               \begin{array}{c}
                                 -\frac{2}{3}-\frac{2\sqrt{3}}{3} i\\
                                 1 -\frac{\sqrt{3}}{3} i \\
                                 1 \\
                               \end{array}
                             \right],\\
   \mathbf{ p}_{AB}&=& \left[
                               \begin{array}{c}
                                  -\frac{2}{3}\\
                                -1 -\frac{\sqrt{3}i}{3} \\
                                 1 \\
                               \end{array}
                             \right].
\end{eqnarray*}

Note that $A(p_{B^{-1}A})=p_{BA^{-1}}$, $A(p_{BA})=p_{AB}$. Let  $I(B)$ and $I(B^{-1})$ be the  isometric spheres for $B$ and its inverse. From the matrix for $B$ given above, we see that $I(B)$ and $I(B^{-1})$ have the same radius $2/\sqrt{3}$, and centers $[0,0]$ and $[-2i/\sqrt{3},0]$ in Heisenberg coordinates  respectively.

We give some notation for the isometric spheres which are the images of $I(B)$ and $I(B^{-1})$ by powers of $A$.
\begin{defn} For $k\in \mathbb{Z}$, let $I_{k}^{+}$ be the isometric sphere $I(A^{k}BA^{-k})=A^{k}I(B)$ and let
 $I_{k}^{-}$ be the isometric sphere $I(A^{k}B^{-1}A^{-k})=A^{k}I(B^{-1})$.
\end{defn}

The action of $A$ on the Heisenberg group is given by
\begin{equation}\label{eq-A}
(z,t)\rightarrow (z-2, t+4\rm{Im}(z)+8/\sqrt{3}).
\end{equation}

In particular, we have
\begin{prop}\label{prop:A-action}
For any value of $t_0\in \mathbb{R}$, the action $A$ preserves every $\mathbb{R}$-circle of the form $[x-2i/\sqrt{3},t_0]$ with $x\in \mathbb{R}$.
\end{prop}

As $A$ is unipotent and fixes $q_{\infty}$, it is a Cygan isometry, and thus preserves the radii of isometric spheres. Moreover, it follows from
Equation (\ref{eq-A}) that $A^{k}$ acts on $\partial{\bf H}^2_{\mathbb C}$ by left Heisenberg multiplication of $[-2k,8k/\sqrt{3}]$. Then we have

\begin{prop}\label{prop:center-radius3pp}
 For any integer $k\in \mathbb{Z}$, the isometric spheres $I_{k}^{+}$ and $I_{k}^{-}$ have the same radius $2/\sqrt{3}$ and are centered at the point with Heisenberg coordinates $[-2k,8k/\sqrt{3}]$ and $[-2k-2i/\sqrt{3},0]$, respectively.
\end{prop}

\begin{figure}
\begin{center}
\begin{tikzpicture}[style=thick,scale=1.5]
\draw (-3/2,0.57735/2) circle (1.1547/2) ;
\draw  (-1/2,0.57735/2) circle  (1.1547/2) ;
\draw (1/2,0.57735/2) circle  (1.1547/2) ;
\draw (3/2,0.57735/2) circle  (1.1547/2) ;
\draw (5/2,0.57735/2) circle  (1.1547/2) ;
\draw (7/2,0.57735/2) circle  (1.1547/2) ;
\draw (9/2,0.57735/2) circle  (1.1547/2) ;

\draw (-3/2,-0.57735/2) circle (1.1547/2) ;
\draw  (-1/2,-0.57735/2) circle  (1.1547/2) ;
\draw (1/2,-0.57735/2) circle  (1.1547/2) ;
\draw (3/2,-0.57735/2) circle  (1.1547/2) ;
\draw (5/2,-0.57735/2) circle  (1.1547/2) ;
\draw (7/2,-0.57735/2) circle  (1.1547/2) ;
\draw (9/2,-0.57735/2) circle  (1.1547/2) ;

\draw (3/2,0.57735/2+1.1547/2) [above ] node{$I_0^{+}$};
\draw (3/2,-0.57735/2-1.1547/2) [below ] node{$I_0^{-}$};

\draw (1/2,0.57735/2+1.1547/2) [above ] node{$I_{-1}^{+}$};
\draw (1/2,-0.57735/2-1.1547/2) [below ] node{$I_{-1}^{-}$};

\draw (-1/2,0.57735/2+1.1547/2) [above ] node{$I_{-2}^{+}$};
\draw (-1/2,-0.57735/2-1.1547/2) [below ] node{$I_{-2}^{-}$};

\draw (5/2,0.57735/2+1.1547/2) [above ] node{$I_{1}^{+}$};
\draw (5/2,-0.57735/2-1.1547/2) [below ] node{$I_{1}^{-}$};

\draw (7/2,0.57735/2+1.1547/2) [above ] node{$I_{2}^{+}$};
\draw (7/2,-0.57735/2-1.1547/2) [below ] node{$I_{2}^{-}$};

\end{tikzpicture}
\end{center}
\caption{The vertical projections of some isometric spheres.}
\label{figure:proj-sphere}
\end{figure}
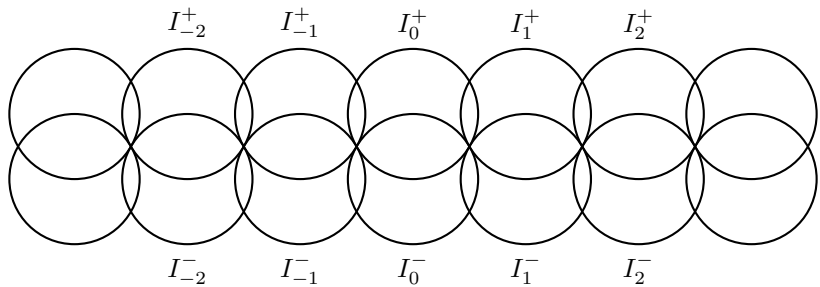

We need to give the details for  the combinatorics of the family of isometric spheres $\{I_{k}^{\pm}: k\in \mathbb{Z}\}$.

\begin{prop}\label{prop:pair-disjoint3pp}
 Any two isometric spheres in $\{I_{k}^{+}: k\in \mathbb{Z}\}$ or in $\{I_{k}^{-}: k\in \mathbb{Z}\}$ are disjoint.
\end{prop}
\begin{proof}  It is sufficient to prove that $I_{0}^{+}$ and $I_{k}^{+}$ are disjoint for any $k\in \mathbb{Z}\backslash \{0\}$.  Since all the isometric spheres have radius $2/\sqrt{3}$, if we can show that the Cygan distance between their centers is greater than  the sum of the radii of the two Cygan  spheres, then the spheres are disjoint.
From Proposition \ref{prop:center-radius3pp}, we get that $I_{0}^{+}$ is a Cygan sphere centered at  $[0,0]$ and $I_{k}^{+}$ is a Cygan sphere centered at  $[-2k,8k/\sqrt{3}]$. The Cygan distance between centers of $I_{0}^{+}$ and $I_{k}^{+}$ is
$|4k^{2}+i8k/\sqrt{3}|^{1/2},$ which is larger than $4/\sqrt{3}$ for any
 $k\in \mathbb{Z}\backslash \{0\}$. This completes the proof of this proposition.
\end{proof}

\begin{prop}\label{prop:tangent3pp}
\begin{itemize}
\item The isometric spheres $I_{0}^{+}$ and $I_{1}^{-}$ are tangent at  the parabolic fixed point $p_{AB}$ of $AB$.
\item The isometric spheres $I_{0}^{+}$ and $I_{-1}^{-}$ are tangent at  the parabolic fixed point $p_{B^{-1}A}$ of $B^{-1}A$.
\item The isometric spheres $I_{0}^{-}$ and $I_{1}^{+}$ are tangent  at the parabolic fixed point $p_{BA^{-1}}$ of $BA^{-1}$.
\item The isometric spheres $I_{0}^{-}$ and $I_{-1}^{+}$ are tangent  at the parabolic fixed point $p_{BA}$ of  $BA$.
 \end{itemize}
\end{prop}
\begin{proof}
The isometric spheres $I_{0}^{+}$ and $I_{1}^{-}$ are Cygan spheres with centers $B^{-1}(q_{\infty})$ and $AB(q_{\infty})$.  It is easy  to verify that
$$d_{\rm{Cyg}}\left( p_{AB},B^{-1}(q_{\infty})\right)=d_{\rm{Cyg}}\left( p_{AB},AB(q_{\infty})\right)=2/\sqrt{3},$$ so
$ p_{AB}$ belongs to both $I_{0}^{+}$ and $I_{1}^{-}$. Observe that the vertical projections of $I_{0}^{+}$ and $I_{1}^{-}$ are tangent discs, see Figure \ref{figure:proj-sphere}
and Remark 2.13 of \cite{ParkerWill:2016}. Since all Cygan spheres are strictly affine convex, their intersection contains at most one point. The remaining cases can be proved by using the same argument.
\end{proof}

\begin{prop}\label{prop:intersection-disk3pp}
For any $k_0\in \mathbb{Z}$, the intersection $I_{k_0}^{+}$ and $I_{k_0}^{-}$ is a Giraud disk.
\end{prop}
\begin{proof}
One can verify that the point $\left[-2/3,-i/\sqrt{3},1\right]^{t}\in\hc$ lies on $I_{0}^{+}$ and $I_{0}^{-}$ at the same time. So $I_{k_0}^{+}$ and $I_{k_0}^{-}$ are
covertical bisectors with nonempty intersection. Then  $I_{k_0}^{+}\cap I_{k_0}^{-}$ is a smooth disk.
\end{proof}

The following fact describes the intersections of  $I_{k_0}^{+}$ with the other isometric spheres.
\begin{prop}\label{prop:intersection3pp}
 The isometric sphere $I_{0}^{+}$ (resp.  $I_{0}^{-}$) is contained in the exterior of the isometric spheres $\{I_{k}^{\pm}: k\in \mathbb{Z}\}$ except for $I_{0}^{-}$(resp. $I_{0}^{+}$).
\end{prop}
\begin{proof}
From Proposition \ref{prop:pair-disjoint3pp}, we know that the isometric spheres $I_{0}^{+}$ and $I_{k}^{+}$, for $k\in \mathbb{Z}\backslash \{0\}$  are disjoint.
Using Proposition \ref{prop:center-radius3pp}, the centers of $I_{0}^{+}$ and $I_{k}^{-}$ are
$[0,0]$ and $[-2k-2/\sqrt{3}i,0]$, respectively.
The Cygan distance between their centers is
$$d_{\rm{Cyg}}([0,0],[-2k-2i/\sqrt{3},0])=|-2k-2i/\sqrt{3}|=\sqrt{4k^{2}+4/3}.$$ This number is larger than $4/\sqrt{3}$ when $|k|\geq 2$.  The case $k=\pm 1$ has been considered in Proposition \ref{prop:tangent3pp}.
\end{proof}

We will give a definition of  the infinite polyhedron $D_{\Gamma_1}$  which will be used to study the group $\Gamma_1$ generated by $A$ and $B$.
\begin{defn} The infinite polyhedron $D_{\Gamma_1}$ is the intersection of the exteriors of all the isometric spheres in $\{I_{k}^{\pm}: k\in \mathbb{Z}\}$ with centers $A^{k}\left(B^{\pm}(q_{\infty})\right)$. That is,
$$D_{\Gamma_1}=\{p\in \hc: |\langle \mathbf{p}, \mathbf{q}_{\infty}\rangle| \leq |\langle \mathbf{p}, A^{k}B^{\pm}(\mathbf{q}_{\infty})\rangle|\, \rm{for}\, \rm{ all}\, k\in \mathbb{Z}\}.$$
\end{defn}

The main tool for our study is the  Poincar\'{e} polyhedron theorem, which gives sufficient condition for $D_{\Gamma_1}$ to be a fundamental domain for the group generated by $A$ and $B$. We shall use a version of the   Poincar\'{e} polyhedron theorem for coset decompositions rather than for groups, because $D_{\Gamma_1}$ is stabilized by the cyclic subgroup generated by $A$. We refer to \cite{dpp, ParkerWill:2016} for the precise statement of this version of the Poincar\'{e} polyhedron theorem we need. We proceed to show:
\begin{thm}\label{thm:fundamental-domain3pp}
$D_{\Gamma_1}$ is a fundamental domain for the cosets of $\langle A \rangle$ in $\Gamma_1$. Moreover, the group $\Gamma_1=\langle A, B \rangle$  is discrete and  has the presentation
\begin{equation*}
\langle  A, B: B^3=id\rangle.
\end{equation*}
\end{thm}

From Propositions \ref{prop:intersection3pp} and  \ref{prop:intersection-disk3pp}, we see that $I_{k}^{\pm}\cap D_{\Gamma_1}$ has nonempty interior in $I_{k}^{\pm}$. Therefore we refer to $s_{k}^{+}=I_{k}^{+}\cap D_{\Gamma_1}$ and $s_{k}^{-}=I_{k}^{-}\cap D_{\Gamma_1}$ as sides of $D_{\Gamma_1}$.  We describe the combinatorics of the sides $s_{k}^{+}$ and $s_{k}^{-}$ in the
following proposition.
\begin{prop}\label{prop:side3pp}
The side $s_{k}^{\pm}$ is topologically a 3-ball in $\hc \cup \partial\hc$, with boundary a union of two disks: one is $s_{k}^{+}\cap s_{k}^{-}=I_{k}^{+}\cap I_{k}^{-}$ and the other one is $s_{k}^{\pm}\cap \partial\hc=I_{k}^{\pm}\cap \partial\hc$.
\end{prop}
\begin{proof}
The side $s_{k}^{+}$ is contained in  the isometric sphere $I_{k}^{+}$. The isometric sphere $I_{k}^{+}$ doesn't intersect with all the other isometric spheres except $I_{k}^{-}$.   By Proposition \ref{prop:tangent3pp}, $I_{k}^{+}$ is tangent to $I_{k+1}^{-}$ and  $I_{k-1}^{-}$ at a parabolic fixed point respectively. Also,  $I_{k}^{+}\cap I_{k}^{-}$ is a smooth disk by Proposition \ref{prop:intersection-disk3pp}. Therefore the description of  $s_{k}^{+}$ follows. One can describe the side  $s_{k}^{-}$ by using a similar argument.
\end{proof}

The side pairing maps are defined by:
$$A^kBA^{-k}:s_{k}^{+}\longrightarrow s_{k}^{-},\quad  A^kB^{-1}A^{-k}:s_{k}^{-}\longrightarrow s_{k}^{+}.$$
Furthermore, we denote by $r_k=s_{k}^{+}\cap s_{k}^{-}$  a ridge of $D_{\Gamma_1}$ by using Propositions \ref{prop:intersection3pp} and \ref{prop:intersection-disk3pp}.

\begin{prop}\label{prop:sideparing3pp}
The side pairing map  $A^kBA^{-k}$ is a homeomorphism from $s_{k}^{+}$ to $s_{k}^{-}$ and sends the ridge $r_k$ to itself.
\end{prop}
\begin{proof} We only need to prove the case where $k=0$. The ridge $r_0=s_{0}^{+}\cap s_{0}^{-}$
is defined by the triple equality
$$\langle z, \mathbf{ q}_{\infty}\rangle=\langle z, B^{-1}(\mathbf{ q}_{\infty})\rangle=\langle z, B(\mathbf{ q}_{\infty})\rangle.$$ The map $B$ of order 3 cyclically permutes  $\mathbf{ q}_{\infty}$,
$B^{-1}(\mathbf{ q}_{\infty} )$  and  $B(\mathbf{ q}_{\infty})$, so $B$ maps  $r_0$ to itself.
\end{proof}

{\bf Proof of Theorem \ref{thm:fundamental-domain3pp}} After Propositions  \ref{prop:side3pp} and  \ref{prop:sideparing3pp}, we now show

\paragraph{{\bf Local tessellation}}
 We prove the tessellation around the sides and ridges of  $D_{\Gamma_1}$.

(1). Since $A^kB^{\pm}A^{-k}$ sends the exterior of $I_k^{\pm}$ to the interior of  $I_k^{\mp}$, we see that $D_{\Gamma_1}$ and $A^kB^{\pm}A^{-k}(D_{\Gamma_1})$ have disjoint interiors and cover a neighborhood of each point in the interior of   $s_k^{\mp}$.

(2).  Observe that the ridge $r_k=s_k^{+}\cap s_k^{-}=I_k^{+}\cap I_k^{-}$ is mapped to itself by $B$.  The cycle transformation for the ridge $r_k$ is
$A^kBA^{-k}$ and the cycle relation is $(A^kBA^{-k})^3=A^kB^3A^{-k}=id$.  By the same argument as in \cite{ParkerWill:2016} or a result of Giraud, we see that $D_{\Gamma_1}\cup B(D_{\Gamma_1} )\cup B^{-1} (D_{\Gamma_1})$ will cover a
small neighborhood of $r_k$.

\paragraph{{\bf Completeness}}

We must construct a system of consistent horoballs at the parabolic fixed points.
First, we consider the side pairing maps on the parabolic fixed points in Proposition \ref{prop:tangent3pp}.
We have
\begin{eqnarray*}
  B &:&   p_{AB} \longrightarrow p_{BA},\\
   AB^{-1} A^{-1}&:&   p_{AB} \longrightarrow A(p_{AB}),\\
  B &:&   p_{B^{-1}A} \longrightarrow p_{BA^{-1}} ,\\
 A^{-1}B^{-1} A &:&   p_{B^{-1}A} \longrightarrow  A^{-1}(p_{B^{-1}A}),\\
   ABA^{-1} &:&   p_{BA^{-1}}  \longrightarrow  A(p_{BA^{-1}}),\\
   A^{-1}BA &:&   p_{BA} \longrightarrow  A^{-1}(p_{BA}).
  \end{eqnarray*}

 Up to powers of $A$,  the cycles for the parabolic fixed points are the following
  \begin{eqnarray*}
&&p_{AB} \xrightarrow{B} p_{BA}\xrightarrow{A} p_{AB},\\
&&p_{B^{-1}A} \xrightarrow{A} p_{BA^{-1}}\xrightarrow{B^{-1}} p_{B^{-1}A}.
\end{eqnarray*}
That is,  $p_{AB}$ and $p_{B^{-1}A}$ are fixed by $AB$ and $B^{-1}A$ respectively.
The elements $AB$ and $B^{-1}A$ are unipotent and  preserve all horoballs at $p_{AB}$ and $p_{B^{-1}A}$ respectively. This ends the proof of Theorem \ref{thm:fundamental-domain3pp}.

\section{The 3-manifold at infinity of the even subgroup of $\Delta_{3,\infty, \infty;\infty}$} \label{section:3ppmanifold}

In this section, based on Section  \ref{sec-Ford-3pp},  we study the manifold  at infinity of the even subgroup $\Gamma_1$ of the complex hyperbolic triangle group  $\Delta_{3,\infty,\infty;\infty}$. We restate Theorem \ref{thm:3pp} as:
\begin{thm} The 3-manifold at infinity of $\Gamma_1$ is homeomorphic to the complement of the magic link $6_3^1$ in the 3-sphere.
\end{thm}

We have proved in Section \ref{sec-Ford-3pp} that    the ideal boundary of the  sides $s_k^{\pm}$, that  is $s_k^{\pm}\cap \partial \hc$, is a sub-disk of the spinal sphere of $A^{k}B^{\pm}A^{-k}$.  Recall that the ideal boundary of the side  $s_k^{\pm}$  is the part of  $\partial I_k^{\pm}=  I_k^{\pm} \cap \partial \hc$  which is outside (the ideal boundary of)
all other isometric spheres. In this section, when we speak of sides and ridges we implicitly mean their intersections with $\partial \hc$.

Take a singular surface $$\Sigma =  \bigcup_{k \in \mathbb{Z}} (s_k^{\pm}\cap \partial \hc)$$  in the Heisenberg group.
The ideal boundary of  $D_{\Gamma_1}$, that is $\partial_{\infty} D_{\Gamma_1}=D_{\Gamma_1} \cap \partial \hc$, is a region in the Heisenberg group bounded by the surface $\Sigma$.
We will see that the ideal boundary  of $D_{\Gamma_1}$ is an infinite genus handlebody, invariant under the action of $\langle A \rangle$. Note that the  ideal boundaries  of the fundamental domains constructed in \cite{mx, ParkerWill:2016} are infinite cylinders. Due to the complicate  topology of $\partial_{\infty} D_{\Gamma_1}$,  we need more explicit disks to cut  it into a 3-ball in  our proof  Theorem \ref{thm:3pp}.

\subsection{The Ford domain of the complex hyperbolic triangle group  $\Delta_{3,\infty,\infty;\infty}$}

We have showed that the Ford domain of  $\Delta_{3,\infty,\infty;\infty}$ is bounded by the isometric spheres of $A$-translates of $B$ and $B^{-1}$.
Figure \ref{figure:3pp0} is a  realistic  view   of the ideal boundary $\Sigma$ of the Ford domain for $\Delta_{3,\infty,\infty;\infty}$ with center $q_{\infty}$. This figure is just for motivation, and we have showed in Section \ref{sec-Ford-3pp} the local picture of this figure is correct.

\begin{figure}
\begin{center}
\begin{tikzpicture}
\node at (0,0) {\includegraphics[width=8cm,height=8cm]{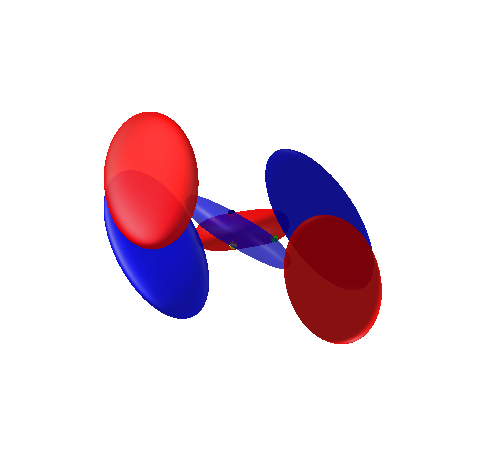}};
\coordinate [label=left:$I_{-1}^{+}$] (S) at (-1.9,2);
\coordinate [label=left:$I_{-1}^{-}$] (S) at (-1.2,-1.5);
\coordinate [label=above:$I_{0}^{-}$] (S) at (0.22,-1.1);
\coordinate [label=above:$I_{0}^{+}$] (S) at (0.22,0.36);
\coordinate [label=above:$I_{1}^{-}$] (S) at (1.5,1);
\coordinate [label=above:$I_{1}^{+}$] (S) at (1.5,-2.5);
\draw[->] (0.6,-0.25)--(0.6,-1.5);
\coordinate [label=below:$B(u)$] (S) at (0.6,-1.5);
\draw[->] (-0.15,-0.35)--(-0.3,-0.6);
\coordinate [label=below:$u$] (S) at (-0.3,-0.6);
\draw[->] (-0.15,0.3)--(-0.25,1);
\coordinate [label=above:$B^2(u)$] (S) at (-0.25,1);

\end{tikzpicture}
\end{center}
  \caption{A  realistic view  of part of the surface $\Sigma$.
		The ideal boundary of the Ford domain of $\Gamma_1$ is the region outside the singular surface $\Sigma$. There are also three  dots  on the intersection circle $I_{0}^{+}\cap I_{0}^{-}$ colored by  yellow, green and black,  which correspond to the points $u$, $B(u)$  and $B^2(u)$. }
  \label{figure:3pp0}
\end{figure}
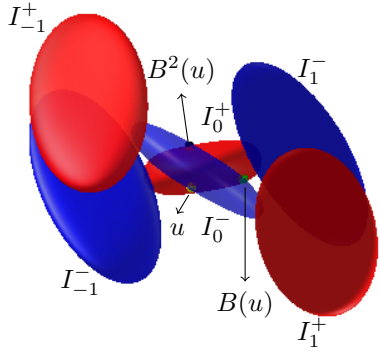

We hope the figure  is explicit enough such that the reader can see the "holes"  in  Figure \ref{figure:3pp0}. One of them is the white "hole" enclosed by the spinal spheres $I_{0}^{-}$, $I_{0}^{+}$, $I_{1}^{-}$ and $I_{1}^{+}$.
 Another one is  the white "hole" enclosed by the spinal spheres $I_{0}^{-}$, $I_{0}^{+}$, $I_{-1}^{-}$ and $I_{-1}^{+}$.   The ideal boundary of a  Ford domain is just the intersection of  the exteriors  of all the spinal spheres, where the "holes"  make the topology of the ideal boundary of this Ford domain  different from  those studied in \cite{Deraux:2016, jwx, mx, ParkerWill:2016}. Since the  realistic  view in Figure \ref{figure:3pp0} is a little  difficult to understand, we show a combinatorial model of it in Figure	\ref{figure:3ppabstract}. Our  computations later will show  the combinatorial model  coincides with Figure \ref{figure:3pp0}.

\begin{figure}
	\begin{tikzpicture}[style=thick,scale=0.5]
	\draw (1,0) arc (-60:240:2);
	\draw (1,0) arc (60:-240:2);
	\draw[dashed] (1,0) arc (0:180:1 and 0.3);
	\draw(-1,0) arc (180:360:1 and 0.3);

	\draw (5,0) arc (-60:240:2);
	\draw (5,0) arc (60:-240:2);
	\draw[dashed] (5,0) arc (0:180:1 and 0.3);
	\draw(3,0) arc (180:360:1 and 0.3);
	
	\draw (9,0) arc (-60:240:2);
	\draw (9,0) arc (60:-240:2);
	\draw[dashed] (9,0) arc (0:180:1 and 0.3);
	\draw(7,0) arc (180:360:1 and 0.3);
	
	\draw (-3,0) arc (-60:240:2);
	\draw (-3,0) arc (60:-240:2);
	\draw[dashed] (-3,0) arc (0:180:1 and 0.3);
	\draw(-5,0) arc (180:360:1 and 0.3);
	
	\draw (-7,0) arc (-60:240:2);
	\draw (-7,0) arc (60:-240:2);
	\draw[dashed] (-7,0) arc (0:180:1 and 0.3);
	\draw(-9,0) arc (180:360:1 and 0.3);
	
	\node [above] at (0,2.132){$I_{0}^{-}$};
	\node [above] at (4,2.132){$I_{1}^{+}$};
	\node [above] at (-4,2.132){$I_{-1}^{+}$};
	\node [above] at (8,2.132){$I_{2}^{-}$};
	\node [above] at (-8,2.132){$I_{-2}^{-}$};
	\node [below] at (0,-2.132){$I_{0}^{+}$};
	\node [below] at (4,-2.132){$I_{1}^{-}$};
	\node [below] at (-4,-2.132){$I_{-1}^{-}$};
	\node [above] at (8,-3.432){$I_{2}^{+}$};
	\node [above] at (-8,-3.432){$I_{-2}^{+}$};
	
	\draw[fill] (-2,1.732) circle [radius=0.1];\node [left] at (-2,1.732){$p_{BA}$};
	\draw[fill] (-2,-1.732) circle [radius=0.1];\node [left] at (-2,-1.732){$p_{B^{-1}A}$};
	\draw[fill] (2,1.732) circle [radius=0.1];\node [left] at (2,1.732){$p_{BA^{-1}}$};
	\draw[fill] (2,-1.732) circle [radius=0.1];\node [left] at (2,-1.732){$p_{AB}$};
	\end{tikzpicture}
	\caption{A  combinatorial  picture of  the ideal boundary of the Ford domain of $\Delta_{3,\infty,\infty;\infty}$, which is  the region  outside all the spheres in  Figure \ref{figure:3ppabstract}. Where $E_1$ is a   topological plane  tangent to these spheres at  $p_{BA}$ and $p_{B^{-1}A}$, and $E_2$ is a   topological plane  tangent to these spheres at  $p_{AB}$ and $p_{BA^{-1}}$. The genus-3 handlebody $H$ is the region bounded  by $E_1$ and $E_2$ and the 2-disks labeled by $I^{-}_0$ and $I^{+}_0$.}
	\label{figure:3ppabstract}
\end{figure}

The intersection of the spinal spheres of $B$ and $B^{-1}$ is a circle as we have shown in Section \ref{sec-Ford-3pp}, since $B$ is an regular element of order three, the action of $B$ on this circle  is a $\frac{2 \pi}{3}$-rotation. We take a particular point  on the circle:
$$u=\left[\frac{(\sqrt{6}-2)(\sqrt{-9+6\sqrt{6}})}{6}, \frac{4\sqrt{3}-9\sqrt{2}}{6}, \frac{2\sqrt{3}}{9}\sqrt{-9+6\sqrt{6}}\right]$$
in Heisenberg coordinates.

Other $B$-orbits of $u$ are
$$B(u)=\left[-\sqrt{\frac{-3+\sqrt{24}}{3}}, -\frac{1}{\sqrt{3}},\frac{2\sqrt{-3+\sqrt{24}}}{3}\right]$$
and
$$B^{2}(u)=\left[-\frac{(6-3\sqrt{6})\sqrt{-9+6\sqrt6}}{18}, \frac{9\sqrt{2}-8\sqrt{3}}{6}, \frac{(2\sqrt{3}-6\sqrt{2})\sqrt{-9+6\sqrt{6}}}{9}\right]$$
in Heisenberg coordinates.
In  Figure \ref{figure:3pp0}, $u$, $B(u)$ and  $B^2(u)$ are indicated as yellow, green and black dots, respectively.

From the points $u$, $B(u)$ and  $B^2(u)$ we know the orientation  of the $B$-action on the intersection circle of these two spinal spheres.
From  (\ref{3ppIplus}) and (\ref{3ppIminus}) below, it is also  easy to check that $u$, $B(u)$ and $B^2(u)$ lie in the intersection circle of the spinal spheres of $B$ and $B^{-1}$.

\subsection{A fundamental domain of $\langle A \rangle$-action on $\partial_{\infty}\mathbf{H}^2_{\mathbb C}\backslash\{q_{\infty}\}$} \label{subsection:3ppfundamentaldomain}

We will take two piecewise linear  planes $E_{1}$ and $E_{2}$ in  $\partial_{\infty}\mathbf{H}^2_{\mathbb C}\setminus\{q_{\infty}\}$, which cut out a region $U$ in  $\partial_{\infty}\mathbf{H}^2_{\mathbb C}\backslash\{q_{\infty}\}$, $U$ is homeomorphic to $\mathbb{R}^2 \times [0,1]$, which  is a  fundamental domain of $\langle A \rangle$-action on $\partial_{\infty}\mathbf{H}^2_{\mathbb C}\backslash\{q_{\infty}\}$. 
Let  $H$ be the  intersection of $U$ with the boundary at infinity of the Ford domain $D_{\Gamma_1}$. Then $H$
is a fundamental domain of $\Gamma_1$ on  $\partial_{\infty}\mathbf{H}^2_{\mathbb C}\backslash\{q_{\infty}\}$.


Let $\mathcal{C}$ be the contact plane based at the point with Heisenberg coordinate $[-\frac{1}{2},-\frac{1}{\sqrt{3}},0]$.
Let $\mathcal{L}_1$, $\mathcal{L}_2$ be the two lines in 
$\mathcal{C}$ whose images under vertical projection are the lines given by 
the equations $y+\sqrt{3}(x-1)=-\frac{1}{\sqrt{3}}$ and $y-\sqrt{3}(x-1)=-\frac{1}{\sqrt{3}}$ respectively, where $z=x+yi$.
The intersection point of $\mathcal{L}_1$ and $\mathcal{L}_2$ is the point
with Heisenberg coordinate $[1,-\frac{1}{\sqrt{3}},1-\frac{4}{\sqrt{3}}]$.
$E_{1,1}$ is the upper part of the fan defined by $\mathcal{L}_1$ and $E_{1,2}$ is the lower part of the fan defined by $\mathcal{L}_2$. See   \cite{Go} for an explicit definition of the fan. $E_{1,3}\cup E_{1,4}$ is the part of  $\mathcal{C}$ between $\mathcal{L}_1$ and $\mathcal{L}_2$.
Let $E_{1}=\cup^4_{i=1} E_{1,i}.$  Then $E_{1}$ is a piecewise linear plane from the description. It is similar to crooked plane appearing in anti de Sitter geometry, see \cite{BurelleCFG:2021, Drumm}. We thank the referee for  bringing this to our attention.  So  we also view $E_1$ as a "crooked-like surface".
See Figure \ref{figure:plane3pp} for  $E_{1,i}$, $1 \leq i \leq 4$.

In Heisenberg coordinates, we have

\begin{equation}
\begin{aligned}\label{E1}
	E_{1,1}&=\left\{\left[c+1, -\sqrt{3}c-\frac{1}{\sqrt{3}}, s+4\sqrt{3}c-\frac{8c+4}{\sqrt{3}}\right]  \,\Big| \, c \in \mathbb{R}, s \leq 1\right\}, \\
E_{1,2}&=\left\{\left[\frac{c}{\sqrt{3}}+1, c-\frac{1}{\sqrt{3}}, s-\frac{8c}{3}-\frac{4}{\sqrt{3
}}\right] \,\Big| \, c \in \mathbb{R}, s \geq 1\right\}, \\
E_{1,3}&=\left\{\left[\frac{ad}{\sqrt{3}}+1, d-\frac{1}{\sqrt{3}}, 1-2d-\frac{2ad}{3}-\frac{4}{\sqrt{3}}\right] \,\Big| \, a \in [-1,1], d \geq 0\right\},\\
E_{1,4}&=\left\{\left[\frac{ad}{\sqrt{3}}+1, d-\frac{1}{\sqrt{3}}, 1-2d-\frac{2ad}{3}-\frac{4}{\sqrt{3}}\right] \,\Big| \, a \in [-1,1], d \leq 0\right\}.	
\end{aligned}
\end{equation}

\begin{figure}
\begin{center}
\begin{tikzpicture}
\node at (0,0) {\includegraphics[width=6cm,height=6cm]{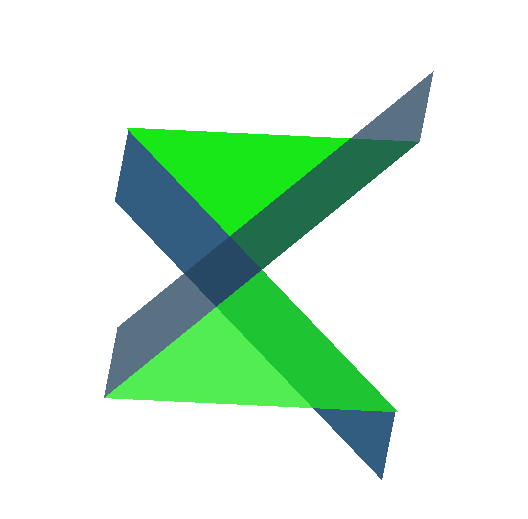}};
\coordinate [label=left:$E_{1,1}$] (S) at (1.8,2.2);
\coordinate [label=left:$E_{1,2}$] (S) at (-1.5,0.4);
\coordinate [label=left:$E_{1,4}$] (S) at (0,1.7);
\coordinate [label=left:$E_{1,3}$] (S) at (0,-1.8);
\end{tikzpicture}
\end{center}
 \caption{A  realistic view of part of the piecewise linear plane $E_{1}=\cup^4_{i=1} E_{1,i}$.}
	\label{figure:plane3pp}
\end{figure}

Let $E_{2,i}(1 \leq i \leq 4)$ be the image of $E_{1,i}(1 \leq i \leq 4)$ under the action of $A$.  Let $$E_{2}=\cup^4_{i=1} E_{2,i}.$$

In Heisenberg coordinates, we have
\begin{equation}
	\begin{aligned}\label{E2}
E_{2,1}&=\left\{\left[c-1, -\sqrt{3}c-\frac{1}{\sqrt{3}}, s-\frac{8c}{\sqrt{3}}\right] \,\Big| \, c \in \mathbb{R}, s \leq 1\right\}, \\
E_{2,2}&=\left\{\left[\frac{c}{\sqrt{3}}-1, c-\frac{1}{\sqrt{3}}
, s+\frac{4c}{3}\right] \,\Big| \, c \in \mathbb{R}, s \geq 1\right\},\\
	E_{2,3}&=\left\{\left[\frac{ad}{\sqrt{3}}-1, d-\frac{1}{\sqrt{3}}, 1+2d-\frac{2ad}{3}\right] \,\Big| \, a \in [-1,1], d \geq 0\right\},\\
	E_{2,4}&=\left\{\left[\frac{ad}{\sqrt{3}}-1, d-\frac{1}{\sqrt{3}}, 1+2d-\frac{2ad}{3}\right] \,\Big| \, a \in [-1,1], d \leq 0\right\}.
\end{aligned}
\end{equation}

\begin{prop}\label{prop:nonintersection}
	The intersection of 	$E_{1}$ and $A^{k}(E_{1})$ is empty  for any $k \neq 0 \in \mathbb{Z}$.
\end{prop}

\begin{proof} We only need to show the intersection of  $E_{1}$ and  $E_{2}=A(E_{1})$ is empty. We show this by direct calculations.

If $$\left[c_1+1, -\sqrt{3}c_1-\frac{1}{\sqrt{3}}, s_1+4\sqrt{3}c_1-\frac{8c_1+4}{\sqrt{3}}\right]\in E_{1,1}$$
and
$$\left[c_2-1, -c_2\sqrt{3}-\frac{1}{\sqrt{3}}, s_2-\frac{8c_2}{\sqrt{3}}\right] \in E_{2,1}$$
 is an intersection point of $E_{1,1}$ and  $E_{2,1}$, then $c_1=c_2$ and $c_1+1=c_2-1$, which is a contradiction, so $E_{1,1} \cap E_{2,1}=\emptyset$.

If $$\left[c_1+1, -\sqrt{3}c_1-\frac{1}{\sqrt{3}}, s_1+4\sqrt{3}c_1-\frac{8c_1+4}{\sqrt{3}}\right] \in E_{1,1}$$
and
$$\left[\frac{c_2}{\sqrt{3}}-1, c_2-\frac{1}{\sqrt{3}}
	, s_2+\frac{4c_2}{3}\right]  \in E_{2,2}$$
 corresponds to an intersection point of $E_{1,1}$ and  $E_{2,2}$,
then $c_2=-\sqrt{3}c_1$
and $c_1+1=\frac{c_2}{\sqrt{3}}-1$, so $c_1=-1$,  $s_1-4\sqrt{3}=s_2$ and
$s_1-4\sqrt{3}+\frac{4}{\sqrt{3}}=s_2-8$, but note that $s_1 \leq 1$ and  $s_2 \geq 1$, so we have  $E_{1,1} \cap E_{2,2}=\emptyset$.
	
If $$\left[c+1, -\sqrt{3}c-\frac{1}{\sqrt{3}}, s+4\sqrt{3}c-\frac{8c+4}{\sqrt{3}}\right]\in E_{1,1}$$
 and
 $$\left[\frac{ad}{\sqrt{3}}-1, d-\frac{1}{\sqrt{3}}, 1+2d-\frac{2ad}{3}\right] \in    E_{2,3}$$
  is an intersection point of $E_{1,1}$ and  $E_{2,3}$,
then $d=-\sqrt{3}c$
and $c+2=-ac$. So we must have $a \in (-1,1]$ (recall  we assume $a\in [-1,1]$),  and $c$ is negative,
then $s=1-6\sqrt{3}c+\frac{6c}{\sqrt{3}}$ has no solution for $s \leq 1$ and $c \leq 0$, so $E_{1,1} \cap E_{2,3} =\emptyset$.

If $$\left[c+1, -\sqrt{3}c-\frac{1}{\sqrt{3}}, s+4\sqrt{3}c-\frac{8c+4}{\sqrt{3}}\right]\in E_{1,1} $$ and
$$\left[\frac{ad}{\sqrt{3}}-1, d-\frac{1}{\sqrt{3}}, 1+2d-\frac{2ad}{3}\right]  \in   E_{2,4}$$
 is an intersection point of $E_{1,1}$ and  $E_{2,4}$,
then $d=-\sqrt{3}c$ and   $c+2=-ac$.
So we must have  $c$ is negative and $d$ is positive,
	but recall that for
	$$\left[\frac{ad}{\sqrt{3}}-1, d-\frac{1}{\sqrt{3}}, 1+2d-\frac{2ad}{3}\right]    \in E_{2,4}$$ we have $d \leq 0$,
	so $E_{1,1} \cap  E_{2,4}=\emptyset$.
	
Similarly, we can show that $E_{1,i} \cap E_{2,j} =\emptyset$ for all $1 \leq i,j \leq 4$, we omit the routine arguments.
		
\end{proof}

\begin{figure}
\begin{center}
\begin{tikzpicture}
\node at (0,0) {\includegraphics[width=6cm,height=6cm]{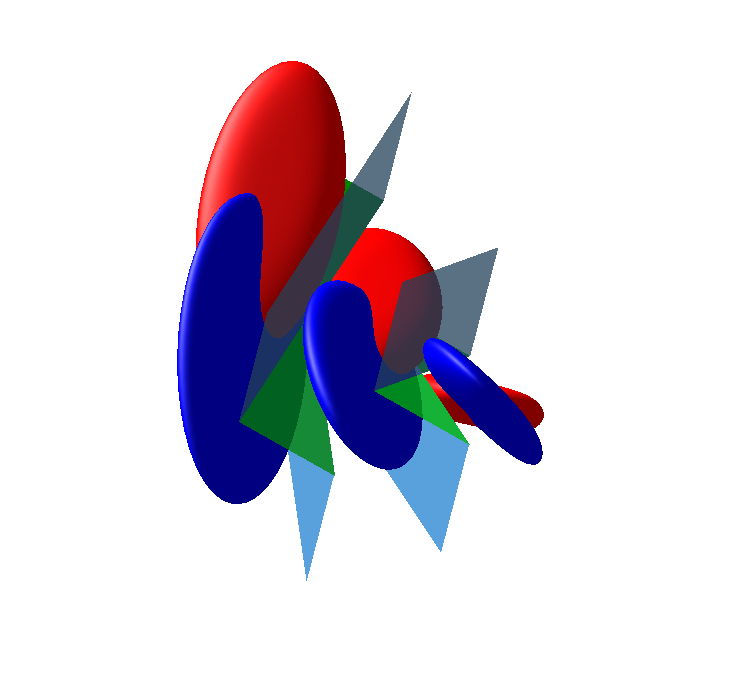}};
\coordinate [label=left:$ A(E_{1})$] (S) at (1.4,2.2);
\coordinate [label=left:$ E_{1}$] (S) at (1.5,1);
\end{tikzpicture}
\end{center}
 \caption{A  realistic view of a fundamental domain of the $\langle A \rangle$-action on $\partial_{\infty}\mathbf{H}^2_{\mathbb C}\backslash \{q_{\infty}\}$.  The plane $E_1$ is the same as in Figure \ref{figure:plane3pp}, and the spinal spheres are the same as in Figure	\ref{figure:3pp0}.}
\label{figure:fundamentaldomain3pp}
\end{figure}

\begin{prop}
	For any $k \in \mathbb{Z}$, the intersection of $E_{1}$ with $A^{k}(I(B))$ (resp.\ $A^{k}(I(B^{-1}))$) is reduced to $p_{B^{-1}A}$ (resp.\ $p_{BA}$).
\end{prop}

\begin{proof}

The  equation of the spinal sphere of  $I^{+}_{k}$    in Heisenberg coordinates  is	\begin{equation}\label{3ppIplus}
	|x+2k+yi|^4+(t-8k/\sqrt{3}-4ky)^2=\frac{16}{9}
	\end{equation}
		for $k \in \mathbb{Z}$.
The  equation of the spinal  sphere of  $I^{-}_{k}$ in Heisenberg coordinates is
\begin{equation}\label{3ppIminus}
	|x+2k+(y+2/\sqrt{3})i|^4+(t-4ky+4x/\sqrt{3})^2=\frac{16}{9}
	\end{equation}
	for $k \in \mathbb{Z}$.
So the vertical projection of the spinal sphere of  $I^{+}_{k}$  is the disk in $\mathbb{C}$  bounded by the circle
	\begin{equation}\label{proj3ppIplus}
	|x+2k+yi|^2=\frac{4}{3}
	\end{equation}
	for $k \in \mathbb{Z}$.
The vertical projection of the spinal sphere of  $I^{-}_{k}$  is the disk in $\mathbb{C}$ bounded by the circle
	\begin{equation}\label{proj3ppIminus}
	|x+2k+(y+2/\sqrt{3})i|^2=\frac{4}{3}
	\end{equation}
	for $k \in \mathbb{Z}$.
The vertical  projection of the boundary of the sector $E_{1,3}$ is the union of two half-lines
	\begin{equation}\label{3ppprocE131}
\{	b/\sqrt{3}+1+(b-i/\sqrt{3}), b\geq 0\}
	\end{equation}
	and
	\begin{equation}\label{3ppprocE132}
\{	-b/\sqrt{3}+1+(b-i/\sqrt{3}),  b\geq 0\}.
	\end{equation}
The vertical projection of the boundary of the sector $E_{1,4}$ is the union of two half-lines with the same equations as (\ref{3ppprocE131}) and (\ref{3ppprocE132}) but both with $b \leq 0$ respectively. See Figure \ref{figure:3ppproj}.

With the coordinates of $p_{B^{-1}A}$ and $p_{BA}$ in Section \ref{sec-Ford-3pp}, now it is easy to see that $p_{B^{-1}A}$ lies in the triple intersection of $E_{1,2}$,  $I^{+}_{0}$ and $I^{-}_{-1}$, and  $p_{BA}$ lies in the triple intersection of $E_{1,1}$,  $I^{+}_{-1}$ and $I_{0}^{-}$.

\begin{figure}
\begin{tikzpicture}[style=thick,scale=1.2]
\draw (0,0)  circle (1.1547);  \node [above right] at (-0.25,0.2){$I_{0}^{+}$};
\draw (0,-1.1547)  circle (1.1547); \node [below right] at (-0.25,-1.2){$I_{0}^{-}$};
\draw (2,0)  circle (1.1547); \node [above right] at (1.75,0.2){$I_{-1}^{+}$};
\draw (2,-1.1547)  circle (1.1547); \node [below right] at (1.75,-1.2){$I_{-1}^{-}$};
\draw (-2,0)  circle (1.1547);  \node [above right] at (-2.25,0.2){$I_{1}^{+}$};
\draw (-2,-1.1547)  circle (1.1547);\node [below right] at (-2.25,-1.2){$I_{1}^{-}$};
\draw[red, ultra thick] (-0.5,2.0207) -- (2.5,-3.1754);
\draw[red, ultra thick] (-0.5,-3.1754) -- (2.5,2.0207);
\filldraw[fill=green, opacity=0.7] (-0.372578,1.8) -- (1,-0.57735) -- (2.37258,1.8) -- cycle;
\filldraw[fill=cyan, opacity=0.7] (-0.340996,-2.9) -- (1,-0.57735) -- (2.341,-2.9) -- cycle;
\node [above ] at (-0.5,2){$E_{1,2}$};
\node [above ] at (2.5,2){$E_{1,1}$};
\node [above ] at (0.9,1){$E_{1,4}$};
\node [below ] at (1,-2.2){$E_{1,3}$};

\filldraw[fill=green, opacity=0.7] (-0.372578-2,1.8) -- (1-2,-0.57735) -- (2.37258-2,1.8) -- cycle;
\filldraw[fill=cyan, opacity=0.7] (-0.340996-2,-2.9) -- (1-2,-0.57735) -- (2.341-2,-2.9) -- cycle;
\node [above ] at (-0.5-2,2){$E_{2,2}$};
\node [above ] at (2.5-2,2){$E_{2,1}$};
\node [above ] at (0.9-2,1){$E_{2,4}$};
\node [below ] at (1-2,-2.2){$E_{2,3}$};

\draw[blue, ultra thick] (-0.5-2,2.0207) -- (2.5-2,-3.1754);
\draw[blue, ultra thick] (-0.5-2,-3.1754) -- (2.5-2,2.0207);
\end{tikzpicture}
\caption{The vertical projections of the planes $E_{i,1}$,  $E_{i,2}$, $E_{i,3}$ and $E_{i,4}$  and  the vertical projections of the  spinal spheres of $A^{k}BA^{-k}$ and $A^{k}B^{-1}A^{-k}$ for $k=-1,0,1$ in $\mathbb{C}$.   The union of the red lines are the vertical projections of $E_{1,1}$  and $E_{1,2}$;  The union of the  blue lines are the vertical  projections of $E_{2,1}$ and $E_{2,2}$;  The green and blue colored  regions co-bounded by the red lines is the vertical  projections of $E_{1,3}$ and $E_{1,4}$; The green and blue colored  regions co-bounded by the blue lines are the projections of $E_{2,3}$ and $E_{2,4}$. }	
\label{figure:3ppproj}
\end{figure}

Firstly,  we consider the intersection of  half-plane $E_{1,2}$ with all the  isometric spheres $I^{\pm}_{k}$:

(1). The vertical  projection of  $E_{1,2}$ is the line
	
\begin{equation}\label{3ppprocE12}
\{	c/\sqrt{3}+1+(c-1/\sqrt{3})i, c \in \mathbb{R}\}.
	\end{equation}
	
	From  (\ref{3ppprocE12}) and (\ref{3ppIplus}), we know that any intersection point  of $E_{1,2}$ and  $I^{+}_{0}$ is
	$$[c/\sqrt{3}+1+(c-1/\sqrt{3})i, t]$$  in Heisenberg coordinates  satisfies
	$$\left(\left(\frac{c}{\sqrt{3}}+1\right)^2 +\left(c-\frac{1}{\sqrt{3}}\right)^2\right)^2+t^2=\frac{16}{9}.$$
	But then $c=t=0$  is the only solution which corresponds to the point $p_{B^{-1}A}$.
	
(2). Similarly we can see  $p_{B^{-1}A}$  is the unique  intersection point  of $E_{1,2}$ and  $I^{-}_{-1}$.

(3). From  (\ref{3ppprocE12}) and (\ref{proj3ppIplus}), the intersection of the projections of   $E_{1,2}$ and  $I^{-}_{0}$ is non-empty, and by Equation  (\ref{3ppIminus}),   any intersection point  of $E_{1,2}$ and  $I^{-}_{0}$
is
$$\left[\frac{c}{\sqrt{3}}+1+(c-\frac{1}{\sqrt{3}})i, s-\frac{8c}{3}-\frac{4}{\sqrt{3
	}}\right] $$  in Heisenberg coordinates   with some $ c \in \mathbb{R}, s \geq 1$
	satisfying
\begin{equation}\label{3ppIminusE12}
\left(\left(\frac{c}{\sqrt{3}}+1\right)^2 +\left(c+\frac{1}{\sqrt{3}}\right)^2\right)^2+t^2=\frac{16}{9}.	\end{equation}
That is,
\begin{equation}\label{E120minus}
	\left(\frac{4c^2}{3}+\frac{4c}{\sqrt{3}} + \frac{4}{3}\right)^2+\left(s-\frac{4c}{3}\right)^2=\frac{16}{9}.
\end{equation}
	
If $c \in [-\infty, -\sqrt{3}] \cup  [0, \infty]$, then $4c^2/3+4c/\sqrt{3}$ is non-negative, so Equation (\ref{E120minus}) has no solution.
Note that $s \geq 1$, so for any $c \in [-\sqrt{3},0]$,  the minimal value of the left side of  Equation \ref{E120minus} is achieved when $s=1$, that is $$\left(\frac{4c^2}{3}+\frac{4c}{\sqrt{3}} + \frac{4}{3}\right)^2+\left(1-\frac{4c}{3}\right)^2.$$
For this degree-4 polynomial, in the interval $c \in [-\sqrt{3},0]$, it is easy to see that
when
$$c=\frac{(8\sqrt{3}+12+4\sqrt{25+12\sqrt{3}})^{1/3}}{4}-\frac{1}{(8\sqrt{3}+12+4\sqrt{25+12\sqrt{3}})^{1/3}}-\frac{\sqrt{3}}{2},$$
which is $c=-0.1937649849$ numerically, we get the minimum of this polynomial in this interval, it is $2.459384508  >  \frac{16}{9}$ numerically. So  the intersection of   $E_{1,2}$ and  $I^{-}_{0}$ is empty.

(4). Similarly, we have   the intersection of   $E_{1,2}$ and  $I^{+}_{-1}$ is empty.

(5).	 $E_{1,2}$ is disjoint from  $I^{\pm}_{k}$ for $k \in \mathbb{Z}\backslash\{0,-1\}$,  since   their projections in $\mathbb{C}$ are disjoint by  (\ref{3ppprocE12}), (\ref{proj3ppIplus}) and  (\ref{proj3ppIminus}).

The projection of  $E_{1,1}$ is the line
\begin{equation}\label{3ppprocE11}
	\{c+1+(-\sqrt{3}c-1/\sqrt{3})i, c\in \mathbb{R}\}.
	\end{equation}
Similarly to the case of $E_{1,2}$, we have   the intersection of   $E_{1,1}$ and $I^{\pm}_{k}$ for all $k \in \mathbb{Z}$  is just the point  $p_{BA}$.
	
Secondly,  we show  the intersection of   $E_{1,3}$ and $I^{\pm}_{k}$ is empty for any  $k \in \mathbb{Z}$.
	
(1). The intersection of  $E_{1,3}$  and $I^{+}_{-1}$ is empty, see the  vertical projection of  $E_{1,3}$ and $I^{+}_{-1}$ into $\mathbb{C}$ in  Figure 	\ref{figure:3ppproj}.

From Equation  (\ref{3ppIplus}), any intersection of  $E_{1,3}$  and $I^{+}_{-1}$ satisfies the equation
	\begin{equation}\label{E13Ifu1plus}
	\left(\left(\frac{ad}{\sqrt{3}}-1\right)^2+\left(d-\frac{1}{\sqrt{3}}\right)^2\right)^2+\left(1+2d-\frac{2ad}{3}\right)^2=\frac{16}{9}
	\end{equation}
	for  $a \in [-1,1]$ and $d \geq 0$.

The left-side of Equation (\ref{E13Ifu1plus}) is larger than  $$\left(\left(\frac{ad}{\sqrt{3}}-1\right)^2+\left(d-\frac{1}{\sqrt{3}}\right)^2\right)^2+\left(1+\frac{4d}{3}\right)^2,$$ since $a \in [-1,1]$ and $d \geq 0$.
	
If $d\in [1/4, \infty)$, then the left-side of  Equation (\ref{E13Ifu1plus}) is larger than $(1+1/3)^2=16/9$. If  $d\in [0,1/4]$, then the absolute value of $ad/\sqrt{3}$ is less than one, so   the left-side of  Equation (\ref{E13Ifu1plus}) is larger than $$\left(\left(1-\frac{d}{\sqrt{3}}\right)^2+\left(d-\frac{1}{\sqrt{3}}\right)^2\right)^2+1,$$ which is at least $(1-1/4\sqrt{3})^4+1$ and larger than $16/9$.
	
In total, there is no solution of (\ref{E13Ifu1plus}) with $d \geq 0$, so  the intersection of  $E_{1,3}$  and $I^{+}_{-1}$ is empty.
	
(2). Similarly,  the intersection of  $E_{1,3}$  and $I^{+}_{0}$ is empty.

(3). The intersection of  $E_{1,3}$  and $I^{-}_{-1}$ is empty.
From Equation  (\ref{3ppIminus}), any intersection of  $E_{1,3}$  and $I^{+}_{-1}$ is
	$$[ad/\sqrt{3}+1, d+1/\sqrt{3}, 1-2d-2ad/3-4/\sqrt{3}]$$  in Heisenberg coordinates   satisfying
\begin{equation}\label{E13Ifu1minus}
	\left(\left(\frac{ad}{\sqrt{3}}-1\right)^2+\left(d-\frac{1}{\sqrt{3}}\right)^2\right)^2+\left( 1+2d+\frac{2ad}{3}+\frac{4}{\sqrt{3}}\right)^2=\frac{16}{9}
	\end{equation}
	for some  $a \in [-1,1]$ and $d \geq 0$.
	Since  $a \in [-1,1]$ and $d \geq 0$,  the left-side of Equation (\ref{E13Ifu1minus}) is larger than $(1+4/\sqrt{3})^2$, which in turn is larger than $16/9$.
	So there is no solution of (\ref{E13Ifu1minus}) with $d \geq 0$,  the intersection of  $E_{1,3}$  and $I^{-}_{-1}$ is empty.
	
(4). Similarly,  the intersection of  $E_{1,3}$  and $I^{-}_{0}$ is empty.
	
(5). $E_{1,3}$ is disjoint from  $I^{\pm}_{k}$ for $k \in \mathbb{Z}\backslash\{0,-1\}$,  since  their vertical  projections in $\mathbb{C}$ are disjoint by (\ref{3ppprocE11}), (\ref{3ppprocE12}), (\ref{proj3ppIplus}) and  (\ref{proj3ppIminus}).

Similarly,  the intersection of $E_{1,4}$ and $I^{\pm}_{k}$ is empty for any  $k \in \mathbb{Z}$.
\end{proof}

\subsection{A fundamental domain of $\langle A \rangle$-action on   the ideal boundary of the Ford domain for $\Gamma_1$  is a genus-3 handlebody} \label{subsection:3pphandlebody}

Figure \ref{figure:3ppabstract} is a  combinatorial  picture of  the ideal boundary of the Ford domain of $\Delta_{3,\infty,\infty;\infty}$.
Each of the halfsphere is the intersection of the  spinal sphere of some $A^{k}BA^{-k}$ or
$A^{k}B^{-1}A^{-k}$ with the ideal boundary of the Ford domain.

Recall that the point $p_{BA}$ is the tangent point of the spinal spheres of $A^{-1}BA$ and $B^{-1}$. Similarly,  the points $p_{B^{-1}A}$, $p_{AB}$ and $p_{BA^{-1}}$ are  the tangent points of the spinal spheres of $\{A^{-1}B^{-1}A, B\}$,  $\{ABA^{-1}, B^{-1}\}$ and $\{AB^{-1}A^{-1}, B\}$ respectively. We also have $B(p_{B^{-1}A})=A(p_{B^{-1}A})=p_{BA^{-1}}$ and $B(p_{AB})=A^{-1}(p_{AB})=p_{BA}$. The $A$-action is the horizontal translation with a half-turn to the right in Figure \ref{figure:3ppabstract}. The ideal boundary of the Ford domain of $\Delta_{3,\infty,\infty;\infty}$ is the region which is  outside all the spheres in  Figure \ref{figure:3ppabstract}. Since there are infinitely  many tangent points between these spinal spheres,  the ideal boundary of the Ford domain is   an infinite genus handlebody in topology, which is  different from those  in  \cite{mx, ParkerWill:2016}.  Where the ideal boundaries of  Ford domains studied in  \cite{mx, ParkerWill:2016} are solid tori, that is, genus one handlebodies. 

We  have taken  a fundamental domain of the $A$-action on the ideal boundary of the Ford domain of $\Delta_{3,\infty,\infty;\infty}$ in Subsection \ref{subsection:3ppfundamentaldomain}. That is, the region $U$ co-bounded by $E_1$ and  $E_2$. So, $U$ is topologically  the product of the plane and the interval.

\begin{defn}The sub-region of $U \subset \mathbb{C} \times \mathbb{R}$ which is outside the isometric spheres of  $B$ and $B^{-1}$ is denoted by $H$.
\end{defn}
Both  the isometric spheres of  $B$ and $B^{-1}$ bound a finite 3-ball in   $U$, these two 3-balls intersect in a disk, so the union of them  is a big 3-ball. But this big 3-ball  is tangent to the frontiers of $U$, say $E_1$ and $E_2$, in four points.   So $H$ can  also be viewed as the complement of this big 3-ball in $U$, then   $H$ is a genus three handlebody. Compare to Figure \ref{figure:3ppabstract}.  We will take three disks $D_1$,  $D_2$ and  $D_3$ which cut $H$  into a 3-ball rigorously. Then the 3-manifold $M$ at infinity of $\Delta_{3,\infty,\infty;\infty}$ is the
quotient of a topological 3-ball by side pairings.  Which in turn enables us to  write down precisely the fundamental group of $M$ in 
 the end of Subsection  \ref{subsection:threemorearcs}.



\subsubsection{The first disk $D_1$} \label{3ppD1}

Firstly, we take a half-plane $\mathcal{P}_1$ in the  Heisenberg group which passes  through  $p_{BA^{-1}}$,   $p_{BA}$ and another point with  Heisenberg coordinates $[0,0,0]$. Note that $[0,0,0]$ is a point in the isometric sphere of $B^{-1}$.
The half-plane   $\mathcal{P}_1$ is given by
\begin{equation}\label{3ppP1}
\left\{\left[a, \frac{-1+b}{\sqrt{3}},\frac{-4a}{\sqrt{3}}\right] \,\Big| \, a \in \mathbb{R}, b
\leq 0 \right\}.
\end{equation}

We now consider the intersections of $\mathcal{P}_1$ with $E_1$, $E_2$ and the isometric sphere of $B^{-1}$. From  (\ref{E1}) and (\ref{3ppP1}), the intersection of  $\mathcal{P}_1$  and $E_{1,1}$ is a half-line

\begin{equation}\label{P1E11}
\left\{\left[c+1, -\sqrt{3}c-\frac{1}{\sqrt{3}}, \frac{4c-4}{\sqrt{3}}\right]  \,\Big| \,  c \geq 0 \right\}.
\end{equation}
with  one vertex  $p_{BA}$ and  diverges to $q_{\infty}$. We denote this half-line by $[p_{BA}, q_{\infty}]$.
From   (\ref{E1}) and (\ref{3ppP1}), it follows that the intersection of  $\mathcal{P}_1$  and $E_{1,i}$ is empty for $i=2,3,4$.

From  (\ref{E2}) and (\ref{3ppP1}), the intersection of  $\mathcal{P}_1$  and $E_{2,2}$ is a half-line
\begin{equation}\label{P1E22}
\left\{\left[\frac{c}{\sqrt{3}}-1, c-\frac{1}{\sqrt{3}}, \frac{4-4c}{\sqrt{3}}\right]  \,\Big| \,  c \leq 0 \right\}.
\end{equation}
with  one vertex  $p_{BA^{-1}}$ and  diverges to  $q_{\infty}$. We denote this half-line by $[p_{BA^{-1}}, q_{\infty}]$.

By  (\ref{E2}) and (\ref{3ppP1}), the intersection of  $\mathcal{P}_1$  and $E_{2,i}$ is empty for $i=1,3,4$.

From  (\ref{3ppIminus}) and (\ref{3ppP1}), the intersection of  $\mathcal{P}_1$  and the spinal sphere of $B^{-1}$ is an arc
\begin{equation}\label{P1b}
\left[a,\frac{-1+b}{\sqrt{3}},\frac{-4a}{\sqrt{3}}\right]
\end{equation}
in the  Heisenberg group
with $a^2+\frac{(b+1)^2}{3}=\frac{4}{3}$ and
$b \leq 0$. It is an arc
with  one vertex  $p_{BA^{-1}}$ and  another vertex $p_{BA}$.
We denote this arc  by $[p_{BA^{-1}}, p_{BA}]$. We remark here   $[p_{BA^{-1}}, p_{BA}]$ is a superior arc in an ellipse, so
we divide this arc into two subarcs and parameterize them  by $b\in [-3,0]$ and $a= \pm\sqrt{\frac{4-(b+1)^2}{3}}$.
By  Equation (\ref{3ppIplus}),  the intersection of  $\mathcal{P}_1$  and the spinal sphere of $B$ is empty.

Combining above equations,  we see that the three arcs  $[p_{BA}, q_{\infty}]$, $[p_{BA^{-1}}, q_{\infty}]$ and  $[p_{BA^{-1}}, p_{BA}]$ have disjoint interiors. So they glue together to get a simple closed curve in $\mathcal{P}_1$.
Now the region in the plane  $\mathcal{P}_1$ co-bounded by $[p_{BA}, q_{\infty}]$, $[p_{BA^{-1}}, q_{\infty}]$ and  $[p_{BA^{-1}}, p_{BA}]$ is a disk $D_1$, and the interior of  $D_1$ is disjoint from all the spinal spheres and $A^{k}(E_1)$ for any  $k \in \mathbb{Z}$.

\subsubsection{The second disk $D_2$} \label{3ppD2}

Secondly, we note  the equation of  $B^{-1}\left([p_{BA},p_{BA^{-1}}]\right)$ is
\begin{equation}\label{3ppbasecurve}
\left[-a, -\frac{b+1}{\sqrt{3}},0\right]
\end{equation}
in   Heisenberg coordinates, with $b \leq 0$ and $a^2+(b+1)^2/3=4/3$. It is an arc in the spinal sphere of $B$,  its end points are $p_{B^{-1}A}$ and $p_{AB}$.   We also divide this arc into two subarcs and  parameterize them  by $b\in [-3,0]$ and $a= \pm\sqrt{\frac{4-(b+1)^2}{3}}$.

The half-line $A^{-1}([p_{BA^{-1}}, q_{\infty}])$
is 
\begin{equation}
\left[\frac{c}{\sqrt{3}}+1, c-\frac{1}{\sqrt{3}}, -\frac{16c}{3}\right]
\end{equation}
in Heisenberg coordinates with $c \leq 0$.
We rewrite it as
\begin{equation}\label{P2E12}
\left[-\frac{c}{\sqrt{3}}+1, -c-\frac{1}{\sqrt{3}}, \frac{16c}{3}\right]
\end{equation}
in Heisenberg coordinates with  $ c \geq 0$.
One vertex of it is  $p_{AB}$ and  it diverges to the infinity. We denote this half-line by $[p_{AB}, q_{\infty}]$. It lies in the  half-plane  $E_{2,1}$.

The equation of  $A([p_{BA}, q_{\infty}])$
is
\begin{equation}\label{P2E210}
\left[c-1, -\sqrt{3}c-\frac{1}{\sqrt{3}}, -\frac{16c}{\sqrt{3}}\right]
\end{equation}
in Heisenberg coordinates with $c \geq 0$.
One vertex of it is  $p_{B^{-1}A}$ and it diverges to the infinity. We denote this half-line by $[p_{B^{-1}A}, q_{\infty}]$. It is a half-line in $E_{2,2}$.

We now construct  a ruled surface $\mathcal{P}_2$ in the Heisenberg group with  base curve   $B^{-1}([p_{BA},p_{BA^{-1}}])$, and  two of  $\mathcal{P}_2$'s  half-lines are  $A^{-1}([p_{BA^{-1}}, q_{\infty}])=[p_{B^{-1}A}, q_{\infty}]$ and $A([p_{BA}, q_{\infty}])=[p_{AB}, q_{\infty}]$.

The disk  $D_2$ is half of the ruled surface  $\mathcal{P}_2$. It  is a union of four subdisks
$$D_2=\cup^4_{i=1} D_{2,i}.$$ Each of $D_{2,i}$ is a part of the  ruled surface $D_2$, for $1 \leq  i \leq 4$.

The subdisk $D_{2,1}$ is
\begin{equation}\label{D21}
\left[\frac{\sqrt{4-(b+1)^2}}{\sqrt{3}}, \frac{-(b+1)}{\sqrt{3}},0\right]+s\left[\frac{-b-1}{\sqrt{3}}, -b-1, \frac{16+6b}{3}\right]
\end{equation}
with $b \in  [-2, 0]$ and $s \in [0, \infty)$
in  Heisenberg coordinates. Here
\begin{equation}
\left[\frac{\sqrt{4-(b+1)^2}}{\sqrt{3}}, \frac{-(b+1)}{\sqrt{3}},0\right]
\end{equation}
with  $b \in  [-2, 0]$ is the base curve. Moreover,  for each fixed  $b \in  [-2, 0]$, the equation
\begin{equation}
s\left[\frac{-b-1}{\sqrt{3}}, -b-1, \frac{16+6b}{3}\right]
\end{equation}
with $s \in[0, \infty)$
is a half-line in the  Heisenberg group.  $D_{2,1}$ is the pink ruled surface in Figure 	\ref{figure:D2-3pp}, with one half-line of it lies in $E_1$.

The subdisk $D_{2,2}$ is
\begin{multline}\label{D22}
\left[\frac{\sqrt{4-(b+1)^2}}{\sqrt{3}}, \frac{-(b+1)}{\sqrt{3}},0\right]\\
+s\left[\frac{4+2\sqrt{3}+b+b\sqrt{3}}{2\sqrt{3}},\frac{4-2\sqrt{3}-b\sqrt{3}+b}{2}, \frac{28-16\sqrt{3}-8\sqrt{3}b+12b}{3}\right]
\end{multline}
with $b \in  [-3, -2]$ and $s \in[0, \infty)$
in  Heisenberg coordinates.  $D_{2,2}$ is the green  ruled surface in Figure 	\ref{figure:D2-3pp}, it has a common half-line with $D_{2,1}$.

The subdisk $D_{2,3}$ is
\begin{multline}\label{D23}
\left[-\frac{\sqrt{4-(b+1)^2}}{\sqrt{3}}, \frac{-(b+1)}{\sqrt{3}},0\right]\\
+s\left[\frac{-2-4\sqrt{3}-b-b\sqrt{3}}{2\sqrt{3}},\frac{4\sqrt{3}-2+\sqrt{3}b-b}{2},
\frac{-28\sqrt{3}+16+8b-12\sqrt{3}b}{3}\right]
\end{multline}
with $b \in  [-3, -2]$ and $s \in[0, \infty)$
in  Heisenberg coordinates.  $D_{2,3}$ is the red   ruled surface in Figure 	\ref{figure:D2-3pp}, it has a common half-line with $D_{2,2}$.

The subdisk $D_{2,4}$ is
\begin{equation}\label{D24}
\left[-\frac{\sqrt{4-(b+1)^2}}{\sqrt{3}}, \frac{-(b+1)}{\sqrt{3}},0\right]+s\left[b+1, -\sqrt{3}b-\sqrt{3}, \frac{-16-6b}{\sqrt{3}}\right]
\end{equation}
with $b \in  [-2, 0]$ and $s \in[0, \infty)$
in   Heisenberg coordinates.  $D_{2,4}$ is the blue    ruled surface in Figure 	\ref{figure:D2-3pp}, it has a common half-line with $E_{2}$.

\begin{figure}

	\begin{center}
		\begin{tikzpicture}
		\node at (0,0) {\includegraphics[width=14cm,height=8cm]{{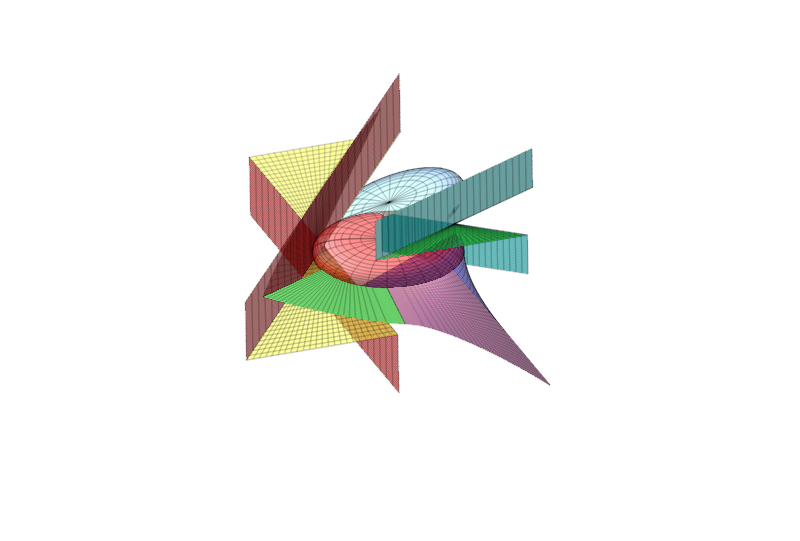}}};
		\coordinate [label=left:$D_{2,1}$] (S) at (-1.6,0.1);
		\coordinate [label=left:$D_{2,2}$] (S) at (-0.5,-0.5);
		\coordinate [label=left:$D_{2,3}$] (S) at (1.2,-0.7);
		\coordinate [label=left:$D_{2,4}$] (S) at (1.9,-0.1);
		\end{tikzpicture}
	\end{center}
	\caption{The cutting disk $D_2$
		we need in the study of the 3-manifold at infinity of  the group $\Delta_{3,\infty,\infty;\infty}$.}
	\label{figure:D2-3pp}
\end{figure}
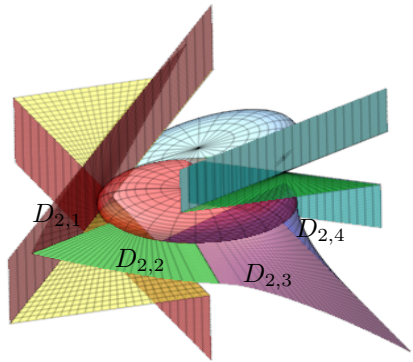

Note that for $b=0$ in Equation (\ref{D21}), we get Equation  (\ref{P2E12})    of $A^{-1}([p_{BA^{-1}}, q_{\infty}])$, and  for $b=0$ in Equation (\ref{D24}), we get  Equation  (\ref{P2E210})    for  $A([p_{BA}, q_{\infty}])$.

We now consider the intersections of  disks $D_{2,i}$ for $i=1,2,3,4$.
\begin{itemize}
\item For $b=-2$ in  (\ref{D21}) and (\ref{D22}), we get a common half-line in both  $D_{2,1}$ and $D_{2,2}$; 
\item For $b=-3$  in  (\ref{D22}) and (\ref{D23}), we get a common half-line in both  $D_{2,2}$ and $D_{2,3}$;

\item  For $b=-2$  in  (\ref{D23}) and (\ref{D24}), we get a common half-line in both  $D_{2,3}$ and $D_{2,4}$.
\end{itemize}
\begin{lem} \label{D2for3pp}
	The disks $D_{2,i}$ for $1 \leq i \leq 4$,   intersect only in above three half-lines $$D_{2,1} \cap D_{2,2}, ~~D_{2,2} \cap D_{2,3},~~D_{2,3} \cap D_{2,4}. $$ Then $D_2=\cup^4_{i=1} D_{2,i}$,  is a disk whose  boundary consists of three arcs  $A^{-1}([p_{BA^{-1}} q_{\infty}])$, $B^{-1}([p_{BA},p_{BA^{-1}}])$  and $A([p_{BA}, q_{\infty}])$.
\end{lem}
\begin{proof}From the  equations of all the sub-disks, Lemma \ref{D2for3pp} can be  checked case-by-case,  we omit the routine details.
\end{proof}

\subsubsection{The third disk $D_3$} \label{3ppD3}
Thirdly, we take a plane $\mathcal{P}_3$ in the  Heisenberg group  with $y=\frac{-1}{\sqrt{3}}$, note that $\mathcal{P}_3$ passes through $p_{BA^{-1}}$ and  $p_{AB}$  (and in fact $\mathcal{P}_3$ also passes through $p_{BA}$ and   $p_{B^{-1}A}$, but we do not use this fact). From  (\ref{3ppIplus}) and (\ref{3ppIminus}), we know there are four triple intersection points of $\mathcal{P}_3$, $I^{+}_{0}$ and $I^{-}_{0}$, which are
$$\left[0, \frac{-1}{\sqrt{3}}, \pm\frac{\sqrt{5}}{\sqrt{3}}\right], $$
$$\left[\sqrt{\frac{-3+\sqrt{24}}{3}}, \frac{-1}{\sqrt{3}}, -\frac{2}{3}\sqrt{-3+\sqrt{24}}\right]$$  and
$$\left[-\sqrt{\frac{-3+\sqrt{24}}{3}}, \frac{-1}{\sqrt{3}}, \frac{2}{3}\sqrt{-3+\sqrt{24}}\right]$$
in   Heisenberg coordinates.
We also note that  the last one is just $B(u)$ we have defined.

Note that the intersection  $\mathcal{P}_3 \cap I^{+}_{0}$
is the set $\left\{\left[x, \frac{-1}{\sqrt{3}},t\right]\right\}$  in   Heisenberg coordinates
with the equation
$$x^4+\frac{2x^2}{3}+t^2=\frac{5}{3}.$$
Let $[B(u),p_{AB}]$ be the arc in $\mathcal{P}_3 \cap I^{+}_{0}$, such that $x \in \left[-1, -\sqrt{\frac{-3+\sqrt{24}}{3}}\right]$
and $t \geq 0$.

The intersection $\mathcal{P}_3 \cap I^{-}_{0}$
is the set $\left\{\left[x, \frac{-1}{\sqrt{3}},t\right]\right\}$  in   Heisenberg coordinates
with the equation
$$x^4+\frac{2x^2}{3}+\left(t+\frac{4x}{\sqrt{3}}\right)^2=\frac{5}{3}.$$
Let  $[B(u), p_{BA^{-1}}]$ be the arc in $\mathcal{P}_3 \cap I^{-}_{0}$, such that $x \in \left[-1, -\sqrt{\frac{-3+\sqrt{24}}{3}}\right]$ and $t=-\sqrt{-x^4-\frac{2x^2}{3}+\frac{5}{3}}-\frac{4x}{3}$.
We note that $\sqrt{\frac{-3+\sqrt{24}}{3}}$  is $0.7956086739$ numerically.

The intersection $\mathcal{P}_3 \cap E_{2,1}$  is
$$\left[-1,  -1/\sqrt{3},s\right]$$ in   Heisenberg coordinates with $s \leq 1$,
the intersection $\mathcal{P}_3 \cap E_{2,2}$  is
$$\left[-1,  -1/\sqrt{3},s\right]$$
in   Heisenberg coordinates with $s \geq 1$.

We take $[p_{BA^{-1}}, p_{AB}]$ to be the arc
$$\left[-1,  -1/\sqrt{3},s\right]$$ in   Heisenberg coordinates with $s \in [0, 4/\sqrt{3}]$. So the half-part of  $[p_{BA^{-1}}, p_{AB}]$ lies in $\mathcal{P}_3 \cap E_{2,1}$, and the other half-part of  $[p_{BA^{-1}}, p_{AB}]$ lies in $\mathcal{P}_3 \cap E_{2,2}$.

From the  equations above, it is easy to see the three arcs  $[B(u), p_{BA^{-1}}]$, $[B(u), p_{AB}]$ and $[p_{BA^{-1}}, p_{AB}]$ have disjoint interior, so they glue together to get a circle in  $\mathcal{P}_3$. Let $D_3$  denote the disk bounded by this circle in $\mathcal{P}_3$.

\begin{lem} \label{disjointcuttingdisks3pp}  The point $p_{BA^{-1}}$ is the  common vertex of $D_1$ and $D_3$, the point $p_{AB}$ is the  common vertex of $D_2$ and $D_3$,   and the point $q_{\infty}$ is the  common vertex of $D_1$ and $D_2$.
Moreover, the  interiors of disks $D_{j}$ for $1 \leq j \leq 3$   are disjoint from each other.
\end{lem}

\begin{lem} \label{disjointcuttingdisksisometricsphere3pp}
	The interiors of  disks $D_{j}$ for $1 \leq j \leq 3$ are disjoint from $E_1$ and $E_2$, they are also disjoint from the spinal spheres of $B$ and $B^{-1}$.
\end{lem}

We can prove Lemmas \ref{disjointcuttingdisks3pp} and \ref{disjointcuttingdisksisometricsphere3pp} from the  equations of all the surfaces.  We omit the routine details here.

\subsection{Three more arcs in planes $E_1$, $E_2$, and in the spinal spheres of $B$, $B^{-1}$.}
\label{subsection:threemorearcs}
We still need more arcs in  the study of the 3-manifold at infinity of  the group $\Delta_{3,\infty,\infty;\infty}$.

Let $[p_{BA}, p_{B^{-1}A}]$ be the arc which is the  $A^{-1}$-image of the arc $[p_{BA^{-1}}, p_{AB}]$, so its equation is
$$\left[-1, -1/\sqrt{3},s- 4/\sqrt{3}\right]$$  in   Heisenberg coordinates with $s \in [0, 4/\sqrt{3}]$.
See the arc in Figure 	\ref{figure:3ppmorearcs} connecting $p_{BA}$ to $p_{BA^{-1}}$.

We define
$$\alpha(s):=-3s^4-2s^2+5, s \in \left[-1, -\sqrt{-1+\sqrt{8/3}}\right].$$
Let $[p_{B^{-1}A},u]$ be the arc which is the  $B^{-1}$-image of the arc $[p_{BA^{-1}}, B(u)]$, so its equation is
$$\left[\frac{-3s^3-s-\sqrt{\alpha(s)}}{4},\frac{-\sqrt{3}s^2-1/\sqrt{3}+s\sqrt{3\alpha(s)}}{4}, \frac{\sqrt{3\alpha(s)}}{3}\right]$$
in   Heisenberg coordinates. It is the arc in Figure 	\ref{figure:3ppmorearcs} connecting the points $p_{B^{-1}A}$ and $u$.

Let $[p_{BA},B^{2}(u)]$ be the arc which is the  $B$-image of the arc $[p_{AB}, B(u)]$, so its equation is
$$\left[-\frac{3s^3+s+\sqrt{\alpha(s)}}{4},\frac{3\sqrt{3}s^2-7\sqrt{3}-3\sqrt{3}s\sqrt{\alpha(s)}}{12},\frac{3s^3+s
}{\sqrt{3}}\right]
,$$
in   Heisenberg coordinates.
It is the arc in Figure 	\ref{figure:3ppmorearcs} connecting the points $p_{BA}$ and $B^{2}(u)$.

\begin{figure}
\begin{center}
\begin{tikzpicture}
\node at (0,0) {\includegraphics[width=7cm,height=7cm]{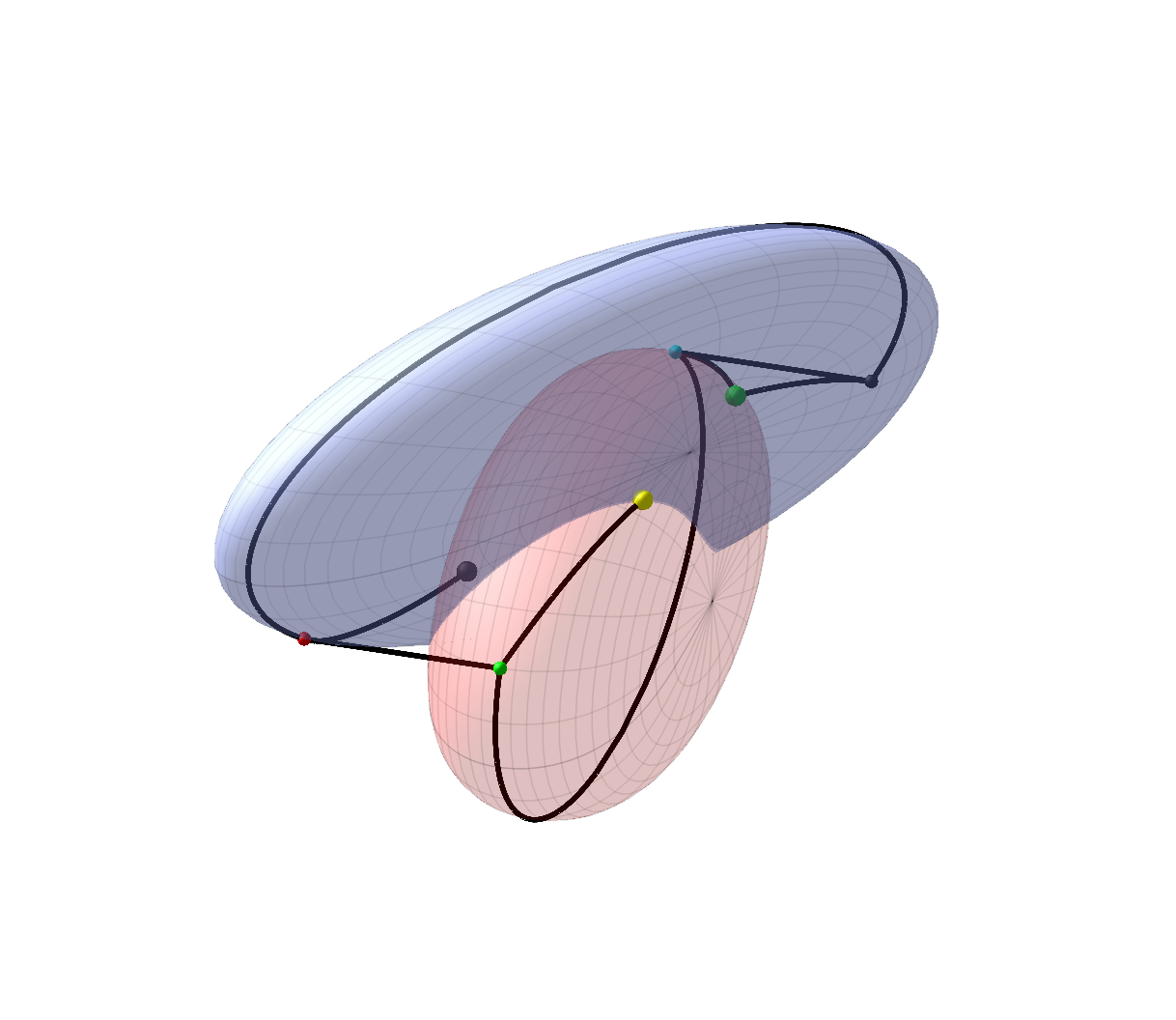}};
\coordinate [label=left:$ \scriptscriptstyle{I_{0}^{-}}$] (S) at (1.4,2.2);
\coordinate [label=left:$ \scriptscriptstyle{I_{0}^{+}}$] (S) at (0.3,-2.4);
\coordinate [label=left:$\scriptscriptstyle{u}$] (S) at (0.6,-0.1);
\coordinate [label=left:$\scriptscriptstyle{p_{B^{-1}A}}$] (S) at (0.5,-1.1);
\coordinate [label=left:$\scriptscriptstyle{p_{BA}}$] (S) at (-1.5,-1.1);
\coordinate [label=left:$\scriptscriptstyle{B^{-1}(u)}$] (S) at (-0.3,-0.2);
\coordinate [label=left:$\scriptscriptstyle{B(u)}$] (S) at (0.85,1.3);
\coordinate [label=left:$\scriptscriptstyle{p_{BA^{-1}}}$] (S) at (2.8,0.9);
\coordinate [label=left:$\scriptscriptstyle{p_{AB}}$] (S) at (1.4,0.6);
\end{tikzpicture}
\end{center}
 \caption{A realistic view  of all the arcs we need in the study of the 3-manifold at infinity of  the group $\Delta_{3,\infty,\infty;\infty}$.}
	\label{figure:3ppmorearcs}
\end{figure}
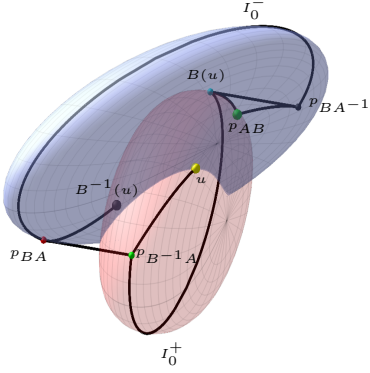

We take $[u,B(u)]$ as the arc in the intersection circle of the spinal spheres of $B$ and $B^{-1}$ with end points   $u$ and $B(u)$, but it does not contain $B^2(u)$; We take $[B(u), B^2(u)]$ as the arc in the intersection circle of the spinal spheres of $B$ and $B^{-1}$ with end points  $B(u)$ and $B^2(u)$, but it does not contain $u$; We take $[B^2(u), u]$ as the arc in the intersection circle of the spinal spheres of $B$ and $B^{-1}$ with end points   $u$ and $B^2(u)$, but it does not contain $B(u)$.

In total, we have  eleven arcs $[u,B(u)]$,  $[B(u), B^2(u)]$, $[B^2(u),u]$,  $[p_{AB},p_{BA^{-1}}]$, $[p_{BA},p_{B^{-1}A}]$,
$[B(u), p_{BA^{-1}}]$, $[B(u), p_{AB}]$,  $[B^2(u), p_{BA}]$,   $[u, p_{B^{-1}A}]$,
$[p_{BA}, p_{BA^{-1}}]$  and
$[p_{B^{-1}A}, p_{AB}]$. These arcs intersect on their common end vertices, that is, $u$, $B(u)$,  $B^2(u)$,  $p_{AB}$, $p_{A^{-1}B}$, $p_{B^{-1}A}$,
and $p_{BA^{-1}}$.

\begin{lem} \label{disjointarc3pp}
	The eleven arcs above only intersect in their common end vertices.
\end{lem}
\begin{proof} This lemma can be checked case-by-case  from the  equations of all the arcs, we omit the details.
\end{proof}

We denote  the complement of the two balls bounded by  the spinal  spheres of $B$ and $B^{-1}$ in $U$ by $H$.
 $H$ is just part of  the ideal boundary of the Ford domain, the  boundary of $H$ is  consists of  $E_1$ and  $E_2$, also parts of  the spinal  spheres of $B$ and $B^{-1}$, see Figure \ref{figure:3ppdisks} for an abstract picture of this,
  and the reader should compare with Figures 	\ref{figure:fundamentaldomain3pp} and \ref{figure:3ppabstract}. $H$ is a fundamental domain for the boundary 3-manifold of $\Gamma_1$.
  We will show  $H$ is a  genus 3 handlebody.

\begin{figure}
\begin{tikzpicture}[style=thick,scale=1]

\draw (-4,4) -- (4,4);
\draw (-4,-4) -- (4,-4);

\draw (-2,4) arc (90:270: 0.8 and 4);
\draw[dashed] (-2,-4) arc (-90:90: 0.8 and 4);
\draw[dashed] (2,4) arc (90:270: 0.8 and 4);
\draw (2,-4) arc (-90:90: 0.8 and 4);
\draw[fill] (-2,4) circle [radius=0.08]; \node [above] at (-2,4.1){$q_{\infty}$};
\draw[fill] (-2,1.732) circle [radius=0.08]; \node [above left] at (-2,1.732){$p_{BA}$};
\draw[fill] (-2,-1.732) circle [radius=0.08];\node [below left] at (-2,-1.732){$p_{B^{-1}A}$};
\draw[fill] (-2,-4) circle [radius=0.08];\node [below] at (-2,-4.1){$q_{\infty}$};

\draw[fill] (2,4) circle [radius=0.08]; \node [above] at (2,4.1){$q_{\infty}$};
\draw[fill] (2,1.732) circle [radius=0.08]; \node [above right] at (2,1.732){$p_{BA^{-1}}$};
\draw[fill] (2,-1.732) circle [radius=0.08];\node [below right] at (2,-1.732){$p_{AB}$};
\draw[fill] (2,-4) circle [radius=0.08];\node [below] at (2,-4.1){$q_{\infty}$};
\node [right] at (-3.5,0){$E_{1}$};
\node [right] at (3,0){$E_{2}$};

\draw[fill] (1,0) circle [radius=0.08];\node [right] at (1,0){$B(u)$};
\draw[fill] (-0.5,0.2598) circle [radius=0.08];\node [above] at (-0.5,0.2598){$B^{-1}(u)$};
\draw[fill] (0,-0.3) circle [radius=0.08];\node [below right] at (0,-0.3){$u$};

\draw[red, line width = 3pt] (1,0) arc (-60:0:2);
\draw[blue,line width = 3pt] (2,1.732) arc (0:180:2);
\draw(-2,1.732) arc (180:240: 2);

\draw[red,line width = 3pt] (1,0) arc (60:0:2);
\draw[green,line width = 3pt] (2,-1.732) arc (0:-180:2);
\draw (-2,-1.732) arc (-180:-240:2);

\draw[blue,line width = 3pt] (-2,1.732) -- (-2,4);
\draw (-2,1.732) -- (-2,-1.732);
\draw[green,line width = 3pt] (-2,-4) -- (-2,-1.732);

\draw[blue,line width = 3pt] (2,4) -- (2,1.732);
\draw[red,line width = 3pt] (2,-1.732) -- (2,1.732);
\draw[green,line width = 3pt] (2,-1.732) -- (2,-4);

 \draw[dashed] (1,0) arc (0:180:1 and 0.3);
\draw(-1,0) arc (180:360:1 and 0.3);

\draw (-2,-1.732) parabola (0,-0.3);
\draw[dashed] (-0.5,0.2598) parabola (-2,1.732);

\end{tikzpicture}
\caption{A  combinatorial  picture of the  cutting disks $D_1$, $D_2$ and $D_3$ of a fundamental domain $H$ for the $\langle A \rangle$-action on the ideal boundary of  the Ford domain of $\Delta_{3,\infty,\infty;\infty}$. $D_1$ is a disk bounded by the blue circle (the disk with three blue thick boundary arcs, with vertices labeled by $q_{\infty}$, $p_{BA}$ and $p_{BA^{-1}}$); $D_2$ is a disk bounded by the green circle  (the disk with three green  thick boundary arcs,  with vertices labeled by $q_{\infty}$, $p_{AB}$ and $p_{B^{-1}A}$); $D_3$ is a disk bounded by the red circle (the disk with three red thick boundary arcs, with vertices labeled by $B(u)$, $p_{AB}$ and $p_{BA^{-1}}$).}
	\label{figure:3ppdisks}
\end{figure}
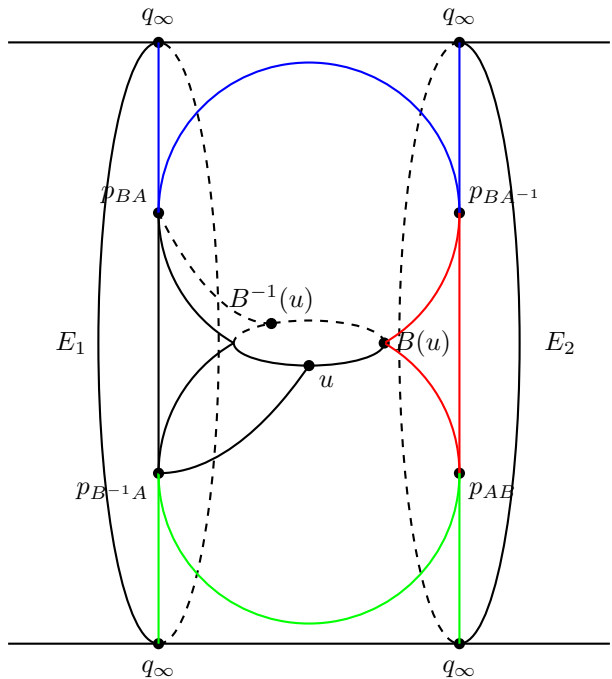

Now we cut $H$ along disks $D_1$, $D_2$ and $D_3$, we get a 3-manifold $N$, we show the boundary of the 3-manifold $N$  in Figure \ref{figure:3pppolytope}.
What is shown in Figure \ref{figure:3pppolytope} is a disk, but we glue the pairs of edges labeled by  $e_{1,i}$ for $i=1,2,3$  matching the labels of ending vertices, we get an annulus,  and then we
glue the pairs of edges labeled by  $e_j$ for $i=2,3$   matching the labels of ending vertices, we get a 2-sphere.

We note that  for each disk $D_{i}$, now there are two copies of it in the boundary of $N$, we denote them by $D_{i,+}$ and $D_{i,-}$.    $D_{i,-}$ is the copy of $D_{i}$  which is closer to us in Figure
\ref{figure:3ppdisks}, and  $D_{i,+}$ is the copy of $D_{i}$  which is further away from us in Figure
\ref{figure:3ppdisks}.

The plane $E_1$ is divided   by three arcs $[q_{\infty}, p_{BA}]$, $[p_{BA},p_{B^{-1}A}]$ and $[p_{B^{-1}A}, q_{\infty}]$ into two disks $E_{1,+}$ and $E_{1,-}$, where $E_{1,-}$  is the one  which is closer to us in Figure
\ref{figure:3ppdisks}, and $E_{1,+}$   is further away from us in Figure
\ref{figure:3ppdisks}.
The plane  $E_2$ is divided by three arcs $[q_{\infty}, p_{BA^{-1}}]$, $[p_{BA^{-1}}, p_{AB}]$ and $[p_{AB}, q_{\infty}]$ into two disks $E_{2,+}$ and $E_{2,-}$, where $E_{2,-}$  is the one  which is closer to us in Figure
\ref{figure:3ppdisks},   and $E_{2,+}$   is further away from us in Figure
\ref{figure:3ppdisks}.
The intersection of  the spinal sphere  of $B$ and the boundary of $H$ is a 2-disk with boundary consisting of $[u, B(u)]$, $[B(u), B^2(u)]$ and $[B^2(u),u]$. Now this 2-disk is divided into two 2-disks by $[B(u), p_{AB}]$, $[p_{B^{-1}A},u]$ and $[p_{B^{-1}A},p_{AB}]$,  one of them is a quadrilateral, we denote it by $B_4$ in Figure \ref{figure:3pppolytope}, another one is a
pentagon and we denote it by $B_5$ in Figure \ref{figure:3pppolytope}.
The intersection of  the spinal sphere  of $B^{-1}$ and the boundary of $H$ is a 2-disk with boundary consisting of $[u, B(u)]$, $[B(u), B^2(u)]$ and $[B^2(u),u]$. This 2-disk is divided into two 2-disks by $[B(u), p_{BA^{-1}}]$, $[p_{BA}, B^2(u)]$ and $[p_{BA},p_{BA^{-1}}]$,  one of them is a quadrilateral, we denote it by $B^{-1}_4$ in Figure \ref{figure:3pppolytope}, another one is a
pentagon and we denote it by $B^{-1}_5$ in Figure \ref{figure:3pppolytope}.

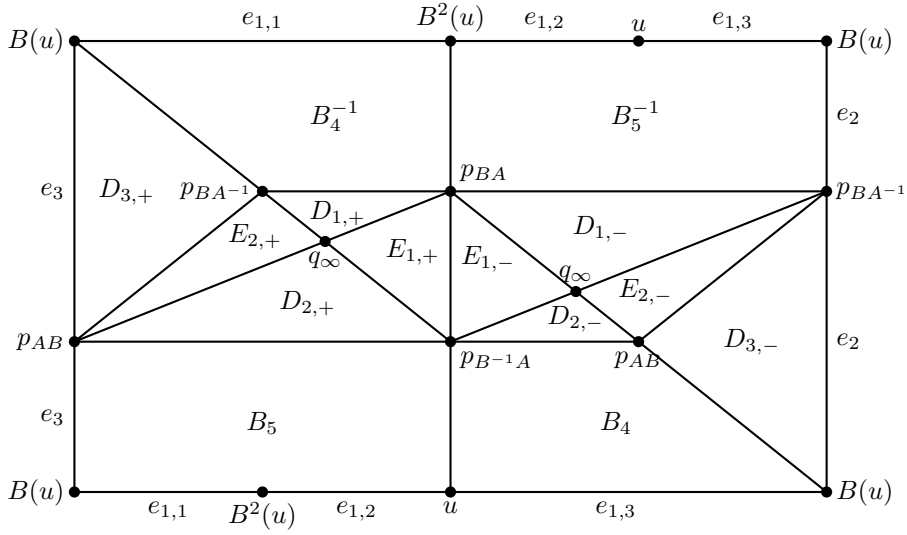
\begin{figure}
\begin{center}
\begin{tikzpicture}[style=thick,scale=1]

\draw (-5,0) -- (5,0);
\draw (-5,-6) -- (5,-6);
\draw (-5,-6) -- (-5,0);
\draw (5,-6) -- (5,0);
\draw[fill] (-5,0) circle [radius=0.06];
\draw[fill] (5,0) circle [radius=0.06];
\draw[fill] (-5,-6) circle [radius=0.06];
\draw[fill] (5,-6) circle [radius=0.06];
\node [ left] at (-5,0) {$B(u)$};
\node [ right] at (5,0) {$B(u)$};
\node [ left] at (-5,-6) {$B(u)$};
\node [ right] at (5,-6) {$B(u)$};

\draw (0,-6) -- (0,0);
\draw (-2.5,-2) -- (5,-2);
\draw (2.5,-4) -- (-5,-4);
\draw (-5,0) -- (0,-4);
\draw (0,-2) -- (5,-6);
\draw (-5,-4) -- (-2.5,-2);
\draw (-5,-4) -- (0,-2);
\draw (0,-4) -- (5,-2);
\draw (2.5,-4) -- (5,-2);
\node [ above] at (0,0) {$B^2(u)$};\draw[fill] (0,0) circle [radius=0.06];
\node [ above right] at (0,-2) {$p_{BA}$};\draw[fill] (0,-2) circle [radius=0.06];
\node [ below right] at (0,-4) {$p_{B^{-1}A}$};\draw[fill] (0,-4) circle [radius=0.06];
\node [ below ] at (0,-6) {$u$};\draw[fill] (0,-6) circle [radius=0.06];
\node [ below ] at (-2.5,-6) {$B^2(u)$};\draw[fill] (-2.5,-6) circle [radius=0.06];
\node [ above] at (2.5,0) {$u$};\draw[fill] (2.5,0) circle [radius=0.06];
\node [ left] at (-5,-4) {$p_{AB}$};\draw[fill] (-5,-4) circle [radius=0.06];
\node [ right] at (5,-2) {$p_{BA^{-1}}$};\draw[fill] (5,-2) circle [radius=0.06];
\node [ left] at (-2.5,-2) {$p_{BA^{-1}}$};\draw[fill] (-2.5,-2) circle [radius=0.06];
\node [ below] at (2.5,-4) {$p_{AB}$};\draw[fill] (2.5,-4) circle [radius=0.06];
\node [ below] at (-5/3,-8/3) {$q_{\infty}$};\draw[fill] (-5/3,-8/3) circle [radius=0.06];
\node [ above] at (5/3,-10/3) {$q_{\infty}$};\draw[fill] (5/3,-10/3) circle [radius=0.06];

\node [ above] at (-2.5,0) {$e_{1,1}$};
\node [ above] at (1.25,0) {$e_{1,2}$};
\node [ above] at (3.75,0) {$e_{1,3}$};
\node [ right] at (-2,-1) {$B^{-1}_{4}$};
\node [ right] at (2,-1) {$B^{-1}_{5}$};
\node [ right] at (5,-1) {$e_{2}$};
\node [ left] at (-5,-2) {$e_{3}$};
\node [ left] at (-5,-5) {$e_{3}$};
\node [ left] at (-3.8,-2) {$D_{3, +}$};
\node [below] at (-2.6,-2.3) {$E_{2, +}$};
\node [ below] at (-1.5,-2) {$D_{1, +}$};
\node [below] at (-0.5,-2.5) {$E_{1, +}$};
\node [below] at (0.5,-2.6) {$E_{1, -}$};
\node [below] at (2,-2.2) {$D_{1, -}$};
\node [right] at (2.1,-10/3)  {$E_{2, -}$};
\node [ below] at (5/3,-3.4) {$D_{2, -}$};
\node [ below] at (-1.9,-3.2){$D_{2, +}$};
\node [ below] at (-2.5,-4.8){$B_{5}$};
\node [ below] at (2.2,-4.8){$B_{4}$};
\node [ below] at (2.2,-6){$e_{1,3}$};
\node [ right] at (3.5,-4){$D_{3, -}$};
\node [ right] at (5,-4){$e_{2}$};
\node [ below] at (-1.25,-6){$e_{1,2}$};
\node [ below] at (-3.75,-6){$e_{1,1}$};

\end{tikzpicture}
\end{center}
\caption{The boundary of the polytope for  $\Delta_{3,\infty,\infty;\infty}$.}
\label{figure:3pppolytope}
\end{figure}

Note that the 2-sphere in Figure  \ref{figure:3pppolytope} lies in $\partial \mathbf{H}^2_{\mathbb C}=\mathbb{S}^3$ and  every piecewise smooth 2-sphere in $\mathbb{S}^3$ separates $\mathbb{S}^3$ into two 3-balls.  $N$ is topologically a 3-ball which is one of them. In fact, $N$ is a polytope with the facet structure in Figure  \ref{figure:3pppolytope}.

Note that
$$\partial_{\infty} D_{\Gamma_1} =\cup^{\infty} _{k= -\infty} A^{k}(H).$$
 The reader can  compare the following proposition with Proposition \ref{prop:unknotted-34p} and  Lemma  \ref{34pcombinatorialmodel} further.

\begin{prop}\label{prop:ideal-boundary}
	$\partial_{\infty} D_{\Gamma_1}$ is an unknotted infinite genus handlebody preserved by the action of $A$.
\end{prop}
\begin{proof}  By Proposition \ref{prop:intersection3pp}, $s_k^{+}\cup s_k^{-}$ are topological spheres. Moreover, the spheres $s_k^{+}\cup s_k^{-}$ and  $A(s_k^{+}\cup s_k^{-})$ are placed in contact with each other on two points from Proposition \ref{prop:tangent3pp}. From this, we conclude that the union of the sets  $\{A^k(s_k^{+}\cup s_k^{-}): k\in \mathbb{Z}\}$ is an infinite genus handlebody. It is preserved by the cyclic group $\langle  A \rangle$.  See Figure \ref{figure:3ppabstract}.
	
Note that the $A$ preserves the $\mathbb{R}$-circle
	\begin{equation*}
	\{[x-2i/\sqrt{3},0]: x\in \mathbb{R}\}.
	\end{equation*}
	
Let $\tau$ be the segment
\begin{equation*}
	\tau=\{[x-2i/\sqrt{3},0]: x\in [-1/2\sqrt{3},1/2\sqrt{3}]\},
	\end{equation*} which in the above $\mathbb{R}$-circle. It is easy to check that $\tau$ belongs to the interior of  $I_0^{-}$.  Furthermore,  it  belongs to the topological
	sphere  $s_0^{+}\cup s_0^{-}$.  So $A^k(\tau)\subset A^k(s_k^{+}\cup s_k^{-})$ and $\partial D_{\Gamma_1}$ can not be knotted.
\end{proof}

If we  denote the 3-manifold at infinity of the even subgroup $\Gamma_1=\langle I_1I_2, I_2I_3 \rangle$ of  $\Delta_{3,\infty,\infty;\infty}$ by $M$,  $M$ is just the quotient space of this 3-ball $N$, then the side-pairings are
\begin{eqnarray*}
	\alpha_1&: &E_{1,-}\longrightarrow E_{2,+},\\
	\alpha_2&:& E_{2,-}\longrightarrow E_{1,+},\\
	\delta_1&:& D_{1,-}\longrightarrow D_{1,+},\\
	\delta_2&: &D_{2,-}\longrightarrow D_{2,+},\\
	\delta_3&: &D_{3,-}\longrightarrow D_{3,+},\\
	\beta_4&: &B_{4}\longrightarrow B^{-1}_{4},\\
	\beta_5&: &B_{5}\longrightarrow B^{-1}_{5}.
 \end{eqnarray*}

 The side-pairing $\alpha_1$ is the homeomorphism  from  $E_{1,-}$ to $E_{2,+}$ such that $\alpha_1(p_{BA})=p_{AB}$,  $\alpha_1(p_{B^{-1}A})=w$ and $\alpha_1(q_{\infty})=q_{\infty}$;  $\alpha_2$ is the homeomorphism  from $E_{2,-}$ to  $E_{1,+}$ such that $\alpha_2(p_{AB})=p_{BA}$,  $\alpha_2(p_{BA^{-1}})=p_{B^{-1}A}$ and $\alpha_2(q_{\infty})=q_{\infty}$;  $\delta_
{i}$ is the homeomorphism  from  $D_{i,-}$ to $D_{i,+}$ which  preserves the labels on vertices for $i=1,2,3$; $\beta_4$  is the homeomorphism  from $B_{4}$ to $B^{-1}_{4}$  such that  $\beta_4(p_{AB})=p_{BA}$,  $\beta_4(p_{B^{-1}A})=p_{BA^{-1}}$,  $\beta_4(u)=B(u)$ and  $\beta_4(B(u))=B^2(u)$; $\beta_5$  is the homeomorphism  from $B_{5}$ to $B^{-1}_{5}$  such that  $\beta_5(p_{AB})=p_{BA}$,  $\beta_5(p_{B^{-1}A})=p_{BA^{-1}}$,  $\beta_5(u)=B(u)$, $\beta_5(B^2(u))=u$ and  $\beta_5(B(u))=B^2(u)$.

We write down the  ridge cycles for the 3-manifold $M$  in Table \ref{table:edgecircle3pp}.

\begin{table}[!htbp]
\caption{Ridge cycles of the 3-manifold at infinity  of $\Delta_{3,\infty,\infty;\infty}$.}
  \centering
\begin{tabular}{c|c}
\toprule
\textbf{Ridge} & \textbf{Ridge Cycle   } \\
\midrule
$e_{1}$ & $B_4\cap B_5^{-1}(e_{1,3})\xrightarrow{\beta_4} B^{-1}_4\cap B_5(e_{1,1}) \xrightarrow{\beta_5} B^{-1}_5\cap B_5(e_{1,2}) \xrightarrow{\beta_5} B^{-1}_5\cap B_4(e_{1,3})$  \\

$e_{2}$ & $B_4\cap B_5\xrightarrow{\beta_5} B^{-1}_5\cap D_{3,-} \xrightarrow{\delta_3} D_{3,+}\cap B^{-1}_4 \xrightarrow{\beta_4^{-1}} B_4\cap B_5$\\

$e_{3}$ & $B^{-1}_4\cap B^{-1}_5\xrightarrow{\beta_5^{-1}} B_5\cap D_{3,+} \xrightarrow{\delta_3^{-1}} D_{3,-}\cap B_4 \xrightarrow{\beta_4} B^{-1}_4\cap B^{-1}_5$  \\
$e_{4}$ & $D_{1,+}\cap E_{1,+}\xrightarrow{\alpha_2^{-1}} E_{2,-}\cap D_{2,-} \xrightarrow{\delta_2} D_{2,+}\cap E_{2,+} \xrightarrow{\alpha_1^{-1}} E_{1,-}\cap D_{1,-} \xrightarrow{\delta_1} D_{1,+}\cap E_{1,+}$ \\

$e_{5}$ & $E_{2,+}\cap D_{1,+}\xrightarrow{\delta_1^{-1}} D_{1,-}\cap E_{2,-} \xrightarrow{\alpha_2} E_{1,+}\cap D_{2,+} \xrightarrow{\delta_2^{-1}} D_{2,-}\cap E_{1,-} \xrightarrow{\alpha_1} E_{2,+}\cap D_{1,+}$  \\

$e_{6}$ & $E_{1,+}\cap E_{1,-}\xrightarrow{\alpha_1} E_{2,+}\cap D_{3,+} \xrightarrow{\delta_3^{-1}}  D_{3,-}\cap E_{2,-} \xrightarrow{\alpha_2} E_{1,+}\cap E_{1,-}$  \\	

$e_{7}$ & $B^{-1}_4\cap D_{1,+}\xrightarrow{\delta_1^{-1}} D_{1,-}\cap B^{-1}_5 \xrightarrow{\beta_5^{-1}} B_5\cap D_{2,+} \xrightarrow{\delta_2^{-1}} D_{2,-}\cap B_4 \xrightarrow{\beta_4} B^{-1}_4 \cap D_{1,+}$ \\
\bottomrule
\end{tabular}
\label{table:edgecircle3pp}
\end{table}

\subsection{The 3-manifold at infinity of $\Delta_{3,\infty, \infty;\infty}$} \label{subsection:3ppmanifoldfinal}
From Table	\ref{table:edgecircle3pp}, we get Table  \ref{table:relation3pp}, then we get a presentation of the fundamental group of  the 3-manifold $M$. It is a group $\pi_{1}(M)$ on seven generators $\alpha_1$, $\alpha_2$, $\delta_1$, $\delta_2$, $\delta_3$, $\beta_4$, $\beta_5$ and seven relations in Table \ref{table:relation3pp}.

\begin{table}[!htbp]
\caption{Cycle relation  of the 3-manifold at infinity of $\Delta_{3,\infty,\infty;\infty}$.}
  \centering
\begin{tabular}{c|c}
\toprule
\textbf{Ridge} & \textbf{Cycle relation  } \\
\midrule
$e_{1}$  & $\beta^2_5\beta_4$ \\

$e_{2}$  & $\beta^{-1}_4 \delta_3\beta_5$ \\

$e_{3}$  & $\beta_4 \delta^{-1}_3\beta^{-1}_5$ \\

$e_{4}$  & $ \delta_1 \alpha^{-1}_1 \delta_2 \alpha^{-1}_2$ \\

$e_{5}$  & $\alpha_1  \delta^{-1}_2 \alpha_2  \delta^{-1}_1$ \\

$e_{6}$  & $\alpha_2 \delta^{-1}_3\alpha_1$ \\

$e_{7}$ & $\beta_4 \delta^{-1}_2\beta^{-1}_5 \delta^{-1}_1$ \\
\bottomrule
\end{tabular}
\label{table:relation3pp}
\end{table}

Let $6_1^3$ be the magic 3-manifold   in Snappy  Census \cite{CullerDunfield:2014}, it is hyperbolic with volume 5.3334895669... and $$\pi_{1}(6_1^3)=\langle s,t,w | s^2w^{-1}sts^{-2}ws^{-1}t^{-1},tw^{-1}t^{-1}w\rangle.$$

\texttt{Magma} tells us that  $\pi_{1}(6_1^3)$ and  $\pi_{1}(M)$  above are isomorphic and finds an  isomorphism $\Psi_1: \pi_1(6_1^3)\longrightarrow \pi_1(M)$ given by
$$
\Psi_1(s)=\beta_5^{-1}, \quad \Psi_1(t)=\beta_4\alpha_2^{-1},\quad \Psi_1(w)=\delta_1.
$$

So, by the prime decompositions of 3-manifolds \cite{Hempel}, $M$ is the connected sum of $6_1^3$  with $L$, where
$L$ is a closed 3-manifold with trivial fundamental group. By the solution of the
Poincar\'e Conjecture, then  $L$ is the 3-sphere, so $M$ is homeomorphic to  $6_1^3$. This finishes the proof of Theorem \ref{thm:3pp}.

\section{The Ford domain for the even subgroup of $\Delta_{3,4,\infty;\infty}$}\label{sec-ford-34p}

Let $\Gamma  =\langle I_1, I_2, I_3\rangle$ be the group  $\Delta_{3,4,\infty;\infty}$ in Subsection \ref{subsec-34p}.
We consider the even  subgroup  $\Gamma_2=\langle A, B\rangle$ of $\langle I_1, I_2, I_3\rangle$ with index two, where $A=I_1I_3I_2I_3$, $B=I_2I_3$, then $B^3=id$ and  $(A^{-1}B)^4=id$.
In this section,
we will  study the combinatorics of the side of  the Ford domain for  $\Gamma_2$. The main result is Theorem \ref{thm:fundamental-domain34p}.

 The matrices for $A$, $B$ are
\begin{equation*}\label{eq-I1-I2}
	A=\left[\begin{matrix}
		1 & \frac{3-\sqrt{7}i}{2}& -2+\sqrt{7} i\\ 0 & 1 &  \frac{-3-\sqrt{7}i}{2}\\ 0 & 0 & 1 \end{matrix}\right], \quad
	B=\left[\begin{matrix}
		1&\frac{1+\sqrt{7}i}{2} &-1\\ \frac{-1+\sqrt{7}i}{2} & -1& 0\\ -1&0 & 0 \end{matrix}\right].
\end{equation*}

Observe that $\Gamma_2=\langle A, B\rangle$ is a discrete subgroup of  $\mathbf{PU}(2,1)$, since $A$, $B$ are contained in the Euclidean Picard modular group $\mathbf{PU}(2,1;\mathbb{Z}[(-1+\sqrt{7}i)/2])$.  For more details of Euclidean Picard modular group we refer the reader to \cite{zhao}.

\begin{defn}For $k\in \mathbb{Z}$, let
\begin{itemize}
\item  $I_{k}^{+}$ be the isometric sphere $I(A^kBA^{-k})=A^k(I(B))$,
\item  $I_{k}^{-}$ be the isometric sphere $I(A^kB^{-1}A^{-k})=A^k(I(B^{-1}))$,
\item  $\widehat{I}_{k}$ be the isometric sphere $I(A^kB^{-1}AB^{-1}A^{-k})=A^k(I(B^{-1}AB^{-1}))$.
\end{itemize}
\end{defn}

Moreover, since $ B^{-1}AB^{-1}=A^{-1}BA^{-1}BA^{-1}$,
$ B^{-1}AB^{-1}$ and $ABA^{-1}BA^{-1}$ have the same isometric  sphere. We summarize some information for these isometric spheres in  Table \ref{tab:isometric spheres} .
\begin{table}[!htbp]
\caption{The centers and radii of the isometric spheres.}
  \centering
\begin{tabular}{c|c|c}
\toprule
\textbf{Isometric sphere} & \textbf{Center  } & \textbf{Radius}\\
\midrule
$I_{k}^{+}$&$[(-3k-\sqrt{7}ki)/2,2\sqrt{7}k]$&$\sqrt{2}$\\
$I_{k}^{-}$&$[(-3k+1-\sqrt{7}(k+1)i)/2,0]$&$\sqrt{2}$\\
$\widehat{I}_{k}$&$[(-3k-1-(k+1)\sqrt{7}i)/2,(k+1)\sqrt{7}]$&$1$\\
\bottomrule
\end{tabular}
\label{tab:isometric spheres}
\end{table}

\begin{prop}\label{B-pairwise-intersection}
\begin{enumerate}
\item The isometric sphere $I_{0}^{+}$ is contained in the exteriors of the isometric spheres $I_{k}^{\pm}$ and $\widehat{I}_{k}$ for all $k\in \Z$, except for $I_{1}^{+}$, $I_{-1}^{+}$, $I_{0}^{-}$, $I_{-1}^{-}$, $\widehat{I}_{0}$ and $\widehat{I}_{-1}$.
 \item The isometric sphere $I_{0}^{-}$ is contained in the exteriors of the isometric spheres $I_{k}^{\pm}$ and $\widehat{I}_{k}$ for all $k\in \Z$, except for $I_{0}^{+}$, $I_{1}^{+}$, $I_{-1}^{-}$, $I_{1}^{-}$, $\widehat{I}_{0}$ and $\widehat{I}_{-1}$.
  \item The isometric sphere $\widehat{I}_{0}$ is contained in the exteriors of the isometric spheres $I_{k}^{\pm}$ and $\widehat{I}_{k}$ for all $k\in \Z$, except for  $I_{0}^{+}$, $I_{1}^{+}$, $I_{0}^{-}$ and $I_{1}^{-}$.
\end{enumerate}
\end{prop}
\begin{proof} (1) All the isometric spheres in  $\{I_{k}^{+}; k\in \Z\}$ have radius $\sqrt{2}$.  So two isometric spheres in  $\{I_{k}^{+}; k\in \Z\}$ are disjoint if their centers are a Cygan distance at least $2\sqrt{2}$  apart. The center of $I_{k}^{+}$ is $[(-3k-\sqrt{7}ki)/2,2\sqrt{7}k],$ see Table \ref{tab:isometric spheres}.  So the Cygan distance between the centers of $I_{0}^{+}$ and $I_{k}^{+}$ is
\begin{equation*}
d_{\rm Cyg}\bigl([0,0],[(-3k-\sqrt{7}ki)/2,2\sqrt{7}k]\bigr)^4=16k^4+28k^2.
\end{equation*}
 This number is larger than 64 when $|k|\geq 2$.

 From Table \ref{tab:isometric spheres}, $\widehat{I}_{k}$ is a Cygan sphere with center $$[(-3k-1-(k+1)\sqrt{7}i)/2,(k+1)\sqrt{7}]$$ and radius 1.
The Cygan distance between the centres of $I_{0}^{+}$ and $\widehat{I}_{k}$ is
 \begin{eqnarray*}
&&d_{\rm Cyg}\bigl([0,0],[(-3k-1-(k+1)\sqrt{7}i)/2,(k+1)\sqrt{7}]\bigr)^4\\&=&\frac{1}{16}((3k+1)^2+7(k+1)^2)^2+7(k+1)^2.
\end{eqnarray*}
  When $k>0$ or $k<-1$, this number is larger than $(1+\sqrt{2})^4$.

 The items (2) and (3)  can be proved by the same argument as above.

\end{proof}

We take two points $p_{AB}$  and $p_{B^{-1}AB^{-1}}$  in  $\partial \hc$ with  homogeneous coordinates

\begin{equation*}
\mathbf{p}_{AB} = \left[
                               \begin{array}{c}
                                 -1\\
                                -1/2 -i\sqrt{7}/2 \\
                                 1 \\
                               \end{array}
                             \right], \quad
 \mathbf{ p}_{B^{-1}AB^{-1}} = \left[
                               \begin{array}{c}
                                -1/4 +i\sqrt{7}/4\\
                                1/4 -i\sqrt{7}/4 \\
                                 1 \\
                               \end{array}
                             \right].
 \end{equation*}

\begin{prop}\label{B-pairwise-tangent}
The isometric spheres $I_{0}^{+}$ and $I_{1}^{-}$ are tangent at $p_{AB}$. The isometric spheres $\widehat{I}_{0}$ and $\widehat{I}_{1}$ are tangent at $A(p_{B^{-1}AB^{-1}})$.
\end{prop}
\begin{proof}
The isometric spheres $I_{0}^{+}$ and $I_{1}^{-}$  are Cygan spheres with centres $[0,0]$ and $[-1-\sqrt{7}i,0]$. Both isometric spheres have radius $\sqrt{2}$. Since
$$d_{\rm Cyg}\bigl([0,0], p_{AB}\bigr)=d_{\rm Cyg}\bigl([-1-\sqrt{7}i,0],p_{AB}\bigr)=\sqrt{2},$$ we see that $p_{AB}$ belongs to both $I_{0}^{+}$ and $I_{1}^{-}$. The vertical projections of $I_{0}^{+}$ and $I_{1}^{-}$  are tangent discs, see Figure   \ref{figure:projE-34p} later. Therefore the intersection of $I_{0}^{+}$ and $I_{1}^{-}$ contains at most one point as the Cygan sphere is strictly convex. The other tangency can be proved in a similar way.

\end{proof}

\begin{prop}\label{empty-ridge1}
\begin{enumerate}	

\item The intersection $I_{0}^{+}\cap I_{1}^{+}\cap I_{0}^{-} $ is empty. Moreover, the intersection $I_{0}^{+}\cap I_{1}^{+}$ is contained in the interior of $I_{0}^{-}$.

\item The intersection $I_{0}^{+}\cap I_{-1}^{+}\cap I_{-1}^{-} $ is empty. The intersection $I_{0}^{+}\cap I_{-1}^{+}$ is contained in the interior of the isometric sphere $I_{-1}^{-}$.

\item The intersection $I_{0}^{+}\cap I_{0}^{-}\cap \widehat{I}_{0} $ is empty. The intersection $I_{0}^{+}\cap \widehat{I}_{0}$ is contained in the interior of the isometric sphere $I_{0}^{-}$.

\item The intersection $\widehat{I}_{0}\cap I_{1}^{-}\cap I_{1}^{+} $ is empty. The intersection $\widehat{I}_{0}\cap I_{1}^{-}$ is contained in the interior of the isometric sphere $I_{1}^{+}$.

\item The intersection $\widehat{I}_{0}\cap I_{0}^{+}\cap I_{0}^{-} $ is empty. The intersection $\widehat{I}_{0}\cap I_{0}^{+}$ is contained in the interior of the isometric sphere $I_{0}^{-}$.
\end{enumerate}
\end{prop}

\begin{proof}
 The intersection $I_{0}^{+}\cap I_{1}^{+}$ can be parameterized by negative vectors of the
form
$$V(z_1,z_2)=\left(\overline{z_1}\mathbf{q}_{\infty}-B^{-1}(\mathbf{q}_{\infty})\right)\boxtimes \left(\overline{z_2}\mathbf{q}_{\infty}-AB^{-1}A^{-1}(\mathbf{q}_{\infty})\right),$$ where $|z_1|=|z_2|=1$.
One can check that $I_{0}^{+}\cap I_{1}^{+}$ is a smooth disk according to Proposition  \ref{prop:Giraud}, since $\langle V(e^{-\frac{\pi}{4}i},e^{\frac{\pi}{4}i}), V(e^{-\frac{\pi}{4}i},e^{\frac{\pi}{4}i})\rangle=13-2\sqrt{14}-4\sqrt{2}\approx -0.140169<0$.

We claim that the intersection $I_{0}^{+}\cap I_{1}^{+}\cap I_{0}^{-} $ is empty.
The equation of the intersection $I_{0}^{+}\cap I_{1}^{+}\cap I_{0}^{-} $ is given by
 \begin{equation}\label{equation:1-I0+-I1+-I0-}
 |\langle V(z_1,z_2), \mathbf{q}_{\infty}\rangle|^{2} = |\langle V(z_1,z_2), B(\mathbf{q}_{\infty})\rangle|^{2}.
\end{equation}

Suppose that $\theta_1, \theta_2\in [0,2\pi]$. We define the function
$$f_1(\theta_1,\theta_2)= \sqrt{7}(\sin(\theta_2-\theta_1)-\sin(\theta_2)+2\sin(\theta_1))-\cos(\theta_1-\theta_2)-3\cos(\theta_2)-2\cos(\theta_1)+5.$$
 Writing $z_1=e^{i\theta_1},z_2=e^{i\theta_2}$, Equation (\ref{equation:1-I0+-I1+-I0-}) is equivalent to $f_1(\theta_1,\theta_2)=0$.

By using the computer algebra system, we find that  the global maximum of  the function $f_1(\theta_1,\theta_2)$ with the  constraint
\begin{eqnarray*}
 &&\langle V(e^{i\theta_1},e^{i\theta_2}),V(e^{i\theta_1},e^{i\theta_2})\rangle\\
 &=&13-4\cos(\theta_1)+2\sqrt{7}\sin(\theta_1)-2\sqrt{7}\sin(\theta_2)-2\cos(\theta_1-\theta_2)-4\cos(\theta_2)<0
\end{eqnarray*}
is given approximately by  $-0.6899$.

By a simple calculation, we have
$$|\langle V(e^{-\frac{\pi}{4}i},e^{\frac{\pi}{4}i}), B(\mathbf{q}_{\infty})\rangle|\approx 0.9977<|\langle V(e^{-\frac{\pi}{4}i},e^{\frac{\pi}{4}i}), \mathbf{q}_{\infty}\rangle|=2.$$
Therefore, $V(e^{-\frac{\pi}{4}i},e^{\frac{\pi}{4}i})\in I_{0}^{+}\cap I_{1}^{+}$ lies inside the isometric sphere  $I_{0}^{-}$. See Figure \ref{fig:ridges-B-and-ABa}. Since the Giraud disk $I_{0}^{+}\cap I_{1}^{+}$ is connect, we get that $I_{0}^{+}\cap I_{1}^{+}$ lies inside $I_{0}^{-}$.
The completes the proof of the item (1).  The rest of the proof runs as before.

\end{proof}

By applying powers of $A$, all pairwise intersections of the isometric spheres can be summarized in the following corollary.
\begin{cor}\label{cor:intersections}
 Let $\mathcal{F}=\{I_k^{\pm}, \widehat{I}_k: k \in \mathbb{Z} \}$ be the set of all the isometric spheres. Then for all $k\in\mathbb{Z}$:
\begin{enumerate}
  \item The isometric sphere $I_k^{+}$ is contained in the exterior of all the isometric  spheres in $\mathcal{F}$, except for $I_{k+1}^{+}$, $I_{-k-1}^{+}$, $I_{k}^{-}$, $I_{-k-1}^{-}$, $\widehat{I}_{k}$ and $\widehat{I}_{-k-1}$. Moreover, $I_{k}^{+}\cap I_{k+1}^{-}$ (resp. $I_{k}^{+}\cap I_{k-1}^{+}$) is contained in the interior of $I_{k}^{-}$(resp. $I_{k-1}^{-}$). $I_{k}^{+}$ is tangent to $I_{k+1}^{-}$ at the parabolic fixed point $A^{k}(p_{AB})$.

  \item The isometric sphere $I_k^{-}$ is contained in the exterior of all the isometric spheres in $\mathcal{F}$, except for $I_{k}^{+}$, $I_{k+1}^{+}$, $I_{-k-1}^{-}$, $I_{k+1}^{-}$, $\widehat{I}_{k}$ and $\widehat{I}_{-k-1}$.

   \item The isometric sphere $\widehat{I}_k$ is contained in the exterior of all the isometric spheres in $\mathcal{F}$, except for $I_{k}^{+}$, $I_{k+1}^{+}$, $I_{k}^{-}$ and $I_{k+1}^{-}$. Moreover, $\widehat{I}_{k}\cap I_{k+1}^{-}$ (resp. $\widehat{I}_{k}\cap I_{k}^{+}$) is contained in the interior of $I_{k}^{+}$(resp. $I_{k
       }^{-}$). $\widehat{I}_{k}^{+}$ is tangent to $\widehat{I}_{k+1}^{+}$ at the parabolic fixed point $A^{k+1}(p_{B^{-1}AB^{-1}})$.
\end{enumerate}
\end{cor}

\begin{defn} Let $D_{\Gamma_2}$ be the intersection of the closures of the exteriors of all the isometric spheres  $I_{k}^{+}, I_{k}^{-}$  and $\widehat{I}_{k}$ for $k\in \mathbb{Z}$.
\end{defn}

 $D_{\Gamma_2}$ above is our guessed Ford domain for the group  $\Gamma_2$.

\begin{defn} For $k\in \mathbb{Z}$, let $s_k^{+}, s_k^{-}, \widehat{s}_k$ be the sides of  $D_{\Gamma_2}$ contained in the isometric spheres $I_{k}^{+}, I_{k}^{-}, \widehat{I}_{k}$ respectively.
\end{defn}

\begin{defn} A  \textit{ridge} is defined to be the $2$-dimensional connected intersections of two sides.
\end{defn}

We now describe the topology of the ridges of $D_{\Gamma_2}$.  First, we describe the topology of the ridges of the side $s_0^{+}$. Then $s_0^{-}$ has the same kind of combinatorial structure of   $s_0^{+}$, since it is paired by $B$.
\begin{prop}\label{Prop:ridgegiraud}
The ridge $s_{0}^{+}\cap s_{0}^{-}$ is a Giraud disk, which is contained in the exterior of the isometric spheres $I_{1}^{+}, I_{-1}^{+},I_{-1}^{-},\widehat{I}_{-1}$ and $\widehat{I}_{0}$.
\end{prop}

\begin{proof}
 The Giraud disk $I_{0}^{+}\cap I_{0}^{-}$ can be parameterized by negative vectors of the
form
$$V(z_1,z_2)=\left(\overline{z_1}\mathbf{q}_{\infty}-B^{-1}(\mathbf{q}_{\infty})\right)\boxtimes \left(\overline{z_2}\mathbf{q}_{\infty}-B(\mathbf{q}_{\infty})\right),$$ where $|z_1|=|z_2|=1$.
We claim that  $I_{0}^{+}\cap I_{0}^{-}$ does not intersect  the isometric spheres $I_{1}^{+}, I_{-1}^{+},I_{-1}^{-},\widehat{I}_{-1}$ and $\widehat{I}_{0}$. This was already proved for $I_{1}^{+}$. One can show that the intersection with the other isometric spheres are  empty as well by using the same technology in the above Proposition \ref{empty-ridge1}.

Since $\langle V(e^{\frac{\pi}{4}i},e^{\frac{\pi}{4}i}), V(e^{\frac{\pi}{4}i},e^{\frac{\pi}{4}i})\rangle=1-2\sqrt{2}<0$, we can choose $V(e^{\frac{\pi}{4}i},e^{\frac{\pi}{4}i})$ as a sample point in the disk $I_{0}^{+}\cap I_{0}^{-}$.

From some simple calculation, we get
\begin{eqnarray*}
&&|\langle V(e^{\frac{\pi}{4}i},e^{\frac{\pi}{4}i}), \mathbf{q}_{\infty}\rangle|^2=4,\\
&&|\langle V(e^{\frac{\pi}{4}i},e^{\frac{\pi}{4}i}), AB^{-1}(\mathbf{q}_{\infty})\rangle|^2=10+\sqrt{14}+3\sqrt{2}>4,\\
&&|\langle V(e^{\frac{\pi}{4}i},e^{\frac{\pi}{4}i}), A^{-1}B^{-1}(\mathbf{q}_{\infty})\rangle|^2=18+5\sqrt{2}-\sqrt{14}>4,\\
&&|\langle V(e^{\frac{\pi}{4}i},e^{\frac{\pi}{4}i}), A^{-1}B(\mathbf{q}_{\infty})\rangle|^2=10+\sqrt{14}+3\sqrt{2}>4,\\
&&|\langle V(e^{\frac{\pi}{4}i},e^{\frac{\pi}{4}i}), A^{-1}BA^{-1}B(\mathbf{q}_{\infty})\rangle|^2=12+2\sqrt{14}+2\sqrt{2}>4,\\
&&|\langle V(e^{\frac{\pi}{4}i},e^{\frac{\pi}{4}i}), BA^{-1}B(\mathbf{q}_{\infty})\rangle|^2=12+2\sqrt{14}+2\sqrt{2}>4.
\end{eqnarray*}
Therefore, $V(e^{\frac{\pi}{4}i},e^{\frac{\pi}{4}i})$ lies outside the isometric spheres  $I_{1}^{+}, I_{-1}^{+},I_{-1}^{-},\widehat{I}_{-1}$ and $\widehat{I}_{0}$. The ridge $s_{0}^{+}\cap s_{0}^{-}$ will be the Giraud disk $I_{0}^{+}\cap I_{0}^{-}$. See Figure \ref{fig:ridge-B-and-b}.

\end{proof}

\begin{figure}[htbp]
	\centering
	\begin{minipage}[t]{0.48\textwidth}
		\centering
		\includegraphics[width=5cm]{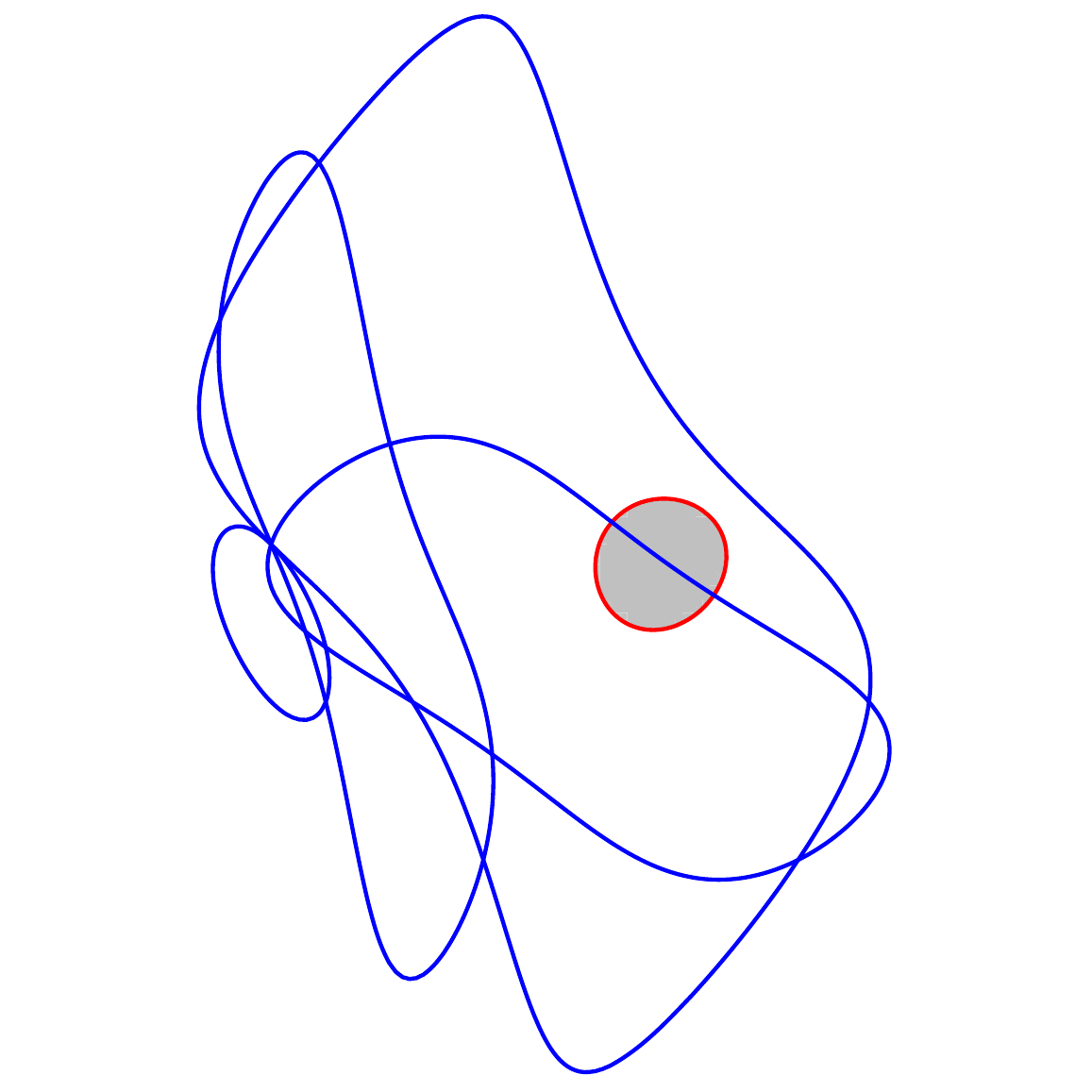}
		\caption{The shaded disk is the intersection $I_{0}^{+}\cap I_{1}^{+}$, which is contained in the interior of $I_{0}^{-}$.}
		\label{fig:ridges-B-and-ABa}
	\end{minipage}
	\begin{minipage}[t]{0.48\textwidth}
		\centering
		\includegraphics[width=5cm]{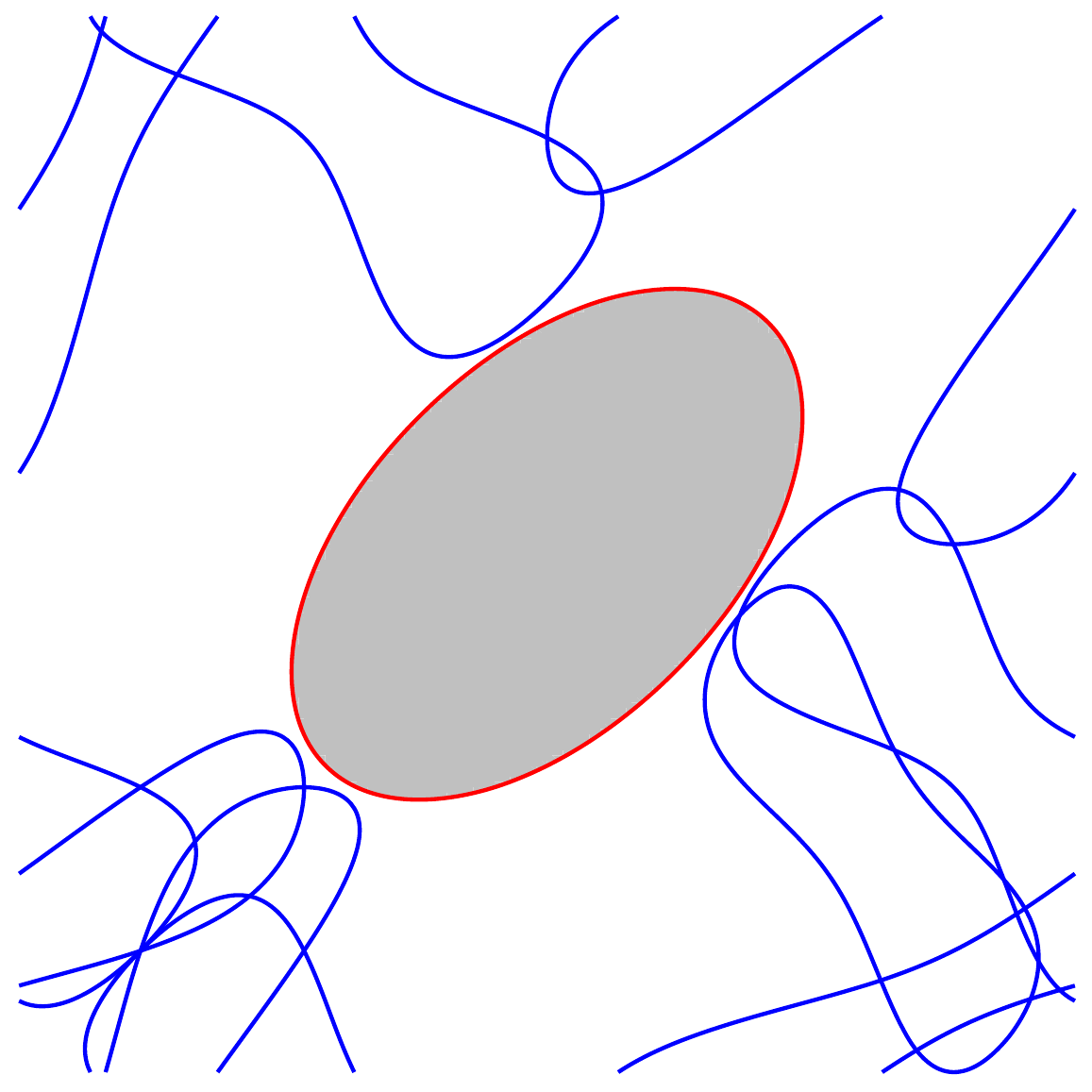}
		\caption{The shaded region is the ridge $I_{0}^{+}\cap I_{0}^{-}$.}
		\label{fig:ridge-B-and-b}
	\end{minipage}
\end{figure}

\begin{prop}\label{ridgesectors34p}
The ridges $s_{0}^{+}\cap s_{-1}^{-}$ and $s_{0}^{+}\cap \widehat{s}_{-1}$ are topologically the union of two sectors.
\end{prop}

\begin{proof}
We parametrize the Giraud disk $I_{0}^{+}\cap I_{-1}^{-}$ by negative vectors of the form
\begin{eqnarray*}
 V(z_1,z_2)&=&\left(\overline{z_1}\mathbf{q}_{\infty}-B^{-1}(\mathbf{q}_{\infty})\right)\boxtimes \left(\overline{z_2}\mathbf{q}_{\infty}-A^{-1}B(\mathbf{q}_{\infty})\right)\\
 &=& \left[
                                \begin{array}{c}
                                 2z_1\\
                                -z_1+z_2-2 \\
                                 -2 \\
                               \end{array}
                             \right],
                             \end{eqnarray*}
 where $|z_1|=|z_2|=1$.

 As in the above proposition, we can show that  $I_{0}^{+}\cap I_{-1}^{-}$ lies outside the isometric spheres  $I_{1}^{+}, I_{-1}^{+},I_{0}^{-}$ and $\widehat{I}_{0}$. So we only
 need to consider the intersection  of   $I_{0}^{+}\cap I_{-1}^{-}$ and the isometric sphere $\widehat{I}_{-1}$. Their intersection can be described by the following equation
 \begin{equation}\label{equation:intersection34p}
 |\langle V(z_1,z_2),\mathbf{q}_{\infty}\rangle|=|\langle V(z_1,z_2),A^{-1}BA^{-1}B(\mathbf{q}_{\infty})\rangle|.
 \end{equation}

 Writing $z_1=e^{i\theta_1},z_2=e^{i\theta_2}$, Equation (\ref{equation:intersection34p}) is equivalent to
   \begin{equation}\label{equaytion:intersection34preformulation}
   -1-\cos(\theta_1-\theta_2)+\cos(\theta_1)+\cos(\theta_2)=0.
  \end{equation}
The three  solutions of  Equation (\ref{equaytion:intersection34preformulation}) are
 $$\theta_1=0,\quad \theta_2=0,\quad\theta_2=\theta_1-\pi.$$
When $\theta_2=\theta_1-\pi$, we have
\begin{eqnarray*}
 &&\langle V(z_1,z_2),V(z_1,z_2)\rangle\\
 &=&6-2\sin(\theta_1)\sin(\theta_2)-2\cos(\theta_1)\cos(\theta_2)-4\cos(\theta_1)-4\cos(\theta_2)\\
 &=& 8.
 \end{eqnarray*}
Therefore, the points  $V(z_1,z_2)$ corresponding to the straight segment  $$\{\theta_2=\theta_1-\pi, \theta_1\in [-\pi,\pi]\}$$ lie entirely outside complex hyperbolic space.

We will see that the two crossed straight segments
\begin{eqnarray*}
&&\{\theta_1=0,\theta_2\in [-\arccos(1/3), \arccos(1/3)]\},\\
&& \{\theta_2=0,\theta_1\in [-\arccos(1/3), \arccos(1/3)]\}
\end{eqnarray*}
decompose  the Giraud disk $I_{0}^{+}\cap I_{-1}^{-}$ into four sectors with common vertices. Furthermore, one can check that one  pair of the four sectors lie inside the isometric  sphere $\widehat{I}_{-1}$. Therefore, the  ridge $s_{0}^{+}\cap s_{-1}^{-}$ is the other pair of the four sectors.  See Figure \ref{fig:ridges-B-and-abA}.

A similar analysis can be applied to the ridge $s_{0}^{+}\cap \widehat{s}_{-1}$. See Figure \ref{fig:ridges-B-and-abAbA}.
\end{proof}

\begin{figure}[htbp]
	\centering
	\begin{minipage}[t]{0.48\textwidth}
		\centering
		\includegraphics[width=5cm]{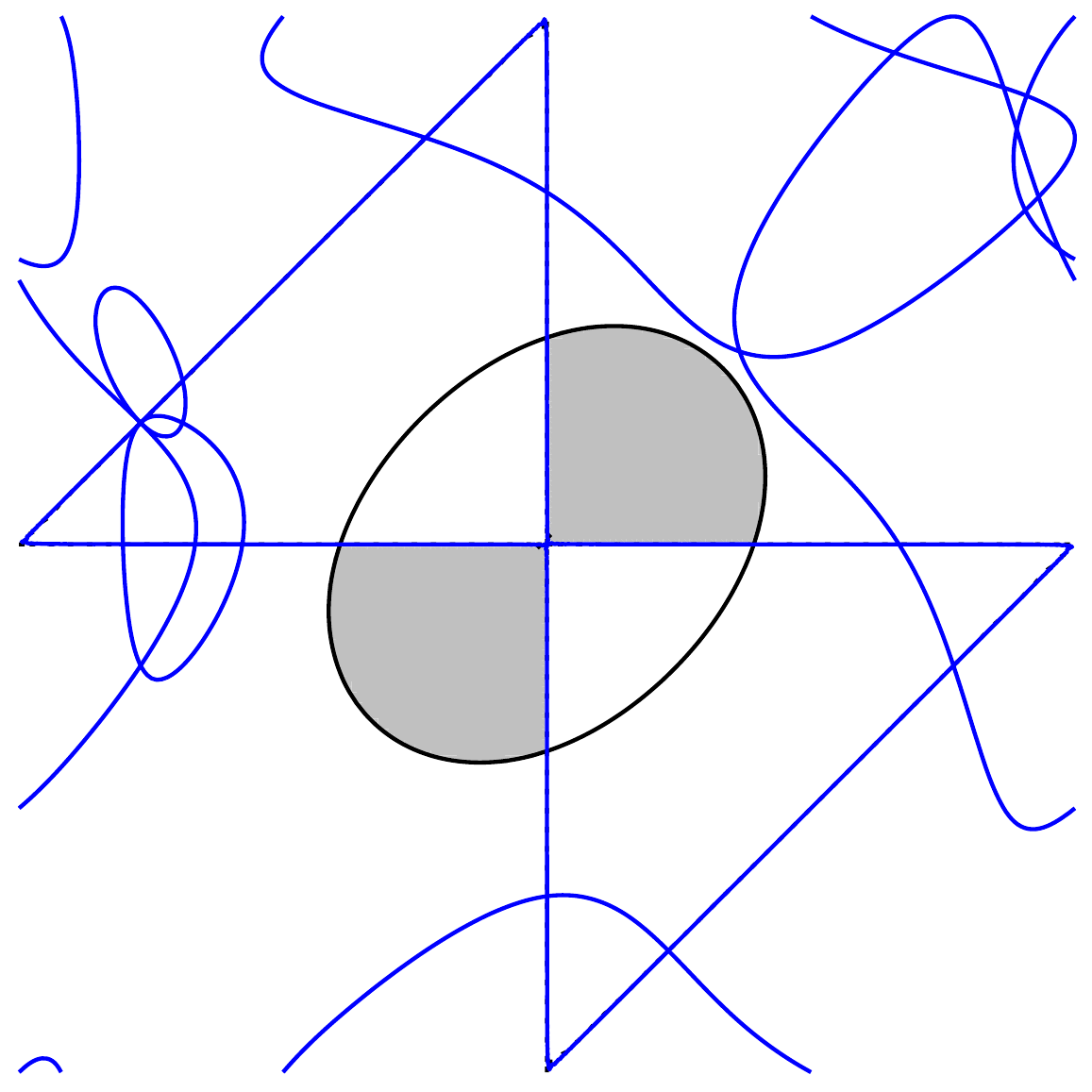}
		\caption{The ridge $s_{0}^{+}\cap s_{-1}^{-}$ on the Giraud disk $I_{0}^{+}\cap I_{-1}^{-}$ is the pair of shaded sectors. }
		\label{fig:ridges-B-and-abA}
	\end{minipage}
	\begin{minipage}[t]{0.48\textwidth}
		\centering
		\includegraphics[width=5cm]{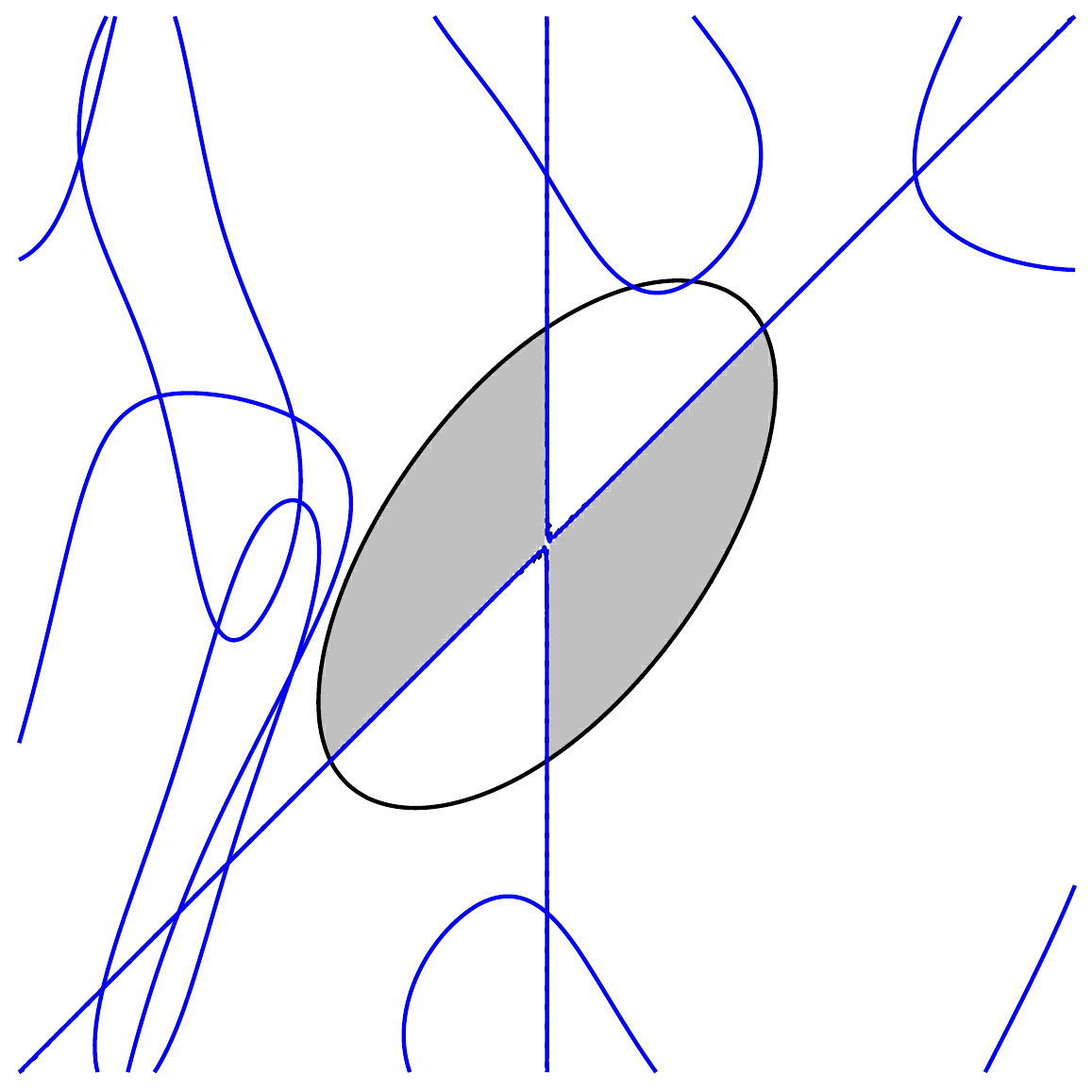}
		\caption{The ridge $s_{0}^{+}\cap \widehat{s}_{-1}$ on the Giraud disk $I_{0}^{+}\cap \widehat{I}_{-1}$ is the pair of shaded sectors.}
		\label{fig:ridges-B-and-abAbA}
	\end{minipage}
\end{figure}

\begin{prop}\label{empty-ridge6}
The ridges $\widehat{s}_{0}\cap s_{0}^{-}$ and $\widehat{s}_{0}\cap s_{1}^{+}$ are topologically the union of two sectors.
\end{prop}

 One can also use the similar argument in Proposition \ref{ridgesectors34p} to show Proposition \ref{empty-ridge6}.  We omit the details for the proof. Please refer to the Figures \ref{fig:ridges-bAb-and-b}, \ref{fig:ridges-bAb-and-ABb}.

\begin{figure}[htbp]
\centering
\begin{minipage}[t]{0.48\textwidth}
\centering
\includegraphics[width=5cm]{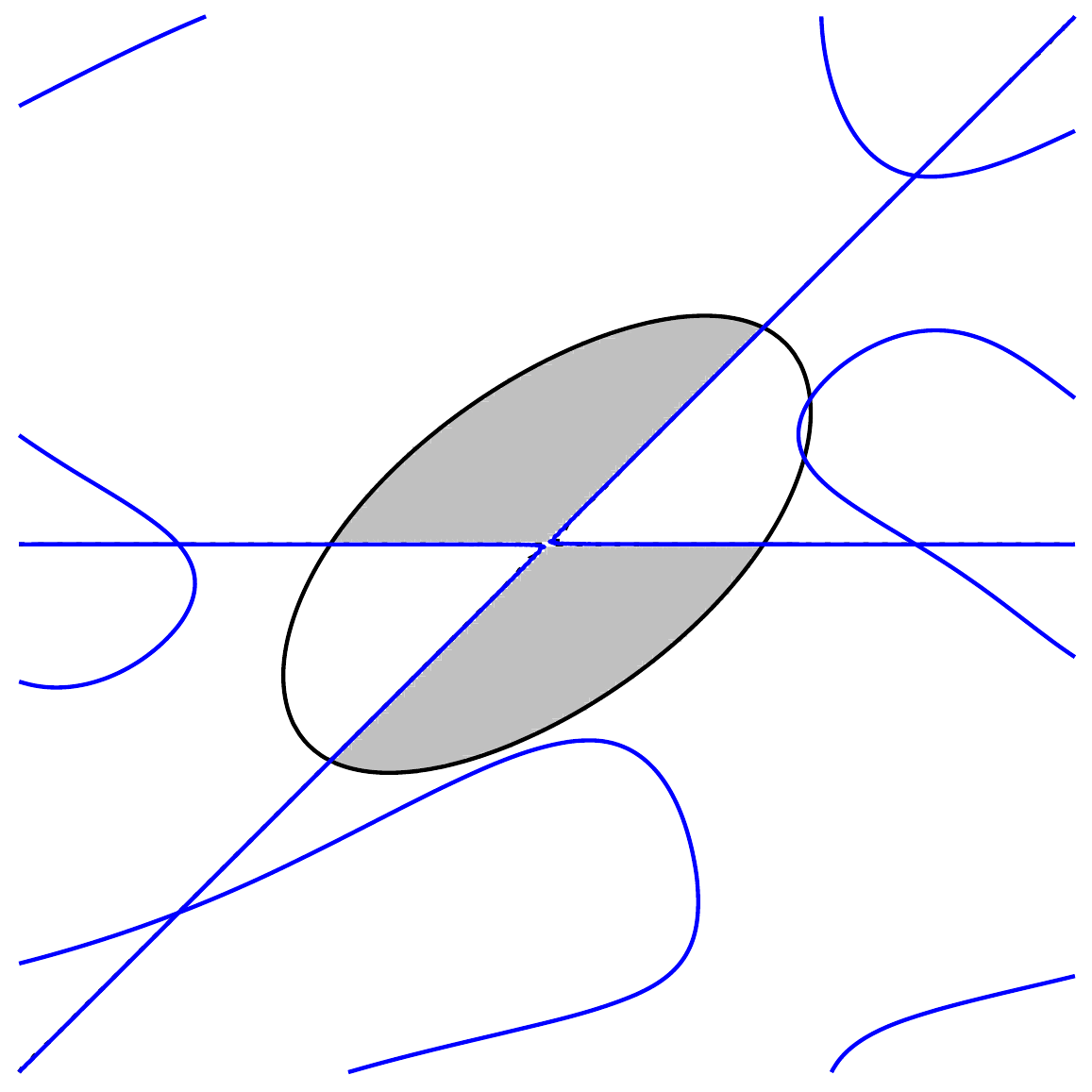}
\caption{ The ridge $\widehat{s}_{0}\cap s_{0}^{-}$ on the Giraud disk $\widehat{I}_{0}\cap I_{0}^{-}$ is the pair of shaded sectors. }
\label{fig:ridges-bAb-and-b}
\end{minipage}
\begin{minipage}[t]{0.48\textwidth}
\centering
\includegraphics[width=5cm]{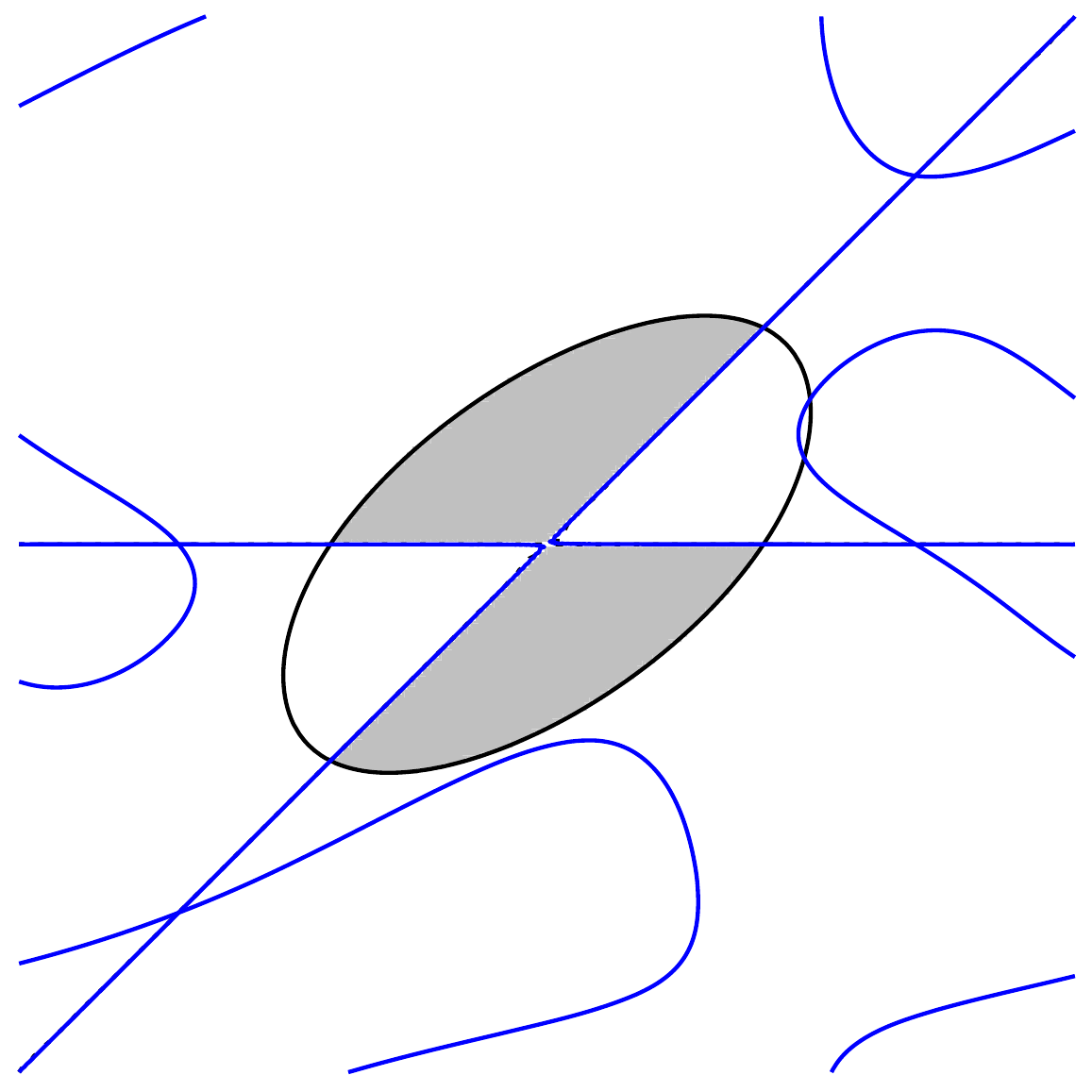}
\caption{The ridge $\widehat{s}_{0}\cap s_{1}^{+}$ on the Giraud disk $\widehat{I}_{0}\cap I_{1}^{+}$ is the pair of shaded sectors. }
\label{fig:ridges-bAb-and-ABb}
\end{minipage}
\end{figure}

Next, we  describe the combinatorics of the sides.

\begin{prop}\label{ridgecylinder-lightcone34p}
\begin{enumerate}
\item The side $s_{k}^{+}$(resp. $s_{k}^{-}$) is a topological solid cylinder in $\hc\cup\partial\hc$. The intersection of $\partial s_{k}^{+}$(resp. $\partial s_{k}^{-}$) with
$\hc$ is the disjoint union of two disks.
\item The side $\widehat{s}_{k}$ is a topological solid light cone in $\hc\cup\partial\hc$. The intersection of $\partial \widehat{s}_{k}$ with
$\hc$ is a  light cone.
\end{enumerate}
\end{prop}

\begin{proof}
The side $s_{k}^{+}$ is contained in the isometric sphere $I_{k}^{+}$. It is bounded by the ridges  $s_k^{+} \cap s_{k-1}^{-}$, $s_k^{+} \cap \widehat{s}_{k-1}$ and $s_k^{+} \cap s_{k}^{-}$. The ridges  $s_k^{+} \cap s_{k-1}^{-}$ and $s_k^{+} \cap \widehat{s}_{k-1}$ are pair of sectors on the isometric sphere $I_{k}^{+}$. Furthermore, $I_{k}^{+}\cap I_{k}^{-}\cap \widehat{s}_{k-1}$ is the union of two crossed segments. Therefore, the union of the ridges $\widehat{s}_{k-1}\cap s_{k}^{+}$ and $s_{k-1}^{-}\cap s_{k}^{+}$ is a topological disc. Proposition \ref{Prop:ridgegiraud} tell us that the smooth disc $s_k^{+} \cap s_{k}^{-}$ does not intersect with any other isometric sphere. This implies that $s_{k}^{+}$ is a topological solid cylinder. See Figure \ref{fig:bigside}.

The side $\widehat{s}_{k}$ is bounded by two  ridges   $\widehat{s}_k \cap s_k^{-}$ and $\widehat{s}_k \cap s_{k+1}^{+}$. This two ridges  are two pairs of sectors with common boundary which is the union of two crossed segments $\widehat{I}_k\cap I_k^{-} \cap I_{k+1}^{+}$.  From the gluing process of two ridges along their boundaries, we see that $\widehat{s}_{k}$ is topological a solid light cone. See Figure \ref{fig:smallside}.
\end{proof}
\begin{figure}
	\begin{center}
		\begin{tikzpicture}[scale=1]
		\coordinate (O) at (0,0) ;
		\draw[black,dashed,fill=gray] (O) ellipse (0.3 and 1) ;
		\coordinate (O1) at (3,0) ;
		\draw[black,fill=red] (O1) ellipse (0.3 and 1) ;
		\draw (0,1) --(3,1);
		\draw (0,-1) --(3,-1);
		\draw[black] (0,0) +(90:1) arc (90:270:0.3 and 1) ;
		\draw (3,-1) --(3,1);
		\draw (2.7,0) --(3.3,0);
		\path [fill=blue] (3,0) --(3,1)--(3,0) +(90:1) arc (90:180:0.3 and 1) -- (2.3,0)-- (3,0);
		\path [fill=blue] (3,0) --(3,-1)--(3,0) +(270:1) arc (270:360:0.3 and 1) -- (3.3,0)-- (3,0);
		\draw[->] (3.1,0.2) -- (4,1);
		\draw[->] (3.1,-0.2) -- (4,-1);
		\coordinate [label=right:$s_k^{+}\cap s_k^{-}$] (r1) at (-1.7,0);
		\coordinate [label=right:$s_k^{+}\cap s_{k-1}^{-}$] (r1) at (4.1,1);
		\coordinate [label=right:$s_k^{+}\cap \widehat{s}_{k-1}$] (r1) at (4.1,-1);
		\end{tikzpicture}
	\end{center}
	\caption{A schematic view of the side $s_k^{+}$.}
	\label{fig:bigside}
\end{figure}
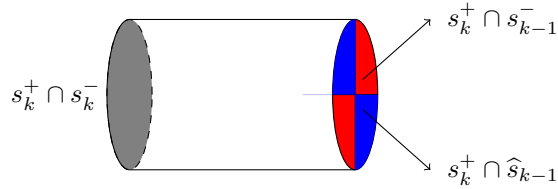

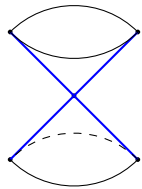
\begin{figure}
	\begin{center}
		\begin{tikzpicture}[scale=0.6]
		\fill (45:2) circle (1.5pt) (45*3:2) circle (1.5pt) (45*5:2) circle (1.5pt) (45*7:2) circle (1.5pt);
		\fill (0,0) circle (1.5pt);
		\draw[thick,blue] (0,0) -- (45:2);
		\draw[thick,blue] (0,0) -- (45*3:2);
		\draw[thick,blue] (0,0) -- (45*5:2);
		\draw[thick,blue] (0,0) -- (45*7:2);
		\draw (45:2) arc (45*1:45*3:2);
		\draw (45*3:2) arc (45*5:45*7:2);
		\draw (45*5:2) arc (45*5:45*7:2);
		\draw[dashed] (45*7:2) arc (45*1:45*3:2);
		\end{tikzpicture}
	\end{center}
	\caption{A schematic view of the side $\widehat{s}_k$.}
	\label{fig:smallside}
\end{figure}

Applying a version of the   Poincar\'{e} polyhedron theorem for coset decompositions, we show that $D_{\Gamma_2}$  will be the Ford domain for the group $\Gamma_2$ generated by $A$ and $B$.
\begin{thm}\label{thm:fundamental-domain34p}
$D_{\Gamma_2}$ is a fundamental domain for the cosets of $\langle A \rangle$ in $\Gamma_2$. Moreover, the group $\Gamma_2=\langle A, B \rangle$  is discrete and  has the presentation
\begin{equation*}
\langle  A, B: B^3=(A^{-1}B)^4=id\rangle.
\end{equation*}
\end{thm}

The side pairing maps are defined by:
\begin{eqnarray*}
 A^{k}BA^{-k}:& s_{k}^{+} \longrightarrow s_{k}^{-}, \\
 A^{k}(B^{-1}AB^{-1})A^{-k} : & \widehat{s}_{k} \longrightarrow \widehat{s}_{k-1}.
\end{eqnarray*}

Note that $D_{\Gamma_2}$ is invariant under a group $\langle A\rangle$ that is compatible with the side pairing maps.

By Corollary \ref{cor:intersections}, the ridges are $s_{k}^{+} \cap s_{k-1}^{-}$, $s_{k}^{+} \cap \widehat{s}_{k-1}$, $s_{k}^{+} \cap s_{k}^{-}$, $s_{k}^{-} \cap \widehat{s}_{k}$  and $s_{k}^{-} \cap s_{k+1}^{+}$  for $k\in\mathbb{Z}$, and the sides and ridges are related as follows:
\begin{itemize}
\item The side $s_k^{+}$ is bounded by the ridges $s_k^{+} \cap s_{k-1}^{-}$, $s_k^{+} \cap \widehat{s}_{k-1}$ and $s_k^{+} \cap s_{k}^{-}$.
\item The side $s_k^{-}$ is bounded by the ridges $s_k^{-} \cap \widehat{s}_k$, $s_k^{-} \cap s_{k+1}^{+}$  and $s_k^{+} \cap s_{k}^{-}$.
\item The side $\widehat{s}_k$ is bounded by the ridges $\widehat{s}_k \cap s_k^{-}$ and $\widehat{s}_k \cap s_{k+1}^{+}$.
\end{itemize}
\begin{prop}\label{prop:sidepairing34p}
\begin{enumerate}
\item[(1)] The  side pairing $A^{k}BA^{-k}$ is a homeomorphism from  $s_{k}^{+}$ to $s_{k}^{-}$. Moreover, it sends the ridges $s_k^{+} \cap s_{k-1}^{-}$, $s_k^{+} \cap \widehat{s}_{k-1}$ and $s_k^{+} \cap s_{k}^{-}$ to $s_k^{-} \cap \widehat{s}_{k}$, $s_k^{-} \cap s_{k+1}^{+}$ and $s_k^{+} \cap s_{k}^{-}$ respectively.
\item[(2)]  The  side pairing $ A^{k}(B^{-1}AB^{-1})A^{-k}$ is a homeomorphism from  $\widehat{s}_{k}$ to $\widehat{s}_{k-1}$. Moreover, it sends the ridges $\widehat{s}_k \cap s_{k}^{-}$ and $\widehat{s}_k \cap s_{k+1}^{+}$ to $s_k^{+} \cap \widehat{s}_{k-1}$ and $s_{k-1}^{-} \cap \widehat{s}_{k-1}
    $ respectively.
\end{enumerate}
\end{prop}

\begin{proof}

We only need to consider the case $k=0$.
$s_0^{+}$ is contained in the isometric sphere $I_0^{+}$ and $s_0^{-}$ in the isometric sphere $I_0^{-}$. The ridge $s_0^{+} \cap s_0^{-}$ is the whole smooth disc $I_0^{-} \cap I_0^{+}$, which is defined by the triple equality
$$
|\langle {\bf z}, {\bf{q}_{\infty}} \rangle|=|\langle {\bf z}, B^{-1}({\bf{q}_{\infty}}) \rangle|=|\langle {\bf z}, B({\bf{q}_{\infty}}) \rangle|.
$$
The elliptic element $B$ of order three sends this ridge to itself.

The ridge $s_0^{+} \cap s_{-1}^{-}$ is a pair of sectors contained in the disc $I_0^{+}  \cap I_{-1}^{-}$, which is defined by the triple equality
$$
|\langle {\bf z}, {\bf{q}_{\infty}} \rangle|=|\langle {\bf z}, B^{-1}({\bf{q}_{\infty}}) \rangle|=|\langle {\bf z}, A^{-1}B({\bf{q}_{\infty}}) \rangle|.
$$

$B$ maps $q_{\infty}$ to $B(q_{\infty})$, $B^{-1}(q_{\infty})$ to $q_{\infty}$ and $A^{-1}B(q_{\infty})$ to $BA^{-1}B(q_{\infty})$. That is,
$B$ maps the disc $I_0^{+}  \cap I_{-1}^{-}$ to the disc $I_0^{-1}  \cap \widehat{I}_{0}$.
Note that the other pair of sectors different from the ridge on the disc $I_0^{+}  \cap I_{-1}^{-}$  lie in  the interior of the isometric sphere $\widehat{I}_{-1}$,  and the other pair of sectors different from the ridge on the disc $I_0^{-}  \cap\widehat{I}_{0}$  lie in  the interior of the isometric sphere $I_1^{+}$. One can check that the point $p^{\star}$ with  homogeneous coordinates 
\begin{eqnarray*}
 p^{\star}=\left[
	\begin{array}{c}
		-\frac{\sqrt{2}}{2}-\frac{\sqrt{2}}{2}i\\
	1 \\
		1 \\
	\end{array}
	\right]\\
\end{eqnarray*}
 is contained in the ridge $s_0^{+} \cap s_{-1}^{-}$.   It is easily verified that $B(p^{\star})$ lies in the exterior of $I_1^{+}$. Thus $B(p^{\star})$ is contained in the ridge $s_0^{-1}  \cap \widehat{s}_{0}$, lies in  the exterior of the isometric sphere $I_1^{+}$.
Hence $B$ maps the ridge $s_0^{+} \cap s_{-1}^{-}$ to the ridge $s_0^{-1}  \cap \widehat{s}_{0}$. Similarly, the rest  can be proved by using the same arguments.
\end{proof}

According to the side-paring maps, the ridge cycles are:
\begin{eqnarray*}
&(s_k^{+} \cap s_{k-1}^{-}, s_k^{+} , s_{k-1}^{-}) \xrightarrow{A^{k}BA^{-k}}
(s_k^{-} \cap \widehat{s}_k, \widehat{s}_k , s_k^{-})  \xrightarrow{A^{k}(B^{-1}AB^{-1})A^{-k}} \\
&(\widehat{s}_{k-1} \cap s_k^{+}, s_k^{+} , \widehat{s}_{k-1})  \xrightarrow{A^{k}BA^{-k}}
(s_k^{-} \cap s_{k+1}^{+}, s_{k+1}^{+} , s_k^{-}) \xrightarrow{A^{-1}} (s_k^{+} \cap s_{k-1}^{-}, s_k^{+} , s_{k-1}^{-}),
\end{eqnarray*}
and
\begin{eqnarray*}
  &(s_k^{+} \cap s_{k}^{-}, s_k^{+} , s_{k}^{-})  \xrightarrow{A^{k}BA^{-k}}
( s_{k}^{-} \cap s_k^{+} , s_{k}^{-} , s_k^{+} ).
\end{eqnarray*}

 Every ridge cycle is equivalent to one of the cycles listed above.

Note that $B^{-1}AB^{-1}=A^{-1}BA^{-1}BA^{-1}$. Thus the cycle transformations are
$$
A^{-1}\cdot A^{k}BA^{-k}\cdot A^{k}(A^{-1}BA^{-1}BA^{-1})A^{-k}\cdot A^{k}BA^{-k} =A^{k}(A^{-1}B)^{4}A^{-k}=id.
$$
and
$$
(A^{k}BA^{-k})^3=A^{k}B^{3}A^{-k},
$$
which are equal to the identity map, since we have $B^3=(A^{-1}B)^{4}=id$.

After Propositions  \ref{ridgecylinder-lightcone34p} and  \ref{prop:sidepairing34p}, we now show

\paragraph{{\bf Local tessellation}.}

In our case, the intersection of each pair of  two bisectors is a smooth disk. It follows from Proposition \ref{prop:Giraud} that the ridges of $D_{\Gamma_2}$ are on precisely three
isometric spheres, hence there are three copies of $D_{\Gamma_2}$ tiling its neighborhood. Therefore the ridge cycles satisfy the local tessellation
conditions of the Poincar\'{e} polyhedron theorem.

\paragraph{{\bf Completeness}}

We must construct a system of consistent horoballs at the parabolic fixed points.
First, we consider the side pairing maps on the parabolic fixed points in Proposition \ref{prop:sidepairing34p}.
We have
\begin{eqnarray*}
  B &:&   p_{AB} \longrightarrow B(p_{AB}),\\
   AB^{-1} A^{-1}&:&   p_{AB} \longrightarrow A(p_{AB}).
  \end{eqnarray*}

  Up to powers of $A$,  the cycles for the parabolic fixed points are the following
$$
p_{AB} \xrightarrow{AB^{-1}A^{-1}} A(p_{AB})\xrightarrow{A^{-1}} p_{AB}.
$$
That is,  $p_{AB}$ is fixed by $B^{-1}A^{-1}$.
The element $B^{-1}A^{-1}$ is parabolic and so preserves all horoballs at $p_{AB}$.

This ends the proof of  Theorem \ref{thm:fundamental-domain34p}.

\section{Manifold at infinity of the complex hyperbolic triangle group  $\Delta_{3,4, \infty;\infty}$}\label{section:3-mfd-34p}

Based on the results in Section \ref{sec-ford-34p}, in particular, the combinatorial structures of $I^{\pm}_{k} \cap \overline {\bf H}^2_{\mathbb{C}}$ and  $\widehat{I}_{k} \cap \overline {\bf H}^2_{\mathbb{C}}$, we study the manifold  at infinity of the even subgroup $\Gamma_2$ of the complex hyperbolic triangle group  $\Delta_{3,4,\infty;\infty}$ in this section.
We restate Theorem \ref{thm:34p} as:

\begin{thm} The manifold at infinity of $\Gamma_2$ is homeomorphic to the hyperbolic 3-manifold $m295$ in the Snappy  Census.
\end{thm}

The ideal boundary of  $D_{\Gamma_2}$ is made up of the ideal boundaries of the  sides $s_k^{\pm}$ and $\widehat{s}_k$.  For convenience,  when we speak of sides and ridges we implicitly mean their intersections with $\partial \hc$ in this section.

We first take four points in Heiserberg coordinates,
\begin{eqnarray*}
y&=&\left[\frac{4}{3}+\frac{\sqrt{2}}{3}i,0\right], \\
c&=&\left[\frac{4}{3}-\frac{\sqrt{2}}{3}i,0\right], \\
g&=&\left[\frac{2}{3}+\frac{\sqrt{2}}{3}i,-\frac{4\sqrt{2}}{3}\right], \\
r&=&\left[\frac{2}{3}-\frac{\sqrt{2}}{3}i,\frac{4\sqrt{2}}{3}\right]. \\
\end{eqnarray*}
In Figure \ref{figure:34pabstract}, the point  $y$ (resp. $g$, $c$, $r$) is indicated by a small yellow (resp. green, cyan, red) dot.

By simply calculation, we have
\begin{eqnarray*}
	A(y)=B(g)&=& \left[-\frac{1}{6}+\frac{2\sqrt{2}-3\sqrt{7}}{6}i, \frac{2\sqrt{7}}{3}+\sqrt{2}\right],\\
	A(g)=B(c)&=&\left[-\frac{5}{6}+\frac{2\sqrt{2}-3\sqrt{7}}{6}i, \frac{4\sqrt{7}-\sqrt{2}}{3}\right],\\
	A(c)=B(r)&=&\left[-\frac{1}{6}-\frac{2\sqrt{2}+3\sqrt{7}}{6}i, \frac{2\sqrt{7}}{3}-\sqrt{2}\right],\\
	A(r)=B(y)&=&\left[-\frac{5}{6}-\frac{2\sqrt{2}+3\sqrt{7}}{6}i, \frac{4\sqrt{7}+\sqrt{2}}{3}\right].
\end{eqnarray*}

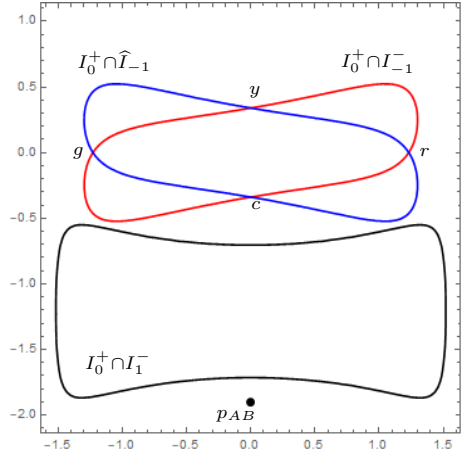
\begin{figure}
\begin{center}
\begin{tikzpicture}
\node at (0,0) {\includegraphics[width=6cm,height=6cm]{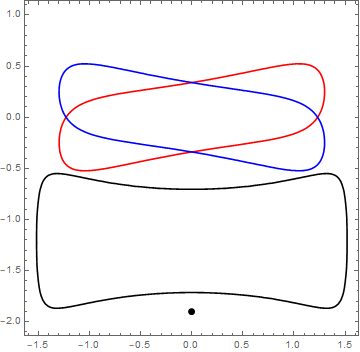}};
\coordinate [label=left:$\scriptstyle{I_{0}^{+}\cap I_{-1}^{-}}$] (S) at (2.5,2.2);
\coordinate [label=left:$\scriptstyle{I_{0}^{+}\cap \widehat{I}_{-1}}$] (S) at (-1, 2.2);
\coordinate [label=left:$\scriptstyle{I_{0}^{+}\cap I_{1}^{-}}$] (S) at (-1, -1.8);
\coordinate [label=left:$\scriptstyle{y}$] (S) at (0.45, 1.8);
\coordinate [label=left:$\scriptstyle{c}$] (S) at (0.45, 0.3);
\coordinate [label=left:$\scriptstyle{p_{AB}}$] (S) at (0.4, -2.5);
\coordinate [label=left:$\scriptstyle{g}$] (S) at (-1.9, 1);
\coordinate [label=left:$\scriptstyle{r}$] (S) at (2.7, 1);
\end{tikzpicture}
\end{center}
  \caption{The ideal side $s_{0}^{+}$ is the region exterior to the three Jordan closed curves $I_{0}^{+}\cap I_{-1}^{-}$, $I_{0}^{+}\cap \widehat{I}_{-1}$ and $I_{0}^{+}\cap I_{1}^{-}$
on the spinal sphere of $B$. These Jordan curves are drawn with geographical coordinates introduced in \cite{ParkerWill:2016}.}
	\label{figure:34p0}
\end{figure}

\begin{figure}[htbp]
	\centering
	\begin{minipage}[t]{0.48\textwidth}
		\centering
		\begin{tikzpicture}
		\node at (0,0) {\includegraphics[width=5cm]{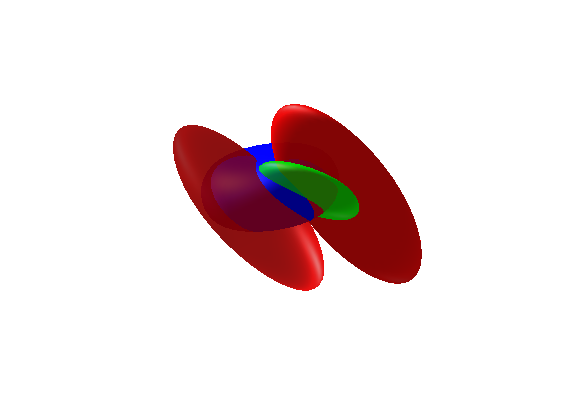}};
		\coordinate [label=left:$I_{0}^{-}$] (S) at (-0.6,0.9);
		\coordinate [label=left:$I_{-1}^{-}$] (S) at (0.6,1);
		\coordinate [label=above:$I_{0}^{+}$] (S) at (-0.3,0.4);
		\draw[->] (0.5,-0.15)--(0.5,-0.65);
		\coordinate [label=below:$\widehat{I}_{-1}$] (S) at (0.5,-0.65);
		\end{tikzpicture}
		\caption{The side $s_{0}^{+}$ on the blue spinal sphere $I_{0}^{+}$ is topologically an annulus.}
		\label{figure:side-B}
	\end{minipage}
	\begin{minipage}[t]{0.48\textwidth}
		\centering
		\begin{tikzpicture}
		\node at (0,0) {\includegraphics[width=5cm]{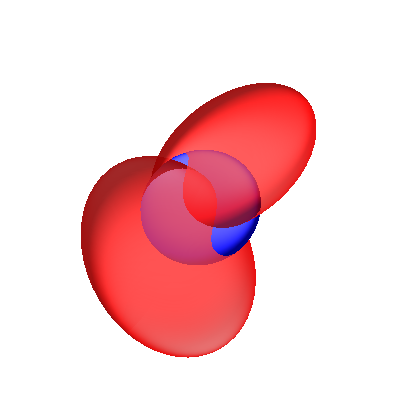}};
		\coordinate [label=left:$I_{0}^{-}$] (S) at (1,-1.8);
		\draw[->] (0.7,-0.35)--(0.9,-0.35);
		\coordinate [label=right:$\widehat{I}_{0}$] (S) at (0.9,-0.35);
		\coordinate [label=right:$I_{1}^{+}$] (S) at (1.4,1);
		\end{tikzpicture}
		\caption{The side $\widehat{s}_{0}$ on the blue spinal sphere with two disjoint connect components.}
		\label{figure:side-bAb}
	\end{minipage}
\end{figure}

Now we study the combinatorial properties of the sides. See Figures 	\ref{figure:34p0}, \ref{figure:side-B},  \ref{figure:side-bAb} and 	\ref{figure:34pabstract}.

\begin{prop}\label{prop:side-plus}
	\begin{itemize}
		\item The interior of the side $s_0^{+}$ is an annulus.  One boundary component of it is the Jordan curve $s_0^{+}\cap s_0^{-}$ and the other boundary component is a quadrilateral with vertices $y,g,c,r$.
		\item The interior of the side $\widehat{s}_0$ has two disjoint connected components. Each component is a bigon. One bigon with vertices $A(y),A(g)$ and the other bigon with vertices $A(c),A(r)$.
	\end{itemize}
\end{prop}
\begin{proof}
	(1). The traces of  the  Giraud disks $I_0\cap I_{-1}^{-}$ and   $I_0\cap \widehat{I}_{-1}$ on the spinal sphere $I_0$ are two   closed Jordan  curves. By simple calculation, we see that there are four intersection points (labeled by  $y$, $g$, $c$  and $r$) of this two closed curves. From the structure of the ridges, we known that the   trace of  the  ridge $s_0\cap s_{-1}^{-}$ (resp. $s_0\cap \widehat{s}_{-1}$) is a union of two Jordan arcs $[y,g]$ and $[c,r]$ (resp. $[g,c]$ and $[r,y]$).
	The trace of  the ridge  $s_0\cap s_{-1}$ on  the spinal sphere $I_0$ is also a  closed Jordan  curve. Note that $I_{-1}$ does not intersect with $I_{-1}^{-}$ and  $\widehat{I}_{-1}$.
	This completes the proof of the first part of  Proposition \ref{prop:side-plus}. See Figure \ref{figure:34p1}.

	(2). From the structure of the side $\widehat{s}_0$, the side $\widehat{s}_0$ in the ideal boundary is the union of two disjoint discs, which  are bounded by the ideal boundary of the  ridges $\widehat{s}_0\cap s_{0}^{-}$ and  $\widehat{s}_0\cap s_{1}^{+}$. By calculation, we find the intersection of $\widehat{I}_0\cap I_{0}^{-}\cap I_{1}^{+}$ with the ideal boundary are the four points $A(y),A(g),A(c),A(r)$.  We will determine that one component on  the ideal side $\widehat{s}_0$ with vertices $A(y),A(g)$ and the other component with vertices $A(c),A(r)$. We only prove the first case.  We choose a point $p^{\diamond}$ on the spinal sphere $\widehat{I}_0$  with Heisenberg coordinates $$\left[-\frac{1}{2}+(1-\frac{\sqrt{7}}{2})i,1+\sqrt{7}\right]$$
	 such that $p^{\diamond}$ lies outside the spinal spheres $I_{0}^{-}$ and $I_{1}^{+}$. One can also check that the two straight line segments $[A(y),p^{\diamond}]$ and $[A(g),p^{\diamond}]$ also lie outside the spinal spheres $I_{0}^{-}$ and $I_{1}^{+}$ other than the two end points $A(y)$ and  $A(g)$.  Therefore, the points $A(y)$, $A(g)$ must be the vertices of one bigon and  the points $A(c)$, $A(r)$ must be the vertices of the another bigon.
	
\end{proof}

\subsection{A global combinatorial model of the boundary at infinity of the Ford domain of   $\Delta_{3,4,\infty;\infty}$} \label{subsection:globalmodel34p}

We have shown in Section \ref{sec-ford-34p} the ideal boundary of the Ford domain of  $\Delta_{3,4,\infty;\infty}$
consists of $A$-translates of isometric spheres  of
 $\{B, B^{-1},  B^{-1}AB^{-1},  BA^{-1}B\}$.
We gave the detailed local information on these spinal spheres previously, and we now study the global topology of the region which is outside  of  all spinal spheres, which is $D_{\Gamma_2} \cap \partial \mathbf{H}^2_{\mathbb{C}}$. From these information we can identify the topology of the manifold  at infinity.

Figure \ref{figure:34p1} is a  realistic  view   of the ideal boundary of the Ford domain of $\Delta_{3,4,\infty;\infty}$ with center $q_{\infty}$. This figure is just for motivation, and we need to show the figure is correct. We hope the figure is explicit enough such that the reader can see the ``holes"  in  Figures \ref{figure:34p1}. For example, one of them is the white region enclosed by the spinal spheres $I_{0}^{+}$, $I_{0}^{-}$, $I_{-1}^{-}$, $I_{-1}^{+}$. The ideal boundary of Ford domain is just the intersection of  the exteriors  of all the spinal spheres. Moreover, since the  realistic  view in Figure \ref{figure:34p1} is difficult to understand, we show a combinatorial model of it in Figure	\ref{figure:34pabstract}. Our  computations later will show  the combinatorial model  coincides with Figure \ref{figure:34p1}.

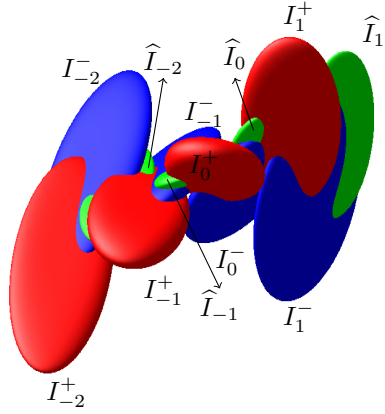
\begin{figure}
\begin{center}
\begin{tikzpicture}
\node at (0,0) {\includegraphics[width=9cm,height=9cm]{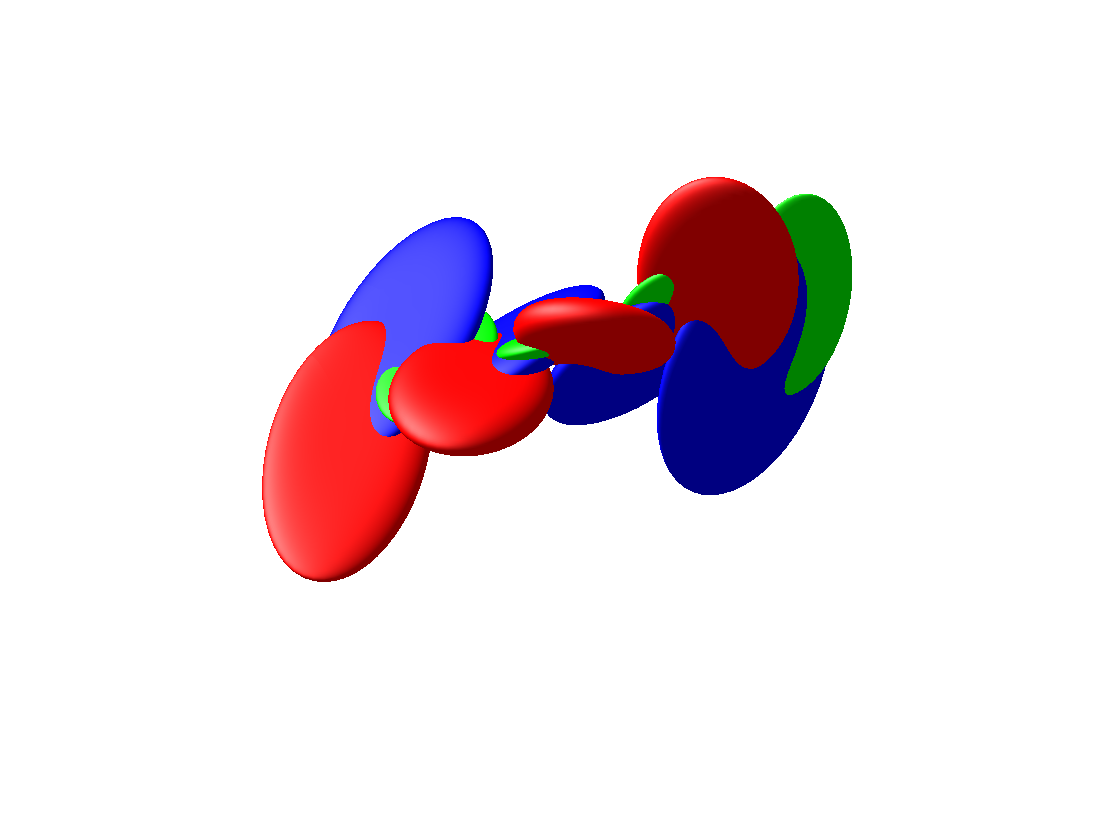}};
\coordinate [label=left:$I_{-2}^{-}$] (S) at (-1,2.1);
\coordinate [label=left:$I_{-2}^{+}$] (S) at (-1.2,-2.2);
\coordinate [label=below:$I_{0}^{-}$] (S) at (0.6,-0.1);
\coordinate [label=above:$I_{0}^{+}$] (S) at (0.24,0.52);
\coordinate [label=above:$I_{1}^{+}$] (S) at (1.5,2.5);
\coordinate [label=above:$I_{1}^{-}$] (S) at (1.5,-1.5);
\coordinate [label=below:$I_{-1}^{+}$] (S) at (-0.3,-0.5);
\coordinate [label=above:$I_{-1}^{-}$] (S) at (0.24,1.2);
\coordinate [label=above:$\widehat{I}_{1}$] (S) at (2.5,2.3);

\draw[->](0.9,1.3)--(0.65,2);
\coordinate [label=above:$\widehat{I}_{0}$] (S) at (0.65,2);

\draw[->](-0.5,0.8)--(-0.3,2);
\coordinate [label=above:$\widehat{I}_{-2}$] (S) at (-0.3,1.9);

\draw[->](-0.25,0.65)--(0.45,-0.8);
\coordinate [label=below:$\widehat{I}_{-1}$] (S) at (0.45,-0.7);

\end{tikzpicture}
\end{center}
  \caption{Realistic view   of the ideal boundary of the Ford domain of $\Delta_{3,4,\infty;\infty}$, which is the region outside all of the spheres.}
  \label{figure:34p1}
\end{figure}

\begin{figure}
	\begin{center}
		\begin{tikzpicture}
		\node at (0,0) {\includegraphics[width=12cm,height=7cm]{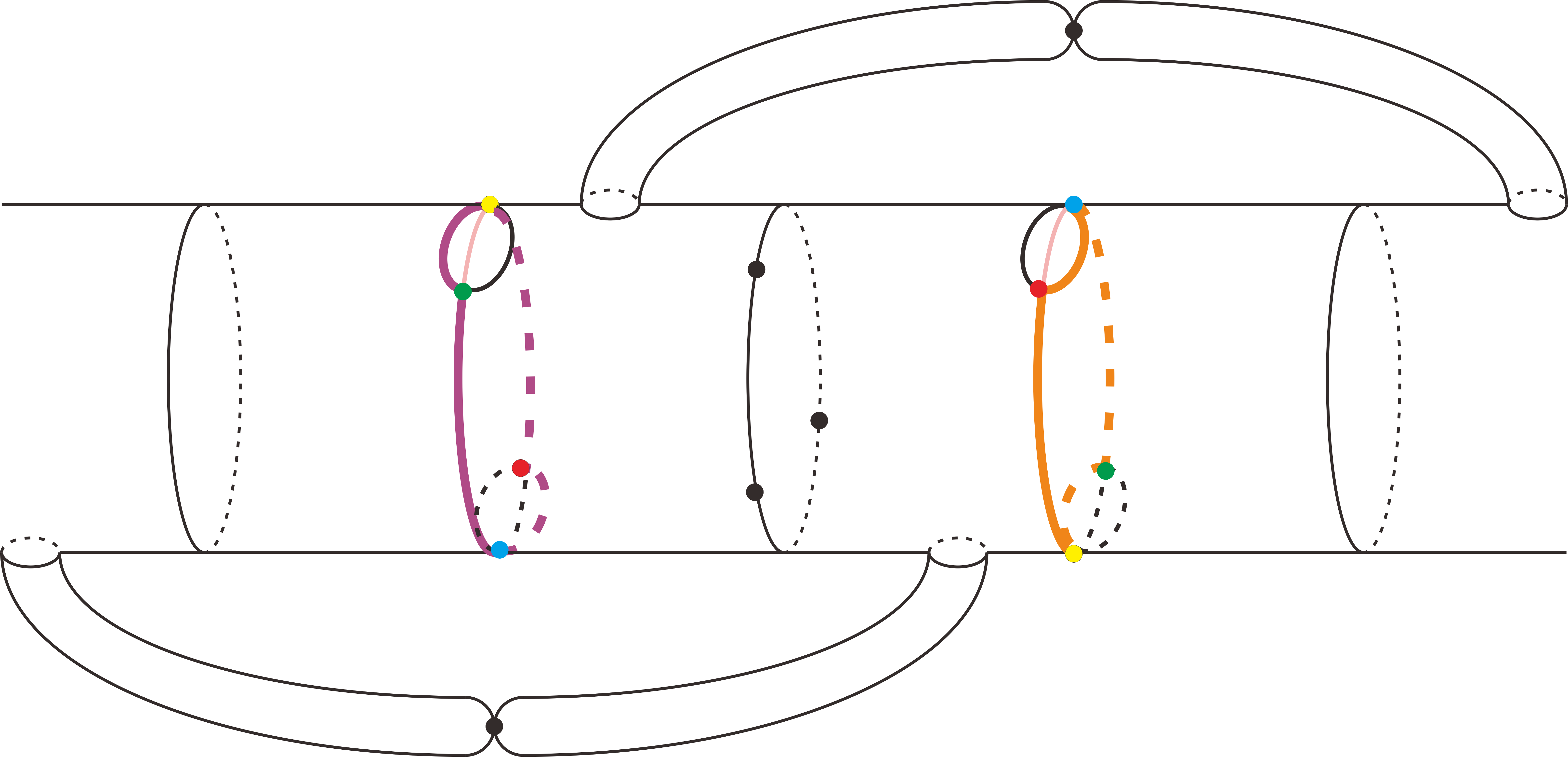}};
		\coordinate [label=left:$\scriptstyle{B(u)}$] (S) at (-0.1,1.2);
		\coordinate [label=left:$\scriptstyle{B^{-1}}(u)$] (S) at (-0.1,-1.2);
		\coordinate [label=left:$\scriptstyle{u}$] (S) at (0.3,-0.2);
		\coordinate [label=left:$\scriptstyle{B}$] (S) at (-0.8,0);
		\coordinate [label=left:$\scriptstyle{A^{-1}B^{-1}A}$] (S) at (-2.6,0);
		\coordinate [label=left:$\scriptstyle{A^{-1}BA}$] (S) at (-5,0);
		\coordinate [label=left:$\scriptstyle{B^{-1}}$] (S) at (1.7,0);
		\coordinate [label=left:$\scriptstyle{ABA^{-1}}$] (S) at (4.1,0);
		\coordinate [label=left:$\scriptstyle{AB^{-1}A^{-1}}$] (S) at (6.5,0);
		\coordinate [label=left:$A^{-1}(p_{AB})$] (S) at (-1.5,-3.7);
		\coordinate [label=left:$c$] (S) at (-2,-1.8);
		\coordinate [label=left:$r$] (S) at (-1.6,-0.7);
		\coordinate [label=left:$g$] (S) at (-2,0.7);
		\coordinate [label=left:$y$] (S) at (-2,1.9);
		\coordinate [label=left:$\scriptstyle{A(y)}$] (S) at (2.7,-1.9);
		\coordinate [label=left:$\scriptstyle{A(g)}$] (S) at (3.3,-0.9);
		\coordinate [label=left:$\scriptstyle{A(r)}$] (S) at (2,0.8);
		\coordinate [label=left:$\scriptstyle{A(c)}$] (S) at (2.7,1.8);
		\coordinate [label=left:$p_{AB}$] (S) at (2.7,3.7);
		\end{tikzpicture}
	\end{center}
	\caption{A combinatorial picture of  the ideal boundary of the Ford domain of $\Delta_{3,4,\infty;\infty}$, which is out side of the singular surface $\Sigma$.}
	\label{figure:34pabstract}
\end{figure}

It seems to the authors  that it is very difficult to write down explicitly  the  definition equations for the intersections  of the spinal spheres of $\{B,A^{-1}B^{-1}A\}$,  $\{B,A^{-1}B^{-1}AB^{-1}A\}$ and  $\{A^{-1}B^{-1}A,A^{-1}B^{-1}AB^{-1}A\}$ in the Heisenberg group, even though we have proved each of them is a circle, and the intersection of any of them with ideal boundary of the Ford domain is one or two arcs.
In this section, we first describe a global combinatorial object. We then use
geometric arguments to show that it is combinatorially equivalent to the ideal boundary of the  Ford domain of   $\Delta_{3,4,\infty;\infty}$. Last, we describe the topology of the manifold at infinity of $\Delta_{3,4,\infty;\infty}$ from the combinatorial model.
In particular, we will show that the ideal boundary  of the Ford domain of   $\Delta_{3,4,\infty;\infty}$ is an infinite genus $A$-invariant handlebody.

We first take $C^{*}$ to be the plane
\begin{equation}\label{C}
\left\{ \left[r-\frac{1}{2}, r-\frac{\sqrt{7}}{2},s\right] \,\Big| \, r,s \in \mathbb{R} \right\}
\end{equation}
in the Heisenberg group. See Figure 	\ref{figure:geodisk34p} for part of this plane, which  is the red plane in Figure	\ref{figure:geodisk34p}.

\begin{figure}
\begin{center}
\begin{tikzpicture}
\node at (0,0) {\includegraphics[width=6cm,height=6cm]{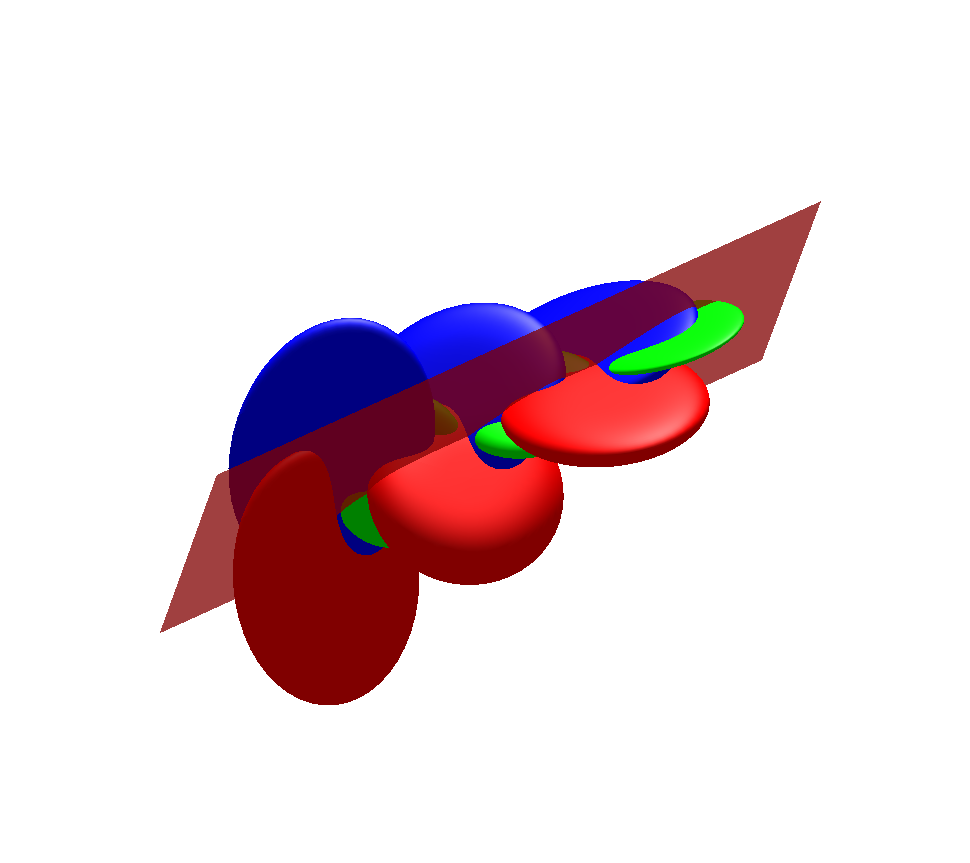}};
\coordinate [label=left:$ C^{*}$] (S) at (1.8,1.4);

\end{tikzpicture}
\end{center}
  \caption{Realistic view of the $A$-invariant plane $ C^{*}$  for the Ford domain of $\Delta_{3,4,\infty;\infty}$.}
	\label{figure:geodisk34p}
\end{figure}

We now consider the intersections of  $C^{*}$ with the spinal spheres of  $B$, $B^{-1}$, $ABA^{-1}$, and $AB^{-1}A^{-1}$:

$I_{0}^{+}  \cap C^{*}$  is the circle in $C^{*}$ with equation
\begin{equation}\label{34p-B-C*}
\left(\left(r-1/2\right)^2+\left(r-\sqrt{7}/2\right)^2\right)^2+s^2 = 4;
\end{equation}
$I_{0}^{-}  \cap C^{*}$  is the circle in $C^{*}$ with equation
\begin{equation}\label{34p-b-C*}
\left(\left(r-1\right)^2+r^2\right)^2+\left(s+r-\sqrt{7}+\sqrt{7}r\right)^2 = 4;
\end{equation}
$I_{1}^{+}  \cap C^{*}$  is the circle in $C^{*}$ with equation
\begin{equation}\label{34p-ABa-C*}
\left(\left(r+1\right)^2+r^2\right)^2+\left(s-\sqrt{7}-3r+\sqrt{7}r\right)^2 = 4;
\end{equation}
$I_{1}^{-}  \cap C^{*}$  is the circle in $C^{*}$ with equation
\begin{equation}\label{34p-Aba-C*}
\left(\left(r+1/2\right)^2+\left(r+\sqrt{7}/2\right)^2\right)^2+\left(s-2r+2\sqrt{7}r\right)^2 = 4.
\end{equation}

Now $C^*$ passes through $[-\frac{1}{2},-\sqrt{7}/2,0]$, which is  the tangent point of the spinal spheres  $I_{0}^{+}$ and $I_{1}^{-}$  in the  Heisenberg group, and  it corresponds to the point $p_{AB}$ in Section \ref{sec-ford-34p},  which is also the point $r=s=0$ in Figure 	\ref{figure:reddisk-34p}.

With the definition equations above, it can be showed  rigorously that the circles $I_{0}^{+}  \cap C^{*}$  and $I_{0}^{-}  \cap C^{*}$ intersect in four points in $C^{*}$, we  take one of them with coordinates $r_1=0.038807$ and $s_1=0.7312373192$  numerically, and in fact $r_1$ is one root of
\begin{equation}\label{34p-B-b-solution}
\begin{split}
1 - 14 r - 6 \sqrt{7}r + 70 r^2 + 16 \sqrt{7} r^2 - 94 r^3 -
30 \sqrt{7}r^3 \\
+ 88 r^4 + 20  \sqrt{7} r^4 - 36 r^5 -
12 \sqrt{7} r^5 + 16 r^6=0.
\end{split}
\end{equation}
 Equation (\ref{34p-B-b-solution}) has a pair of complex roots $0.693921 \pm1.3301i$ and four real roots  $0.0388075$, $0.713275$, $0.741563$
and	$1.35283$ numerically.

From the equations above, we can see that  $I_{0}^{-}\cap C^*$ and  $ I_{1}^{-} \cap C^*$ intersect in four points in $C^*$. We take one of them, say $r=0$ and $s=\sqrt{7}-\sqrt{3}$, which is the upper vertex of the red disk in Figure 	\ref{figure:reddisk-34p}.

The circles $I_{1}^{+}  \cap C^{*}$  and $I_{1}^{-}  \cap C^{*}$ intersect in four points in $C^{*}$, we take one of them with coordinates  $r_2=-0.038807$ and $s_2=0.858972$ numerically, and in fact $r_2$ is one root of
\begin{equation}\label{34p-ABa-Aba-solution}
\begin{split}
1 + \sqrt{7} + 56 r + 20 \sqrt{7} r + 182 r^2 + 86 \sqrt{7} r^2 +
304 r^3 +124 \sqrt{7} r^3\\
+ 228 r^4 + 108  \sqrt{7}r^4 + 120 r^5 +
48 \sqrt{7} r^5 + 16 r^6 + 16 \sqrt{7} r^6=0.
\end{split}
\end{equation}
 Equation (\ref{34p-ABa-Aba-solution}) has a pair of complex roots
$0.693921 \pm 1.3301i$ and four real roots $-0.0388075$, $-1.35283$, $-0.741563$ and  $-0.713275$
numerically.

\begin{figure}
\begin{center}
\begin{tikzpicture}
\node at (0,0) {\includegraphics[width=7cm,height=7cm]{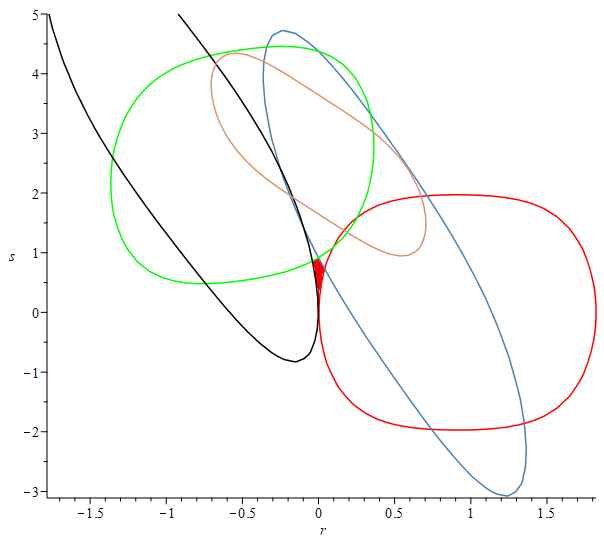}};
\draw[->] (1.2,0.2)--(1.2,0);
\coordinate [label=below:$ C^{*}\cap \widehat{I}_0$] (S) at (1.2,0.2);

\draw[->] (-0.7,-0.65)--(-0.9,-0.85);
\coordinate [label=left:$ C^{*}\cap I_{1}^{-}$] (S) at (-0.8,-0.95);

\draw[->] (-0.4,-0.05)--(-0.6,0.2);
\coordinate [label=left:$ C^{*}\cap I_{1}^{+}$] (S) at (-0.1,0.5);

\draw[->] (0.5,-1.55)--(0.4,-2.1);
\coordinate [label=left:$ C^{*}\cap I_{0}^{+}$] (S) at (0.7,-2.3);

\draw[->] (1.2,-1.55)--(1.4,-1.55);
\coordinate [label=right:$ C^{*}\cap I_{0}^{-}$] (S) at (1.3,-1.55);

\end{tikzpicture}
\end{center}
  \caption{The small red colored disk $C$ in the plane $C^{*}$ which is co-bounded by four arcs  lying in  $C^{*}\cap I_{0}^{-}$, $C^{*}\cap I_{0}^{+}$, $C^{*}\cap I_{1}^{+}$ and $C^{*}\cap I_{1}^{-}$ respectively.}

	\label{figure:reddisk-34p}
\end{figure}

We take $L_0$ be the arc in $I_{0}^{+}  \cap C^{*}$ with $r\in [0,r_1]$ and $s \in [0,s_1]$, which is the intersection of the red circle and  the red disk in Figure \ref{figure:reddisk-34p}. Take $L_1$ to be the arc in $I_{0}^{-}  \cap C^{*}$ with $r\in [0,r_1]$ and $s \in [s_1,\sqrt{7}-\sqrt{3}]$,  which is the steelblue  boundary arc of the red disk in Figure \ref{figure:reddisk-34p}. Note that $s$  decreases  as function of $r$ in $L_1$.   Take $L_2$ to be the arc in $I_{1}^{+}  \cap C^{*}$ with $r\in [r_2,0]$ and $s \in [s_2,\sqrt{7}-\sqrt{3}]$,  which is the green boundary arc of the red disk in Figure \ref{figure:reddisk-34p}, $s$ also  decreases  as function of $r$ in $L_2$.   Take $L_3$ to be the arc in $I_{1}^{-}  \cap C^{*}$ with $r\in [r_2,0]$ and $s \in [0,s_2]$, which is the black boundary arc of the red disk in Figure \ref{figure:reddisk-34p}.

\begin{lem} \label{34p-boundary-red-disk}The union of four arcs $L_0\cup L_1 \cup L_2 \cup L_3$  is a simple closed curve in  $C^{*}$.
	
\end{lem}
\begin{proof}From the above chosen parameters of  these arcs, it is easy to see that $L_{i}$ only intersects with   $L_{i-1}$ and  $L_{i+1}$  in its two end vertices, and $L_{i}$ does not intersect $L_{i+2}$, for $i=0, 1, 2, 3\mod 4$. So  $L_0\cup L_1 \cup L_2 \cup L_3$  is a simple closed curve in  $C^{*}$.
\end{proof}

We take $C$ be the disk in the plane  $C^{*}$  with boundary  $L_0\cup L_1 \cup L_2 \cup L_3$.

\begin{lem} \label{geodisk34pintersetion}The disk $C$ has the following properties:
	
(1). The boundary of $C$ is a union of four arcs, in the spinal spheres of $B$, $B^{-1}$, $ABA^{-1}$ and $AB^{-1}A^{-1}$ cyclically.
	
(2). The interior of  the disk $C$  is disjoint from all the spinal spheres of $A$-translations of $B$, $B^{-1}$ and  $B^{-1}AB^{-1}$.
	
(3). The  disk $C$ and any translation $A^k(C)$ for $k \neq 0$ are disjoint.
	
\end{lem}

\begin{proof}
	
We have proved (1) of Lemma \ref{geodisk34pintersetion} above.

For (2) of Lemma \ref{geodisk34pintersetion}, from  the vertical projections of the disk $C$, and the spinal spheres  $I_{k}^{\pm}$  and $\widehat{I}_{k}$ for $k \in \mathbb{Z}$, to the plane $\mathbb{C}$, we can easy see that $C$ does not intersect the spinal spheres  $I_{k}^{\pm}$  and $\widehat{I}_{k}$ except
$I_{0}^{+}$, $I_{0}^{-}$, $I_{1}^{+}$, $I_{1}^{-}$  and $\widehat{I}_{0}$. See Figure \ref{figure:projE-34p} for the vertical projections of them.

The intersection $\widehat{I}_{0} \cap C^{*}$ is a circle
\begin{equation}\label{34p-bAb-C-solution}
	(2r^2)^2+(s-\sqrt{7}-r+\sqrt{7}r)^2 = 1,
	\end{equation}
in the $rs$-plane $C^{*}$. Note that $\sqrt{7}-\sqrt{3}$ is the largest one  of $s$-coordinate of the disk $C$, but take $s=\sqrt{7}-\sqrt{3}$ in Equation (\ref{34p-bAb-C-solution}), there is no real solution of $r$ (numerically, there are four solutions  $r=-0.555156 \pm 0.0702174i$ and	$0.555156 \pm 1.13516i$). Take a sample point on   $\widehat{I}_{0} \cap C^{*}$, for example $r=0$, $s=\sqrt{7}- 1$. Note that $s=\sqrt{7}-1 > s=\sqrt{7}-\sqrt{3}$, this means that the $s$-coordinate of any point  in the circle $\widehat{I}_{0} \cap C^{*}$ is  larger than $\sqrt{7}-\sqrt{3}$. So the disk  $C$  does not intersect with $\widehat{I}_{0}$.

For (3) of Lemma \ref{geodisk34pintersetion}, note that  $A^{k}(C)$ is a subdisk  of the square $r \in[-0.038, -0.038]$ numerically and $s\in [0, \sqrt{7}-\sqrt{3}]$ in the plane $A^{k}(C^{*})$
\begin{equation}\label{AkC}
	\left[r-\frac{1+3k}{2}, r-\frac{\sqrt{7}(k+1)}{2},s+3kr+k\sqrt{7}-kr\sqrt{7}\right],
	\end{equation}
	with $r, s$ reals in Heisenberg coordinates.
Then it is easy to see $C$ and $A^{k}(C)$ are disjoint disks  for $k \neq 0$  in the Heisenberg group.
	
\end{proof}

We then take $E^{*}$ to be the plane
\begin{equation}\label{E^{*}}
\left[r+\frac{1}{4}, -2r-\frac{\sqrt{7}}{4},r+s-\frac{\sqrt{7}}{2}\right]
\end{equation}
with $r$ and $s$ reals
in   Heisenberg coordinates. See Figure 	\ref{figure:eulerdisk34p} for part of this plane, $E^{*}$  is one of the  blue planes in Figure	\ref{figure:eulerdisk34p}.

\begin{figure}
\begin{center}
\begin{tikzpicture}
\node at (0,0) {\includegraphics[width=6cm,height=6cm]{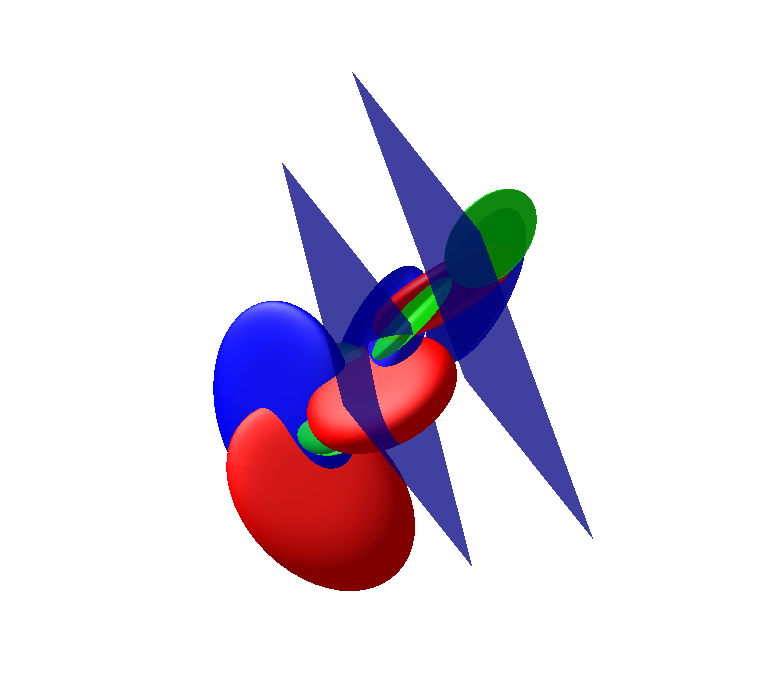}};
\coordinate [label=left:$ E^{*}$] (S) at (-0.5,1.8);
\coordinate [label=left:$ A(E^{*})$] (S) at (0.8,2.4);
\end{tikzpicture}
\end{center}
  \caption{The plane $E^{*}$  intersects only the spinal spheres of $B$ and $B^{-1}$ and the  plane $A(E^{*})$   intersects only the spinal spheres of $ABA^{-1}$ and $AB^{-1}A^{-1}$.}
	\label{figure:eulerdisk34p}
\end{figure}

We consider the intersections of $I_{0}^{+}$ and $I_{0}^{-}$ with  $E^{*}$:
$I_{0}^{+}  \cap E^{*}$ has the definition  equation
\begin{equation}\label{34p-B-E*}
\left(\left(r+\frac{1}{4}\right)^2+\left(2r+\frac{\sqrt{7}}{4}\right)^2\right)^2+\left(r+s-\frac{\sqrt{7}}{2}\right)^2 = 4
\end{equation}
with $r$ and $s$ reals. It is easy to see  Equation (\ref{34p-B-E*}) gives us a closed curve in the $rs$-plane  $E^{*}$.

\begin{lem} \label{scc$I_{0}^{+} E^{*}$-34p} $I_{0}^{+}  \cap E^{*}$ is a simple closed curve in  $E^{*}$.
\end{lem}
\begin{proof}
For the Equation (\ref{34p-B-E*}), we can solve $s$ in terms of $r$ as
$$s = \frac{1}{4}\sqrt{7}-r \pm\frac{1}{2}\sqrt{-40r^3\sqrt{7}-100r^4-4r^2\sqrt{7}-20r^3+15-4r\sqrt{7}-49r^2-2r}.$$
For one real $r$, there are at most two reals $s$ satisfying Equation (\ref{34p-B-E*}), and there is  exactly  one real $s$ satisfying Equation (\ref{34p-B-E*}) when

\begin{equation}\label{34p-delta}
-40r^3\sqrt{7}-100r^4-4r^2\sqrt{7}-20r^3+15-4r\sqrt{7}-49r^2-2r=0.
\end{equation}
Equation (\ref{34p-delta}) has a pair of complex roots, and  a pair of real toots
$$r=-\frac{1}{20}-\frac{\sqrt{7}}{10}\pm\frac{\sqrt{149+4\sqrt{7}}}{20},$$ which are the maximum and the minimum  of $r$ satisfies Equation (\ref{34p-B-E*}). So the curve $I_{0}^{+}  \cap E^{*}$ is a simple closed curve, that is, there is no self-intersection.
\end{proof}

Similarly, we can prove $I_{0}^{-}  \cap E^{*}$  is a simple closed curve  in the plane $E^{*}$ with definition equation
\begin{equation}\label{34p-b-E*}
\left(\left(r-1/4\right)^2+\left(-2r+\sqrt{7}/4\right)^2\right)^2+\left(s-r+\sqrt{7}r-\sqrt{7}/2\right)^2 = 4
\end{equation}
with  $r$ and $s$ reals.

Now we can solve algebraically that  the triple intersection $I_{0}^{+} \cap I_{0}^{-} \cap E^{*}$ are two points with $r=0$ and $s= \pm \sqrt{15}/2+\sqrt{7}/2$ in the $rs$-plane  $E^{*}$. They are the points $\left[1/4, -\sqrt{7}/4, \pm\sqrt{15}/2\right]$ in the Heisenberg group. In other words, simple closed curves  $I_{0}^{+} \cap E^{*}$ and  $ I_{0}^{-} \cap E^{*}$
intersect in exactly two points.  Each of $I_{0}^{+} \cap E^{*}$ and  $ I_{0}^{-} \cap E^{*}$ bounds a closed disk in $E^{*}$, now the union of them is a bigger disk in $E^{*}$, we denote it by $E$.  See Figure 	\ref{figure:greendisk-34p}.

\begin{figure}
\begin{center}
\begin{tikzpicture}
\node at (0,0) {\includegraphics[width=7cm,height=7cm]{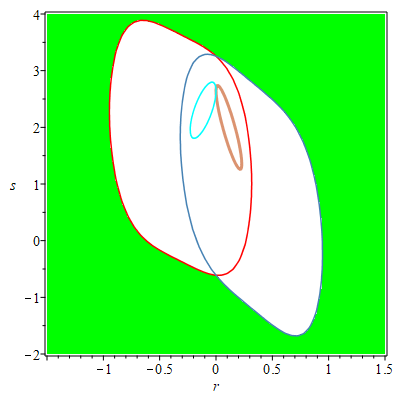}};
\coordinate [label=left:$I_{0}^{+}\cap E^{*}$] (S) at (-0.5,2.5);
\coordinate [label=left:$I_{0}^{-}\cap E^{*}$] (S) at (2.8,-1.5);

\draw[->] (-0.15,1.1)--(-0.55,0.8);
\coordinate [label=below:$\widehat{I}_{-1}\cap E^{*}$] (S) at (-0.55,0.8);
\draw[->] (0.78,0.6)--(1.1,0.3);
\coordinate [label=below:$\widehat{I}_{0}\cap E^{*}$] (S) at (1.1,0.3);

\end{tikzpicture}
\end{center}
\caption{The intersection of the plane $E^{*}$ with $I_{0}^{+}$, $I_{0}^{-}$, $\widehat{I}_{-1}$ and $\widehat{I}_{0}$. The disk $E$ lies in the plane $E^{*}$  is co-bounded by two arcs lying in  two Jordan curves $I_{0}^{+}\cap E^{*}$ and $I_{0}^{-}\cap E^{*}$. $E$ is the  disk which is the complement of the outer green region.  The tangent point of $\widehat{I}_{-1}\cap E^{*}$ and $\widehat{I}_{0}\cap E^{*}$ is the point $[\frac{1}{4},-\frac{\sqrt{7}}{4},-\frac{\sqrt{7}}{2}]$ in the Heisenberg group, which is the triple intersection $\widehat{I}_{-1} \cap \widehat{I}_{0} \cap E^{*}$.}
	\label{figure:greendisk-34p}
\end{figure}

\begin{lem} \label{eulerdisk34pintersetion}The plane $E^{*}$ and disk $E$ have the following properties:
	
(1). The boundary of $E$ is a union of two arcs: one in the spinal sphere of $B$ and the other  in the spinal sphere of  $B^{-1}$.
	
(2). The interior of  the holed plane $E^{*}-E$  is disjoint from all the spinal spheres of $A$-translations of $B$, $B^{-1}$ and  $B^{-1}AB^{-1}$.
	
(3). The  plane $E^{*}$ and any translation $A^k(E^{*})$ for $k \neq 0$ are disjoint.

\end{lem}

\begin{proof}We have proved (1) of Lemma \ref{eulerdisk34pintersetion} above.

For (2) of Lemma \ref{eulerdisk34pintersetion}, the vertical projection of the plane  $E^{*}$ to $\mathbb{C}$-plane is a line
\begin{equation}\label{projection-E^{*}-34p}	\left\{r+\frac{1}{4}+\left(-2r-\sqrt{7}/4\right)i \in \mathbb{C} \right\}
\end{equation}
with $r$ reals.
From the  equations  of  $I^{+}_{k}$,  $I^{-}_{k}$ and  $\widehat{I_{k}}$, we can get the vertical projections of them to the $\mathbb{C}$-plane.
It is easy to see that the  vertical projection of $E^{*}$ only intersects with the  vertical projections of
$I^{+}_{0}$,  $I^{-}_{0}$,   $\widehat{I}_{0}$ and  $\widehat{I}_{-1}$ from  the equations of their vertical projections on the $\mathbb{C}$-plane.

We can write down the  equation  of the circle   $\widehat{I}_{0}\cap E^{*}$ in the $rs$-plane $E^{*}$, which is a circle
\begin{equation}\label{bAb-E^{*}-34p}
\left(\left( r+\frac{3}{4}\right)^2+\left(2r-\frac{\sqrt{7}}{4}\right)^2\right)^2+ \left(s+3r-\sqrt{7}+\sqrt{7}r\right)^2=1.
\end{equation}
Then $s$ is solved as function of $r$,
$$	-\sqrt{7}r+\sqrt{7}-3r\pm\frac{\sqrt{-12r+8\sqrt{7}r-77r^2+12\sqrt{7}r^2-60r^3+40r^3\sqrt{7}-100r^4}}{2}.$$

We can also write down the  equation  of the circle   $\widehat{I}_{-1}\cap E^{*}$ in the $rs$-plane $E^{*}$, which is a circle
\begin{equation}\label{abAbA-E^{*}-34p}
\left(\left( r-\frac{3}{4}\right)^2+\left(2r+\frac{\sqrt{7}}{4}\right)^2\right)^2+ \left(s-3r-\sqrt{7}\right)^2=1,
\end{equation}
and we can solve $s$ as function of $r$
\begin{equation}
\sqrt{7}+3r \pm\frac{\sqrt{-40r^3\sqrt{7}-100r^4+12\sqrt{7}r^2+60r^3-8\sqrt{7}r-77r^2+12r}}{2}.
\end{equation}	
	
See Figure \ref{figure:greendisk-34p} for the circles  $\widehat{I}_{0}\cap E^{*}$ and $\widehat{I}_{-1}\cap E^{*}$   in the $rs$-plane $E^{*}$. By solving the  equations of them,  it is easy to see that there is no triple intersection of $I_{0}^{+}$,  $\widehat{I}_{0}$ and  $ E^{*}$, then the two circles  $\widehat{I}_{0}\cap E^{*}$ and $I_{0}^{+}\cap E^{*}$ are disjoint  in the place $E^{*}$. Take a sample point on $\widehat{I}_{0}\cap E^{*}$, for example, take $r=s=0$, we can see $\widehat{I}_{0}\cap E^{*}$ lies in the disk bounded by $I_{0}^{+}\cap E^{*}$ in the place $E^{*}$. Similarly,  we can see $\widehat{I}_{0}\cap E^{*}$ lies in the disk bounded by $I_{0}^{-}\cap E^{*}$ in the place $E^{*}$;   $\widehat{I}_{-1}\cap E^{*}$ lies in the disk bounded by $I_{0}^{+}\cap E^{*}$ in the place $E^{*}$;   $\widehat{I}_{-1}\cap E^{*}$ lies in the disk bounded by $I_{0}^{-}\cap E^{*}$ in the place $E^{*}$. So these two circles lie in the disk $E$.

For (3) of Lemma \ref{eulerdisk34pintersetion},  by the equation of  $E^{*}$ and the action of $A$,  we get that the vertical projection of the plane  $A^{k}(E^{*})$ to $\mathbb{C}$-plane is the  line,
\begin{equation}\label{E^{*}}
\left\{	r+\frac{1}{4}-\frac{3k}{2}+\left( -2r-\frac{\sqrt{7}(1+2k)}{4}\right)i\in \mathbb{C} \right \}
	\end{equation}
with $r$ real, so the  vertical projections of  $E^{*}$ and  $A^{k}(E^{*})$ for $k \neq 0$ are disjoint parallel lines, then  $E^{*}$ and  $A^{k}(E^{*})$ are disjoint in the Heisenberg group for $k \neq 0$.
\end{proof}

\begin{figure}
\begin{tikzpicture}[style=thick,scale=1]
\draw (0,0)  circle (1.414);  \node [above right] at (-0.9,1.3){$I_{0}^{+}$};
\draw (-1.5,-1.32288)  circle (1.414);  \node [above right] at (-2.5,0.1){$I_{1}^{+}$};
\draw (1.5,1.32288)  circle (1.414);  \node [above right] at (1,2.8){$I_{-1}^{+}$};

\draw[red] (0.5,-1.32288)  circle (1.414);  \node [above right] at (1.3,-3){$I_{0}^{-}$};
\draw[red] (2,0)  circle (1.414); \node [above right] at (3.25,-1.2){$I_{-1}^{-}$};
\draw[red] (-1,-2.64575)  circle (1.414); \node [above right] at (-0.8,-4.55){$I_{1}^{-}$};

\draw [blue](-0.5,-1.32288)  circle (1);  \node [above right] at (-0.1,-2.7){$\widehat{I}_{0}$};
\draw[blue] (1,0)  circle (1); \node [above right] at (1.375,-2.7+1.32288){$\widehat{I}_{-1}$};
\draw[blue] (-2,-2.64575)  circle (1);   \node [above right] at (-1.6,-2.7-1.32288){$\widehat{I}_{1}$};

\draw[ thick] (-1.5,2.8386)--(2,-4.1614);\node [above right] at (-1.5,2.8386){$ E^{*}$};

\draw[ thick] (-3.5,-4.323)--(3.5,2.677);\node [above right] at (3.5,2.677){$ C^{*}$};
\draw[red, line width=3pt] (-0.55,-1.373)--(-0.45,-1.273);

\end{tikzpicture}
\caption{The vertical projections of the planes $E^{*}$, $C^{*}$ and the spinal spheres of $I_{k}^{-}$,  $I_{k}^{+}$ and $\widehat{I}_{k}$ for $k=-1,0,1$. The (very) short thick red arc is the vertical projection of the disk $C$. }	
\label{figure:projE-34p}
\end{figure}

In Figure \ref{figure:B-b-annulus-34p}, we give the abstract pictures of  annuli of the intersections of $I_{0}^{+}$ and $I_{0}^{-}$ with the ideal  boundary  of the Ford domain of  $\Delta_{3,4,\infty;\infty}$. The abstract pictures of  annuli of the intersections of $I_{k}^{+}$ and $I_{k}^{-}$ with the ideal boundary  of the Ford domain of  $\Delta_{3,4,\infty;\infty}$ are similar by the $A$-action.

\begin{figure}
\begin{center}
\begin{tikzpicture}
\node at (0,0) {\includegraphics[width=11.5cm,height=5cm]{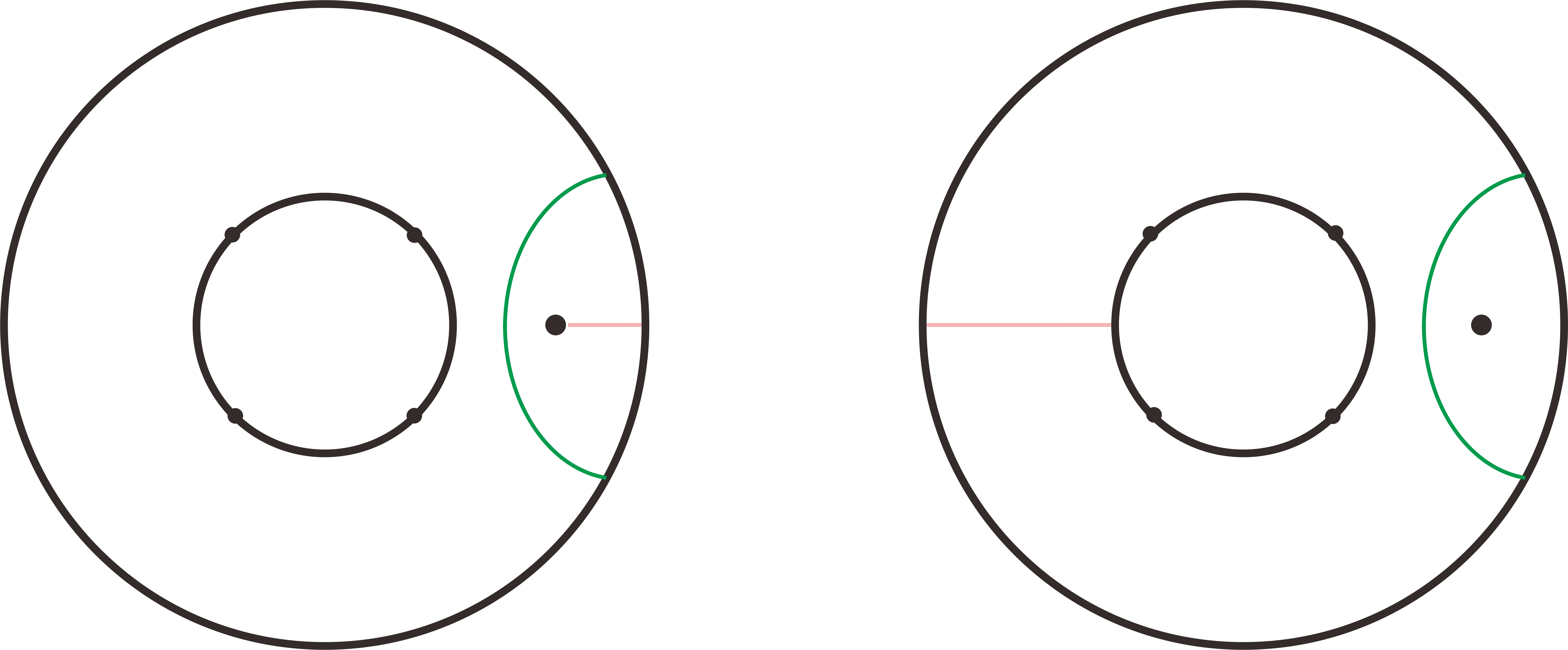}};
\coordinate [label=left:$ \scriptscriptstyle{B^{-1}}$] (S) at (-0.3,0.7);
\coordinate [label=left:$ \scriptscriptstyle{B^{-1}}$] (S) at (-0.3,-0.5);
\coordinate [label=left:$ \scriptscriptstyle{B^{-1}}$] (S) at (-4.3,2.3);
\coordinate [label=left:$\scriptscriptstyle{C}$] (S) at (-1,-0.2);
\coordinate [label=left:$\scriptscriptstyle{p_{AB}}$] (S) at (-1.3,0.16);
\coordinate [label=left:$\scriptscriptstyle{E}$] (S) at (-1.6,0.5);
\coordinate [label=left:$\scriptscriptstyle{A^{-1}B^{-1}A}$] (S) at (-2.35,0.08);
\coordinate [label=left:$\scriptscriptstyle{A^{-1}B^{-1}AB^{-1}A}$] (S) at (-2,1.12);
\coordinate [label=left:$\scriptscriptstyle{A^{-1}B^{-1}AB^{-1}A}$] (S) at (-2,-1.12);
\coordinate [label=left:$\scriptscriptstyle{A^{-1}B^{-1}A}$] (S) at (-4.2,0.08);

\coordinate [label=left:$ \scriptscriptstyle{B}$] (S) at (6.2,0);
\coordinate [label=left:$ \scriptscriptstyle{A^{-1}(p_{AB})}$] (S) at (5.8,0.3);
\coordinate [label=left:$ \scriptscriptstyle{E}$] (S) at (5.2,-0.6);
\coordinate [label=left:$ \scriptscriptstyle{ABA^{-1}}$] (S) at (4.75,0);
\coordinate [label=left:$ \scriptscriptstyle{ABA^{-1}}$] (S) at (3.25,0.1);
\coordinate [label=left:$ \scriptscriptstyle{B^{-1}AB^{-1}}$] (S) at (4,1.2);
\coordinate [label=left:$ \scriptscriptstyle{B^{-1}AB^{-1}}$] (S) at (4,-1.1);
\coordinate [label=left:$ \scriptscriptstyle{C}$] (S) at (1.75,-0.1);
\coordinate [label=left:$ \scriptscriptstyle{B}$] (S) at (1.7,1.9);
\coordinate [label=left:$ \scriptscriptstyle{B}$] (S) at (1.7,-1.9);
\end{tikzpicture}
\end{center}
\caption{The punctured annuli of the intersections of $I_{0}^{+}$ and $I_{0}^{-}$ with the ideal  boundary  of the Ford domain of  $\Delta_{3,4,\infty;\infty}$. The left-side is the annulus  for  $I_{0}^{+}$, the green arc is the intersection of this annulus and the disk $E$; the  pink  arc is the intersection of this annulus and the disk $C$; the three arcs labeled by  $B^{-1}$ mean  these arcs are glued to  annulus  for  $I_{0}^{-}$; the two arcs labeled by  $A^{-1}B^{-1}A$ mean  these arcs are glued to  annulus  for  $I_{-1}^{-}$; the two arcs labeled by  $A^{-1}B^{-1}AB^{-1}A$  mean  these arcs are glued to  bigons  for  $\widehat{I}_{-1}$. The right-side is the annulus  for  $I_{0}^{+}$, the labels obey the similar law as the left-side.}
	\label{figure:B-b-annulus-34p}
\end{figure}

Take a singular surface $$\Sigma =  \bigcup_{k \in \mathbb{Z}} ((\hat{s}_{k} \cup s_k^{\pm})\cap \partial \hc)$$  in the Heisenberg group. See Figure \ref{figure:34pabstract} for an abstract picture of this surface.
  $\Sigma$ is an annulus with $A$-action and which   is tangent to itself  at infinite many points.

From the $A$-action on the plane $E^{*}$ we get a plane
$A(E^{*})$
\begin{equation}\label{$A(E^{*})$}
\left[r-\frac{5}{4}, -2r-\frac{3\sqrt{7}}{4},-5r+s+\frac{(1-2r)\sqrt{7}}{2}\right]
\end{equation}
with $r$ and $s$ reals in  the  Heisenberg group. In particular, $E^{*}$ and $A(E^{*})$ are parallel planes in the    Heisenberg group. One component of $(\partial{\bf H}^2_{\mathbb C}\backslash \{q_{\infty}\})\backslash(E^{*}\cup A(E^{*}))$ contains the disk $C$, we denote the closure of  this component by $U$,  $U$ is homeomorphic $E^{*} \times [0,1]$.  We denote by $H$ the closure of  complement of  all  the 3-balls bounded $\hat{s}_{k}$ and $s_k^{\pm}$  in $U$. Note that $$\partial_{\infty} D_{\Gamma_2} =\cup^{\infty} _{k= -\infty} A^{k}(H).$$
We will show $\partial_{\infty} D_{\Gamma_2}$ is an infinite  genus  handlebody in the Heisenberg group.

\begin{figure}
\begin{center}
\begin{tikzpicture}
\node at (0,0) {\includegraphics[width=11.5cm,height=18cm]{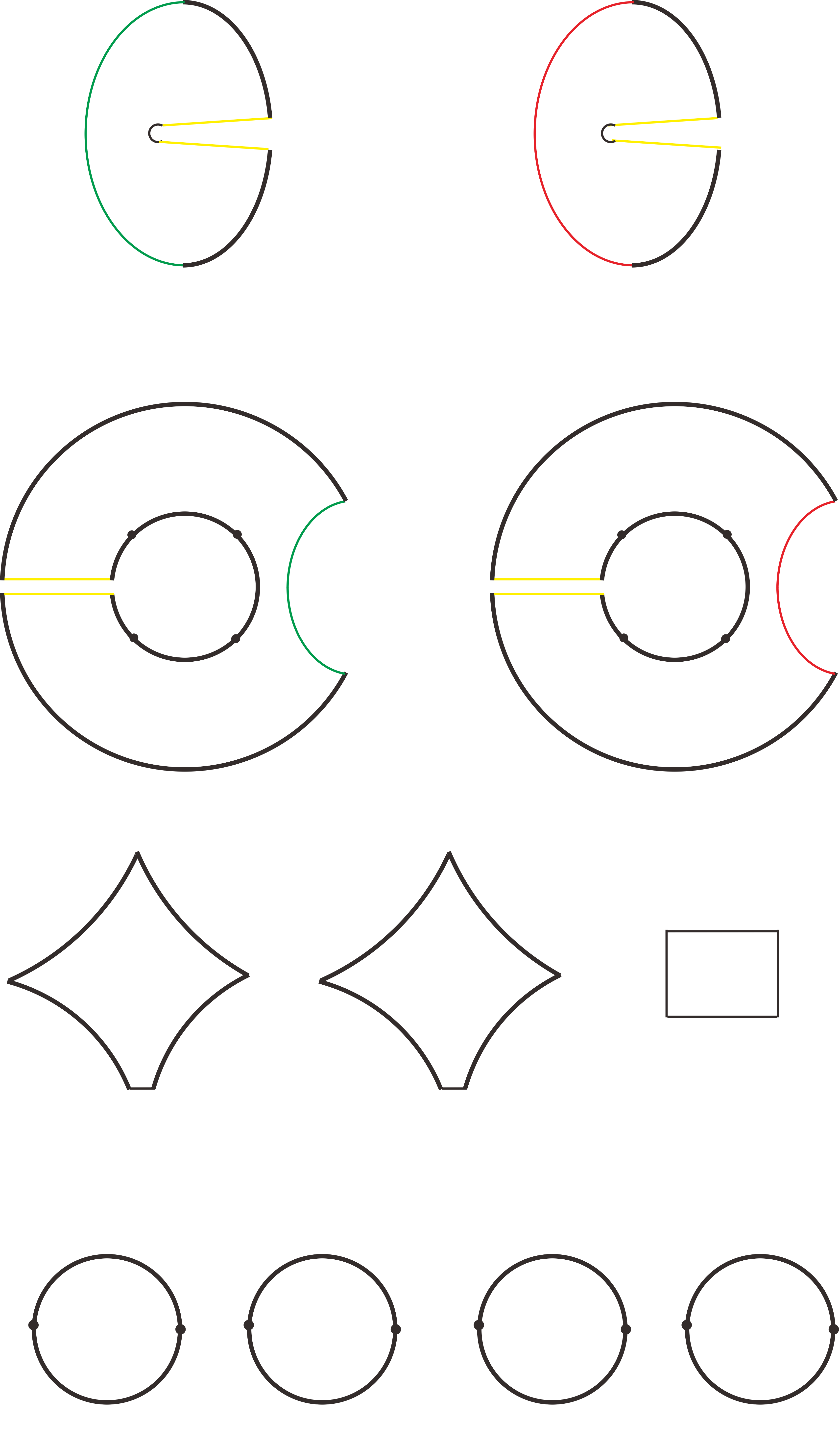}};
\coordinate [label=left:$ \scriptscriptstyle{B^{-1}}$] (S) at (-1.7,6.2);
\coordinate [label=left:$ \scriptscriptstyle{F(B)}$] (S) at (-2.5,5.2);
\coordinate [label=left:$ \scriptscriptstyle{B^{-1}}$] (S) at (-1.7,8.6);
\coordinate [label=left:$ \scriptscriptstyle{C(-)}$] (S) at (-2.3,7);
\coordinate [label=left:$ \scriptscriptstyle{C(+)}$] (S) at (-2.3,7.7);
\coordinate [label=left:$ \scriptscriptstyle{F}$] (S) at (-3,7.35);
\coordinate [label=left:$ \scriptscriptstyle{E}$] (S) at (-4.2,8.4);

\coordinate [label=left:$ \scriptscriptstyle{F(AB^{-1}A^{-1})}$] (S) at (4,5.2);
\coordinate [label=left:$ \scriptscriptstyle{ABA^{-1}}$] (S) at (4.8,6.2);
\coordinate [label=left:$ \scriptscriptstyle{ABA^{-1}}$] (S) at (4.8,8.6);
\coordinate [label=left:$ \scriptscriptstyle{C(-)}$] (S) at (3.9,7);
\coordinate [label=left:$ \scriptscriptstyle{C(+)}$] (S) at (3.9,7.7);
\coordinate [label=left:$ \scriptscriptstyle{F}$] (S) at (3.2,7.35);
\coordinate [label=left:$ \scriptscriptstyle{A(E)}$] (S) at (2.1,8.55);

\coordinate [label=left:$\scriptscriptstyle{F(B^{-1})}$] (S) at (-2,-1.2);
\coordinate [label=left:$\scriptscriptstyle{B}$] (S) at (-4.5,-0.4);
\coordinate [label=left:$\scriptscriptstyle{B}$] (S) at (-4.5,3.75);
\coordinate [label=left:$\scriptscriptstyle{C(-)}$] (S) at (-4.7,1.4);
\coordinate [label=left:$\scriptscriptstyle{C(+)}$] (S) at (-4.7,2);
\coordinate [label=left:$\scriptscriptstyle{ABA^{-1}}$] (S) at (-3.2,2);
\coordinate [label=left:$\scriptscriptstyle{ABA^{-1}}$] (S) at (-3.1,1.4);
\coordinate [label=left:$\scriptscriptstyle{B^{-1}AB^{-1}}$] (S) at (-2.3,2.9);
\coordinate [label=left:$\scriptscriptstyle{B^{-1}AB^{-1}}$] (S) at (-2.2,0.6);
\coordinate [label=left:$\scriptscriptstyle{AB^{-1}A^{-1}}$] (S) at (-1.6,1.7);
\coordinate [label=left:$\scriptscriptstyle{E}$] (S) at (-1.25,1);

\coordinate [label=left:$\scriptscriptstyle{F(ABA^{-1})}$] (S) at (4.8,-1.2);
\coordinate [label=left:$\scriptscriptstyle{AB^{-1}A^{-1}}$] (S) at (2.3,-0.4);
\coordinate [label=left:$\scriptscriptstyle{AB^{-1}A^{-1}}$] (S) at (2.3,3.75);
\coordinate [label=left:$\scriptscriptstyle{C(-)}$] (S) at (2.1,1.4);
\coordinate [label=left:$\scriptscriptstyle{C(+)}$] (S) at (2.1,2);
\coordinate [label=left:$\scriptscriptstyle{B^{-1}}$] (S) at (3.3,2);
\coordinate [label=left:$\scriptscriptstyle{B^{-1}}$] (S) at (3.3,1.4);
\coordinate [label=left:$\scriptscriptstyle{B^{-1}AB^{-1}}$] (S) at (4.4,2.8);
\coordinate [label=left:$\scriptscriptstyle{B^{-1}AB^{-1}}$] (S) at (4.4,0.6);
\coordinate [label=left:$\scriptscriptstyle{B^{-1}}$] (S) at (4.6,1.7);
\coordinate [label=left:$\scriptscriptstyle{A(E)}$] (S) at (5.85,1);

\coordinate [label=left:$\scriptscriptstyle{B^{-1}}$] (S) at (-4.5,-2.4);
\coordinate [label=left:$\scriptscriptstyle{B}$] (S) at (-4.6,-3.9);
\coordinate [label=left:$\scriptscriptstyle{ABA^{-1}}$] (S) at (-2,-2.4);
\coordinate [label=left:$\scriptscriptstyle{AB^{-1}A^{-1}}$] (S) at (-2,-3.9);
\coordinate [label=left:$\scriptscriptstyle{F}$] (S) at (-3.6,-4.35);
\coordinate [label=left:$\scriptscriptstyle{C(+)}$] (S) at (-3.3,-5.2);

\coordinate [label=left:$\scriptscriptstyle{B^{-1}}$] (S) at (-0.2,-2.4);
\coordinate [label=left:$\scriptscriptstyle{B}$] (S) at (-0.3,-3.9);
\coordinate [label=left:$\scriptscriptstyle{ABA^{-1}}$] (S) at (2.3,-2.4);
\coordinate [label=left:$\scriptscriptstyle{AB^{-1}A^{-1}}$] (S) at (2.3,-3.9);
\coordinate [label=left:$\scriptscriptstyle{F}$] (S) at (0.7,-4.35);
\coordinate [label=left:$\scriptscriptstyle{C(-)}$] (S) at (1,-5.2);

\coordinate [label=left:$\scriptscriptstyle{C(+)}$] (S) at (3.4,-3.15);
\coordinate [label=left:$\scriptscriptstyle{C(-)}$] (S) at (5.65,-3.15);
\coordinate [label=left:$\scriptscriptstyle{B}$] (S) at (4.4,-2.35);
\coordinate [label=left:$\scriptscriptstyle{AB^{-1}A^{-1}}$] (S) at (4.8,-3.9);
\coordinate [label=left:$\scriptscriptstyle{F}$] (S) at (4.5,-5.2);

\coordinate [label=left:$\scriptscriptstyle{B}$] (S) at (-4.1,-6.4);
\coordinate [label=left:$\scriptscriptstyle{B^{-1}}$] (S) at (-3.8,-8.6);
\coordinate [label=left:$\scriptscriptstyle{E}$] (S) at (-4.2,-9.4);

\coordinate [label=left:$\scriptscriptstyle{ABA^{-1}}$] (S) at (-0.4,-6.4);
\coordinate [label=left:$\scriptscriptstyle{AB^{-1}A^{-1}}$] (S) at (-0.3,-8.6);
\coordinate [label=left:$\scriptscriptstyle{A(E)}$] (S) at (-0.5,-9.4);

\coordinate [label=left:$\scriptscriptstyle{B^{-1}}$] (S) at (2.3,-6.4);
\coordinate [label=left:$\scriptscriptstyle{ABA^{-1}}$] (S) at (2.4,-8.6);
\coordinate [label=left:$\scriptscriptstyle{F(B^{-1}AB^{-1},1)}$] (S) at (3,-9.4);

\coordinate [label=left:$\scriptscriptstyle{B^{-1}}$] (S) at (5.1,-6.4);
\coordinate [label=left:$\scriptscriptstyle{ABA^{-1}}$] (S) at (5.2,-8.6);
\coordinate [label=left:$\scriptscriptstyle{F(B^{-1}AB^{-1},2)}$] (S) at (5.7,-9.4);

\end{tikzpicture}
\end{center}
  \caption{The eleven pieces of a 2-sphere $S$ for the Ford domain of  $\Delta_{3,4,\infty;\infty}$.}
	\label{figure:piece-2-sphere-34p}
\end{figure}

We consider a surface $S$ which is a union of eleven pieces, see Figure 	\ref{figure:piece-2-sphere-34p}:

\begin{itemize}
\item
the top two in Figure 	\ref{figure:piece-2-sphere-34p}
are the remaining disks of $I_{0}^{+}$ and $I_{1}^{-}$, we denote them by $F(B)$ and $F(AB^{-1}A^{-1})$ respectively, that is,  $F(B)=(I_{0}^{+} \cap H)\backslash C$ and  $F(AB^{-1}A^{-1})=(I_{1}^{-} \cap H)\backslash C$.  Each of $F(B)$ and $F(AB^{-1}A^{-1})$ is a hexagon. The boundary arcs of   $F(B)$  are labeled by $B^{-1}$, $E$, $B^{-1}$, $C(-)$, $F$ and $C(+)$ cyclically, this means that  $F(B)$  is glued to the disks $F(B^{-1})$, $E$, $F(B^{-1})$, $C(-)$, $F$ and $C(+)$ in Figure \ref{figure:piece-2-sphere-34p}  along these arcs. Similarly  $F(AB^{-1}A^{-1})$ is glued to the disks $F(ABA^{-1})$, $A(E)$, $F(ABA^{-1})$, $C(-)$, $F$ and $C(+)$ in Figure \ref{figure:piece-2-sphere-34p}  along its boundary.
	
\item the two disks in the second row of Figure 	\ref{figure:piece-2-sphere-34p}	are the remaining disks of $I_{0}^{-}$ and $I_{1}^{+}$, we denote them by $F(B^{-1})$ and $F(ABA^{-1})$ respectively, that is,  $F(B^{-1})=(I_{0}^{-} \cap H)\backslash C$ and  $F(ABA^{-1})=(I_{1}^{+} \cap H)\backslash C$.  Each of $F(B^{-1})$ and $F(ABA^{-1})$ is a decagon. $F(B^{-1})$ is glued to the disks $E$, $F(B)$,  $C(+)$, $F(ABA^{-1})$,  $F(B^{-1}AB^{-1})$, $F(AB^{-1}A^{-1})$, $F(B^{-1}AB^{-1})$, $F(ABA^{-1})$, $C(-)$ and $F(B)$ in Figure \ref{figure:piece-2-sphere-34p}  along its boundary.  Similarly  $F(ABA^{-1})$  is glued to ten disks  in Figure \ref{figure:piece-2-sphere-34p}.

\item 	The disks $C(+)$ and $C(-)$ in the third row of Figure 	\ref{figure:piece-2-sphere-34p} are two copies of the disk $C$ but with a vertex link of the vertex $p_{AB}$ deleted, so each of $C(+)$ and $C(-)$  is a pentagon. The disk $F$ in the third row of Figure 	\ref{figure:piece-2-sphere-34p} is a square, which is glued to  $C(+)$,  $C(-)$, $F(B)$ and $F(AB^{-1}A^{-1})$.

\item 	The disks $E$ and $A(E)$ in the fourth row of Figure 	\ref{figure:piece-2-sphere-34p} are glued to $F(B)$, $F(B^{-1})$ and $F(ABA^{-1})$, $F(AB^{-1}A^{-1})$ respectively. 	The disks $F(B^{-1}AB^{-1},1)$ and $F(B^{-1}AB^{-1},2)$ in the fourth row of Figure 	\ref{figure:piece-2-sphere-34p} are the remaining disks of $\widehat{I}_{0}^{-}$, that is,  $F(B^{-1}AB^{-1},1) \cup F(B^{-1}AB^{-1},2)=\widehat{I}_{0}^{-} \cap H$.
	
\end{itemize}

\begin{figure}
\begin{center}
\begin{tikzpicture}
\node at (0,0) {\includegraphics[width=9cm,height=9cm]{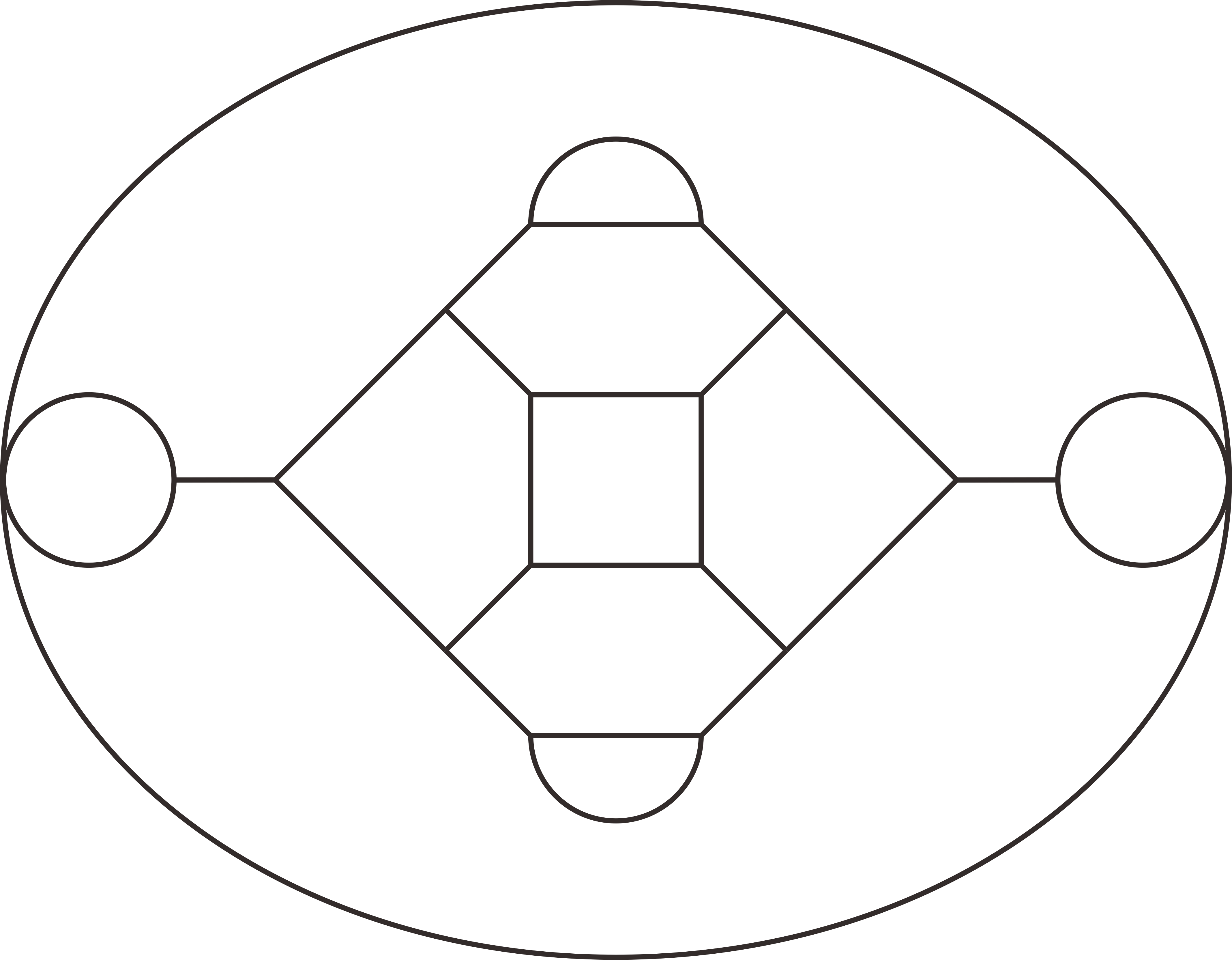}};
\coordinate [label=left:$\scriptstyle{F(B)}$] (S) at (0.4,1.5);
\coordinate [label=left:$\scriptstyle{E}$] (S) at (0.3,2.8);
\coordinate [label=left:$\scriptstyle{F(B^{-1})}$] (S) at (0.5,3.8);
\coordinate [label=left:$\scriptstyle{F}$] (S) at (0.3,0);
\coordinate [label=left:$\scriptstyle{F(AB^{-1}A^{-1})}$] (S) at (1,-1.5);
\coordinate [label=left:$\scriptstyle{A(E)}$] (S) at (0.4,-2.8);
\coordinate [label=left:$\scriptstyle{C(-)}$] (S) at (-1.2,0);
\coordinate [label=left:$\scriptstyle{C(+)}$] (S) at (2,0);
\draw[thick,->] (3.6,-0.5) -- (3.5,-1);
\coordinate [label=left:$\scriptstyle{F(B^{-1}AB^{-1},1)}$] (S) at (4,-1.2);
\draw[thick,->] (-3.5,-0.5) -- (-3.4,-0.9);
\coordinate [label=right:$ \scriptstyle{F(B^{-1}AB^{-1},2)}$] (S) at (-4,-1.1);
\coordinate [label=left:$\scriptstyle{F(ABA^{-1})}$] (S) at (0.6,-3.8);
\end{tikzpicture}
\end{center}
  \caption{Gluing all of eleven pieces in Figure \ref{figure:piece-2-sphere-34p}, we get a 2-sphere.}
	\label{figure:2-sphere-34p}
\end{figure}

We now glue all together  the eleven pieces to a surface $S$.
\begin{lem} \label{Lemma:2sphere-34p} The surface $S$ above is a 2-sphere.	
\end{lem}
\begin{proof}Even through we can show this lemma by the calculation of the Euler  characteristic of  $S$, but we prefer to show it via a picture. Figure \ref{figure:2-sphere-34p}	is a 2-disk,  we glue  the two  biggest (and outer) arcs which is one arc of the boundary of the disk labeled by $F(B^{-1})$ and $F(ABA^{-1})$ respectively, we get a 2-sphere. Now the gluing patterns of the eleven pieces are the same as the gluing pattern of Figure \ref{figure:piece-2-sphere-34p}, so  $S$ above is a two sphere.
\end{proof}

Note that the sphere $S$  is the boundary of a neighborhood of  $(U-H)\cup C$ in the Heisenberg group. In particular,  $S$ bounds a 3-ball containing $(U-H)\cup C$ in the Heisenberg group. Note that $S$ and $A(S)$ intersect in the disk $A(E)$, so  the two 3-balls bounded by $S$ and $A(S)$ intersect in the disk $A(E)$. Then a neighborhood of $$\cup^{\infty}_{k= -\infty} A^{k}((U-H)\cup C)$$ in the Heisenberg group is homeomorphic to $\mathbb{D}^{2} \times (-\infty, \infty)$. By Proposition \ref{prop:unknotted-34p} bellow, this $\mathbb{D}^{2} \times (-\infty, \infty)$ with $A$-action is unknotted in the Heisenberg group. So the complement of a neighborhood of $\cup^{\infty}_{k= -\infty} A^{k}((U-H)\cup C)$ in the Heisenberg space  is homeomorphic to   $(\mathbb{R}^2\backslash\mathbb{D}^{2})\times (-\infty, \infty)$, which is a genus one handlebody.  Then  $\partial _{\infty}D_{\Gamma_2}$ is an infinite genus handlebody with $A$-action, such that if we cut
$\partial _{\infty}D_{\Gamma}$ along the infinite many disjoint disks $\{A^{k}(C)\}^{\infty}_{k=-\infty}$, we get  $(\mathbb{R}^2\backslash\mathbb{D}^{2})\times(-\infty, \infty)$.

\begin{prop}\label{prop:unknotted-34p}
The affine line in Heisenberg coordinates $$L=\left\{\left[-\frac{3}{2}x+\frac{5}{4}, -\frac{\sqrt{7}}{2}x+\frac{\sqrt{7}}{5}, \frac{27\sqrt{7}}{20}x-\frac{3\sqrt{7}}{4}\right]\in \partial \hc; x\in \mathbb{R} \right\}$$ is preserved by the map $A$. Moreover $L$ is contained  in  the complement of $D_{\Gamma_{2}}$.
\end{prop}
\begin{proof}
	We denote by
$$L_{x}=\left[-\frac{3}{2}x+\frac{5}{4}, -\frac{\sqrt{7}}{2}x+\frac{\sqrt{7}}{5}, \frac{27\sqrt{7}}{20}x-\frac{3\sqrt{7}}{4}\right]$$ a point in $L$.
The parabolic map $A$ acts on $L$ as:
	$$A:L_{x}\rightarrow L_{x+1}.$$
	So $A$ acts on $L$ as a translation through 1. We claim that the segment on $L$ with parameter $x\in [0,1]$ is contained in $I_{0}^{+}\cup I_{-1}^{-}$.
	It is easy to check that the segment on $L$ with parameter $x\in [1/3,1]$ is contained in the interior of the spinal sphere $I_{0}^{+}$ and the segment on $L$ with parameter $x\in [0,1/3]$ in the interior of the spinal sphere $I_{-1}^{-}$. Therefore, the line $L$ is in  the complement of $D_{\Gamma_{2}}$.	
\end{proof}

Now we have
\begin{lem} \label{34pcombinatorialmodel}  $\partial _{\infty}D_{\Gamma_2}$ is an infinite genus handlebody with $A$-action,   such that  the  3-manifold $\partial _{\infty}D_{\Gamma_2}/ \Gamma_2$ at infinity  of $\Gamma_2$ is obtained from   $\partial _{\infty}D_{\Gamma_2}$ by side-pairings on $I_{k}^{+} \cap \partial _{\infty}D_{\Gamma_2}$,  $I_{k}^{-} \cap  \partial _{\infty}D_{\Gamma_2}$ and $\widehat{I}_{k}^{+} \cap \partial _{\infty}D_{\Gamma_2}$ for all $k \in \mathbb{Z}$.	
\end{lem}

In Figure 	\ref{figure:34pabstract}, we give a  combinatorial model of  $\partial _{\infty}D_{\Gamma_2}$.  Note that in Figure 	\ref{figure:34pabstract}, the plane $E^{*}$ is twisted (and we do not draw it), which is an infinite (topological) disk  intersects the punctured annuli labeled by $B$ and $B^{-1}$ both in one finite arc, and  is disjoint from all of the remaining punctured annuli.

We now show that the infinite polyhedron $\partial _{\infty}D_{\Gamma_2}$ is an infinite genus handlebody with $A$-action, and the region $H$ of  $\partial _{\infty}D_{\Gamma_2}$ co-bounded by $E^{*}$ and $A(E^{*})$ is a fundamental domain of $A$-action on $\partial _{\infty}D_{\Gamma_2}$,  $\partial _{\infty}D_{\Gamma_2}/ \Gamma_2$ is our 3-manifold at infinity of $\Gamma_2$. But the authors have difficulty to   study  $\partial _{\infty}D_{\Gamma_2}/ \Gamma_2$ directly  from the region $H$, by this we mean that, for example,  we can write down the definition equation for arc $E \cap I_{0}^{+}$ in the Heisenberg group, but it seems difficult to write down  the definition equation for  the $B$-action of  $E \cap I_{0}^{+}$, which is an arc in  $I_{0}^{-}$, the explicit  definition equation is helpful to get the   side-pairing maps. So we will cut out a fundamental domain $H'$ of the $A$-action on $\partial _{\infty}D_{\Gamma_2}$ in a new way in next subsection. We cut $\partial _{\infty}D_{\Gamma_2}$ in a geometrical way in the boundary of $\partial _{\infty}D_{\Gamma_2}$, that is in $\Sigma$, but in a topological way in
$\partial _{\infty}D_{\Gamma_2}\backslash \Sigma$. This fundamental domain $H'$ is a genus  three handlebdody (but $A(H')$ and $H'$ intersect in a two-holed plane). From $H'$, we can use side-pairing maps to show the quotient space from  $\partial _{\infty}D_{\Gamma_2}$ modulo $\Gamma_2$ is  the 3-manifold   $m295$ in Snappy  Census \cite{CullerDunfield:2014}.

\subsection{Cutting out the Ford domain of the complex hyperbolic triangle group  $\Delta_{3,4,\infty;\infty}$}\label{subsection:newcutting34p}

Recall the fours points $y$, $g$, $c$ and $r$, and also their $A$-translations.
In  Figure \ref{figure:34p1} we take small neighborhoods of the points $y$, $g$, $c$ and $r$ with yellow, green, cyan and red colors respectively.

We fix a point $u$ in  the intersection of the  spinal spheres of  $B$ and
$B^{-1}$, then we consider $B(u)$ and $B^{2}(u)$, they also lie in the intersection of the  spinal spheres of  $B$ and
$B^{-1}$.

We take
$$u=\left[\frac{1}{4}, -\frac{\sqrt{7}}{4}, \frac{\sqrt{15}}{2}\right]$$  in Heisenberg coordinates, then
$$B(u)=\left[\frac{7-\sqrt{105}}{16},-\frac{\sqrt{15}+7\sqrt{7}}{16},0\right],$$
in Heisenberg coordinates, and
$$B^2(u)=\left[\frac{1+\sqrt{105}}{16},\frac{\sqrt{15}-\sqrt{7}}{16},-\frac{\sqrt{15}}{2}\right]$$
in Heisenberg coordinates.

Figure \ref{figure:34pabstract} is a combinatorial picture of  the ideal boundary of the Ford domain of $\Delta_{3,4,\infty;\infty}$.
Each of the big annulus with a cusp is the intersection of the  spinal sphere of some $A^{k}BA^{-k}$ or
$A^{k}B^{-1}A^{-k}$ with the ideal boundary of the Ford domain.
There are also bigons  which are  the intersections of the  spinal sphere of some $A^{k}B^{-1}AB^{-1}A^{-k}$ or $A^{k}BA^{-1}BA^{-k}$ with the ideal boundary of the Ford domain.  For example, the two bigons with vertices $y$, $g$, and $c$, $r$   are  the intersection of the  spinal sphere of  $A^{-1}B^{-1}AB^{-1}A$  with the ideal boundary of the Ford domain (we add a
diameter in these bigons for future purposes, which are the colored arcs in Figure  \ref{figure:34pabstract}). In other words, the  spinal sphere of  $A^{-1}B^{-1}AB^{-1}A$ contributes two parts in the ideal boundary of the Ford domain of $\Delta_{3,4,\infty;\infty}$, this is due to the fact that $B^{-1}AB^{-1}$ and $A^{-1}BA^{-1}A$ have the same spinal sphere. The point $A^{-1}(p_{AB})$ is the tangent point of the  spinal spheres of  $A^{-1}BA$ and  $B^{-1}$,  and the point $p_{AB}$ is the tangent point of the  spinal spheres of  $AB^{-1}A^{-1}$ and  $B$.
The $A$-action is the horizontal translation with a half-turn to the right. The ideal boundary of the Ford domain of $\Delta_{3,4,\infty;\infty}$ is the region which is  out side all the spheres in  Figure \ref{figure:34pabstract}. Note that $A(A^{-1}(p_{AB}))=p_{AB}$ and  $B(p_{AB})=A^{-1}(p_{AB})$.

We take a fundamental domain of $A$-action on the ideal boundary of the Ford domain of $\Delta_{3,4,\infty;\infty}$. That is, we take a topological  holed-plane $E'$ in $\partial_{\infty}{\bf H}^2_{\mathbb C}$, such that $\{q_{\infty}, A^{-1}(p_{AB}) \}\subset E'$, and the boundary of the holed-plane $E'$ is exactly the
purple-colored circle in Figure \ref{figure:34pabstract} (one of the thickest circle with vertices labeled by $y$, $g$, $c$ and $r$). Unfortunately, this holed-plane $E'$ is not a  subsurface of  $E^{*}$ in Subsection \ref{subsection:globalmodel34p}. The key point is that   the boundary of $E'$ is geometric, but on the other hand the entire $E'$ is topological. We can not find a   piecewise geometrical meaningful  plane which pass through the boundary of $E'$ and $A^{-1}(p_{AB})$
simultaneously, but $E'$ above with geometrical  boundary (and passing through  $A^{-1}(p_{AB})$) is enough for the using of side-pairing map.
Then $A(E')$ is a holed-plane in $\partial_{\infty}{\bf H}^2_{\mathbb C}$, such $\{q_{\infty}, p_{AB}\}\subset A(E')$, and the boundary of the disk $A(E')$ is exactly the orange-colored circle in  Figure \ref{figure:34pabstract} (one of the thickest circle with vertices labeled by $A(c)$, $A(r)$, $A(y)$ and $A(g)$). We may also assume $E' \cap A(E')=\emptyset$.

In  Figure \ref{figure:34pabstract}, there are three arcs with  end vertices $y$ and $g$:
\begin{itemize}
	\item  We denote by $[y,g]_1$ the purple one. The arc  $[y,g]_1$  is one component of the intersection of the spinal spheres of $A^{-1}B^{-1}A$ and $A^{-1}B^{-1}AB^{-1}A$;
	\item   We denote by  $[y,g]_2$ the pink  one.  The arc  $[y,g]_2$  lies in the spinal sphere of $A^{-1}B^{-1}AB^{-1}A$;
	\item   We denote by  $[y,g]_3$ the black one.   The arc $[y,g]_3$  is one component of the intersection of the spinal spheres of $B$ and $A^{-1}B^{-1}AB^{-1}A$.
	\end{itemize} 
		 Similarly,   there are three arcs with  end vertices $c$ and $r$:
		 \begin{itemize}
		 	\item 
		 	 We denote by $[c,r]_1$ the black arc with ending points $c$ and $r$. It is one component of the intersection of the spinal spheres of $A^{-1}B^{-1}A$ and $A^{-1}B^{-1}AB^{-1}A$;	\item    We denote $[c,r]_2$ the black arc with end vertices $c$ and $r$. It lies in the spinal sphere of  $A^{-1}B^{-1}AB^{-1}A$;
		 	 	\item     We denote  $[c,r]_3$ the purple arc with end vertices $c$ and $r$. It is one component of the intersection of the spinal spheres of $B$ and $A^{-1}B^{-1}AB^{-1}A$.
		 	 \end{itemize}
		 	 Moreover, 	 there are three arcs with  end vertices $A(c)$ and $A(r)$:
		 	  \begin{itemize}
		 	 	\item 
		 	   We denote by $[A(c),A(r)]_1$ the black arc with end vertices $A(c)$ and $A(r)$. It  is one component of the intersection of the spinal spheres of $B^{-1}$ and $B^{-1}AB^{-1}$;
		 	      	\item     We denote $[A(c),A(r)]_2$ the pink  arc with end vertices $A(c)$ and $A(r)$.   It lies in the spinal sphere of  $B^{-1}AB^{-1}$;
		 	      	   	\item    We denote  $[A(c),A(r)]_3$ the orange  arc with end vertices $A(c)$ and $A(r)$. It  is one component of the intersection of the spinal spheres of $ABA^{-1}$ and $B^{-1}AB^{-1}$.
		 	    \end{itemize}
	For the  three arcs with  end vertices $A(g)$ and $A(y)$:
		 	    \begin{itemize}
		 	   	\item 
We  denote by $[A(g),A(y)]_1$ the orange arc with end vertices $A(g)$ and $A(y)$. It  is one component of the intersection of the spinal spheres of $B^{-1}$ and $B^{-1}AB^{-1}$;
   	\item  
   We denote $[A(g),A(y)]_2$ the black arc with end vertices $A(g)$ and $A(y)$. It lies in the spinal sphere of  $B^{-1}AB^{-1}$;
   	\item     We denote  $[A(r),A(y)]_3$ the black  arc with end vertices $A(y)$ and $A(y)$. It  is one component of the intersection of the spinal spheres of $ABA^{-1}$ and $B^{-1}AB^{-1}$.
 \end{itemize}

It is easy to show that  $B^{-1}AB^{-1}(F_{2,-}\cup F_{2,+})=(F_{2,-}\cup F_{2,+})$ is the bigon in Figure  \ref{figure:34pabstract} with boundary consists of $[c,r]_1$ and $[c,r]_3$, so we have
$$A(B^{-1}AB^{-1}(F_{2,-}\cup F_{2,+}))=F_{2,-}\cup F_{2,+},$$ but more precisely we  can show that
$$B^{-1}AB^{-1}([A(c), A(r)]_3)=[c,r]_1$$
and
$$B^{-1}AB^{-1}([A(c), A(r)]_1)=[c,r]_3.$$

Now  the part  of  the ideal boundary of the Ford domain of  $\Delta_{3,4,\infty;\infty}$  co-bounded by $E'$  and  $A(E')$, is denoted by $H'$, we now show $H'$ is a genus 3 handlebody.
The  boundary of $H'$ consists of  $E'$, $A(E')$, and also parts of  the spinal  spheres of $B$,  $B^{-1}$,  $A^{-1}B^{-1}AB^{-1}A$, $B^{-1}AB^{-1}$,   see
Figure \ref{figure:34pdisk}.
Now in Figure \ref{figure:34pdisk}, we take three disks $D_1$, $D_2$ and $D_3$ carefully which cut the  $H'$ into a 3-ball:
\begin{itemize}
\item The boundary of the  disk   $D_1$   consists of three arcs $[A^{-1}(p_{AB}),c]$, $[c,u]$ and $[u,A^{-1}(p_{AB})]$. The arc  $[A^{-1}(p_{AB}),c]$  is in $E'$,  the arc $[c,u]$ is in the spinal sphere of $B$ and the arc  $[u,A^{-1}(p_{AB})]$ is  in the spinal sphere of $B^{-1}$. These are the blue-colored arcs in  Figure \ref{figure:34pdisk}. We may also assume   the arc $[c,u]$ does not intersect  the spinal sphere of $B^{-1}$.
\item
The boundary of the  disk   $D_2$    consists of three arcs $[p_{AB}, A(c)]$, $[A(c),B^{2}(u)]$ and $[B^{2}(u), p_{AB}]$. The arc $[p_{AB}, A(c)]$ is in $A(E')$,  the arc $[A(c),B^{-1}(u)]$ is in the spinal sphere of $B^{-1}$ and the arc $[B^{-1}(u), p_{AB}]$ is in the spinal sphere of  $B$.  These are the red-colored arcs in  Figure \ref{figure:34pdisk}. We may also assume   the arc  $[B^{2}(u), p_{AB}]$ does not intersect  the spinal sphere of $B^{-1}$.
\item
The boundary of the  disk   $D_3$  consists of six arcs $[q_{\infty}, y]$, $[y,r]$, $[r, B(u)]$,  $[B(u),A(g)]$,   $[A(g), A(y)]_1$ and $[A(y), q_{\infty}]$. The arc  $[q_{\infty}, y]$ lies in the disk $E$,  the arc  $[y,r]$ lies in the intersection of  spinal spheres of $B$ and $A^{-1}B^{-1}A$,  the arc  $[r, B(u)]$ lies in the spinal sphere of $B$, the arc  $[B(u),A(g)]$ lies in the   spinal sphere of $B^{-1}$,  the arc  $[A(g), A(y)]_1$  lies in the intersection of   spinal spheres of $B^{-1}$ and $B^{-1}AB^{-1}$,  $[A(y), q_{\infty}]$ lies in the disk $A(E')$.
	These are the green-colored arcs in  Figure \ref{figure:34pdisk}.
\end{itemize}

We may also assume that $B([r,B(u)])=[A(c),B^{2}(u)]$,   $B([c,u])=[A(g),B(u)]$,   $B([p_{AB},B^{2}(u)])=[A^{-1}(p_{AB}),u]$  $A([y,q_{\infty}])=[A(y),q_{\infty}]$ and  $A([c,A^{-1}(p_{AB})])=[A(c),p_{AB}]$.

We now need more arcs:  $u$, $B(u)$, $B^2(u)$ divide the intersection circle of the spinal spheres of $B$ and $B^{-1}$ into three arcs $[u, B(u)]$, $[B(u), B^2(u)]$ and $[B^2(u),u]$.

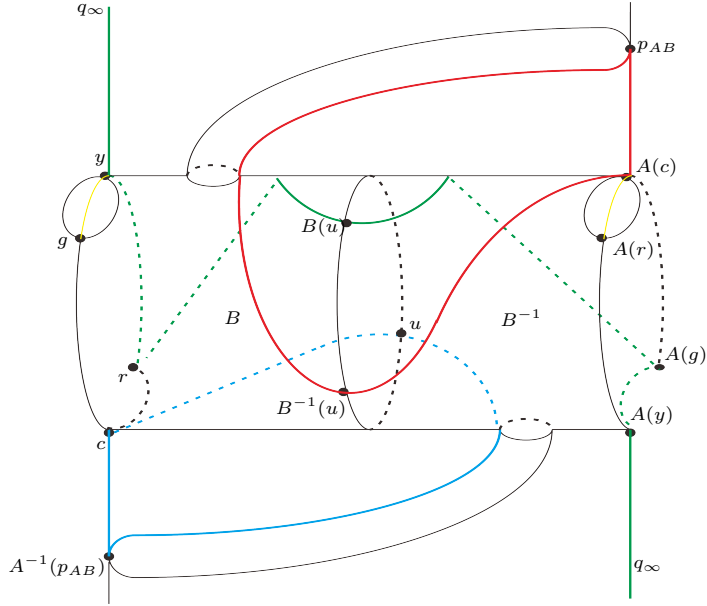
\begin{figure}
	\begin{center}
		\begin{tikzpicture}
		\node at (0,0) {\includegraphics[width=8cm,height=8cm]{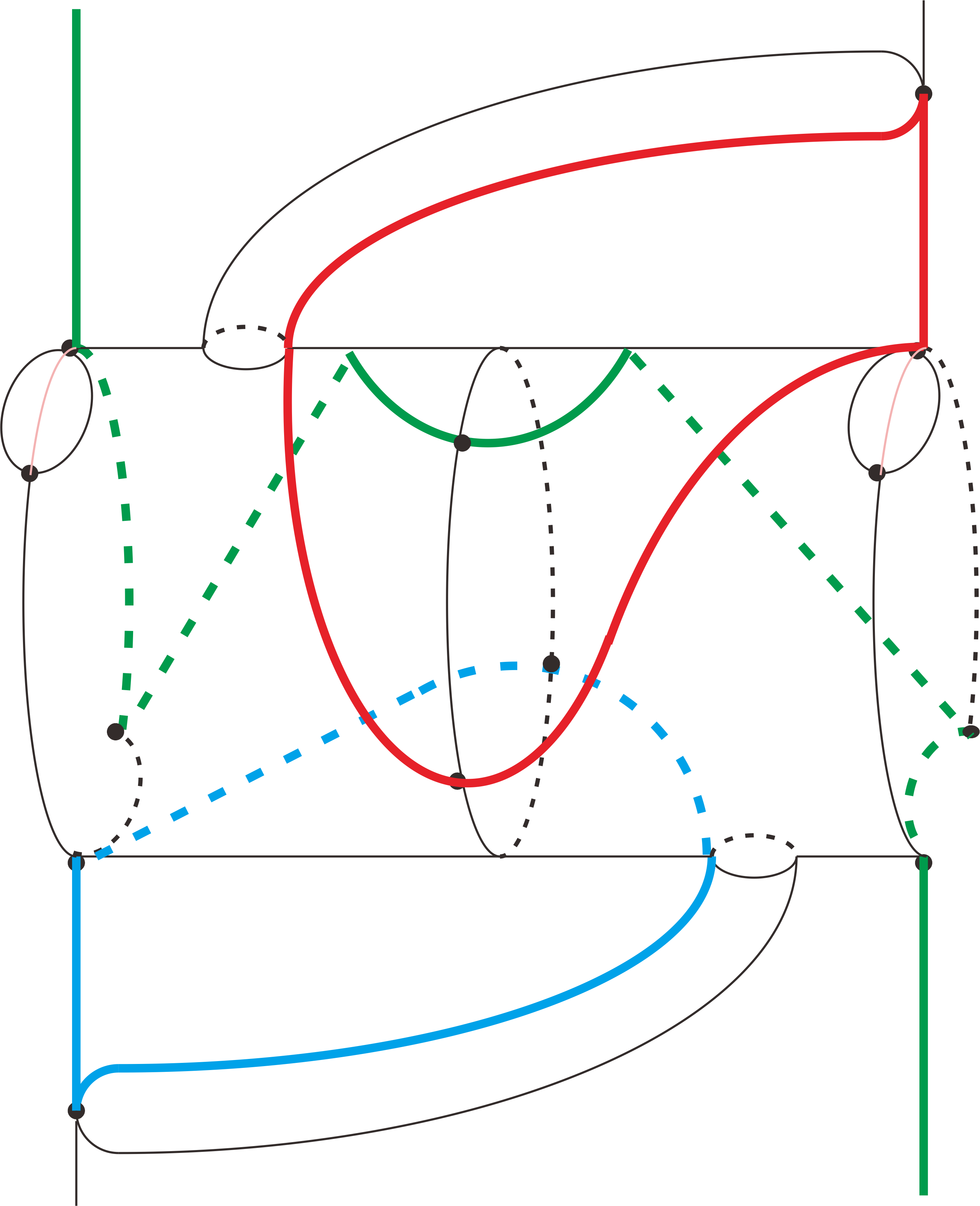}};
		\coordinate [label=left:$\scriptstyle{B(u)}$] (S) at (-0.1,1.0);
		\coordinate [label=left:$\scriptstyle{B^{-1}(u)}$] (S) at (-0.1,-1.4);
		\coordinate [label=left:$\scriptstyle{u}$] (S) at (0.9,-0.3);
		\coordinate [label=left:$\scriptstyle{B^{-1}}$] (S) at (2.5,-0.2);
		\coordinate [label=left:$\scriptstyle{B}$] (S) at (-1.5,-0.2);
		\coordinate [label=left:$\scriptstyle{r}$] (S) at (-3,-1);
		\coordinate [label=left:$\scriptstyle{c}$] (S) at (-3.3,-1.9);
		\coordinate [label=left:$\scriptstyle{A^{-1}(p_{AB})}$] (S) at (-3.3,-3.5);
		\coordinate [label=left:$\scriptstyle{g}$] (S) at (-3.8,0.8);
		\coordinate [label=left:$\scriptstyle{y}$] (S) at (-3.3,1.9);
		\coordinate [label=left:$\scriptstyle{q_{\infty}}$] (S) at (-3.3,3.9);
		\coordinate [label=left:$\scriptstyle{p_{AB}}$] (S) at (4.3,3.4);
		\coordinate [label=left:$\scriptstyle{A(c)}$] (S) at (4.3,1.8);
		\coordinate [label=left:$\scriptstyle{A(r)}$] (S) at (4,0.7);
		\coordinate [label=left:$\scriptstyle{A(g)}$] (S) at (4.7,-0.7);
		\coordinate [label=left:$\scriptstyle{A(y)}$] (S) at (4.25,-1.5);
		\coordinate [label=left:$\scriptstyle{q_{\infty}}$] (S) at (4.1,-3.5);
		\end{tikzpicture}
	\end{center}
	\caption{Cutting disks of a fundamental domain $H'$ for the $\langle A \rangle$-action on the ideal boundary of  the Ford domain of $\Delta_{3,4,\infty;\infty}$. They are the disks with thick red, blue and green-colored boundaries.}
	\label{figure:34pdisk}
\end{figure}

As in Subsection \ref{subsection:globalmodel34p}, we now cut $H'$ alone the three disks $D_{i}$ for $i=1,2,3$, we get a 3-ball, so $H'$ is a genus three handlebody.
In Figure \ref{figure:34ppolytope}, we give  a picture of the boundary of the polytope we obtained from $H'$ by cutting along the disks $D_{i}$ for $i=1,2,3$.
What is shown in Figure \ref{figure:34ppolytope} is a disk, but we glue the pairs of edges labeled  with ending vertices $q_{\infty}$ and $A(y)$, we get a new disk, then  we glue the pairs of edges labeled  with ending vertices $q_{\infty}$ and $y$, we get a 2-sphere.

We note that  for each   disk $D_{i}$, now there are two copies of them in the boundary of the 3-ball, we denote them by $D_{i,+}$ and $D_{i,-}$.  $D_{i,-}$ is the copy of $D_{i}$  which is closer to us in Figure \ref{figure:34pdisk}, and  $D_{i,+}$ is the copy of $D_{i}$  which is further away from us in Figure \ref{figure:34pdisk}.
Cutting the  2-disk $E'$ along arcs $[q_{\infty}, y]$ and  $[c,A^{-1}(p_{AB})]$  we get a 2-disk $E_{1}$. Cutting the  2-disk $A(E')$ along arcs $[q_{\infty}, A(y)]$ and  $[A(c),A^{-1}(p_{AB})]$  we get a 2-disk $E_{2}$.

The intersection of  the spinal sphere  of $B$ and the boundary of $H'$ is an annulus with one boundary consists of $[u, B(u)]$, $[B(u), B^2(u)]$ and $[B^2(u),u]$, another boundary  consists of $[y, r]$, $[r, c]$,$[c,g]_3$  and $[g,y]_3$ and one cusp $p_{AB}$. Now this 2-disk is divided into two 2-disks by $[B(u), r]$, $[c,u]$ and $[p_{AB},B^2(u)]$. One of them is a quadrilateral, we denote it by $B_4$ in Figure \ref{figure:34ppolytope}, another one is a
nonagon  and we denote it by $B_9$ in Figure \ref{figure:34ppolytope}.
The intersection of  the spinal sphere  of $B^{-1}$ and the boundary of $H'$ is an annulus with one boundary consists of $[u, B(u)]$, $[B(u), B^2(u)]$ and $[B^2(u),u]$, another boundary  consists of $[A(y), A(r)]$,  $[A(r), A(c)]_1$ $[A(c), A(g)]$  and  $[A(g), A(y)]_1$  and one cusp $A^{-1}(p_{AB})$. Now this 2-disk is divided into two 2-disks by $[B(u), r]$, $[c,u]$ and $[p_{AB},B^2(u)]$. One of them is a quadrilateral, we denote it by $B^{-1}_4$ in Figure \ref{figure:34ppolytope}, another one is a
nonagon  and we denote it by $B^{-1}_9$ in Figure \ref{figure:34ppolytope}.
Cutting  $E'$ alone arcs $[q_{\infty}, y]$ and  $[c,A^{-1}(p_{AB})]$  we get a disk  $E_{1}$, and cutting  $A(E')$ alone arcs $[q_{\infty}, A(y)]$ and  $[A(c),p_{AB}]$  we get a disk  $E_{2}$.

Now $[y,g]_2$ divides the bigon with vertices $y$ and $g$ in 	Figure \ref{figure:34pdisk} into two disks $F_{1,-}$  and  $F_{1,+}$, where $F_{1,-}$  is the  bigon  with boundary $[y,g]_1 \cup [y,g]_2$, and  $F_{1,+}$  is the  bigon  with boundary $[y,g]_2 \cup [y,g]_3$. Similarly $[A(c),A(r)]_2$ divides the bigon with vertices $A(c)$ and $A(r)$ in 	Figure \ref{figure:34pdisk} into two disks $F_{2,-}$  and  $F_{2,+}$, where $F_{2,-}$  is the  bigon  with boundary $[A(c),A(r)]_1 \cup [A(c),A(r)]_2$, and  $F_{2,+}$  is the  bigon  with boundary $[A(c),A(r)]_2 \cup [A(c),A(r)]_3$.

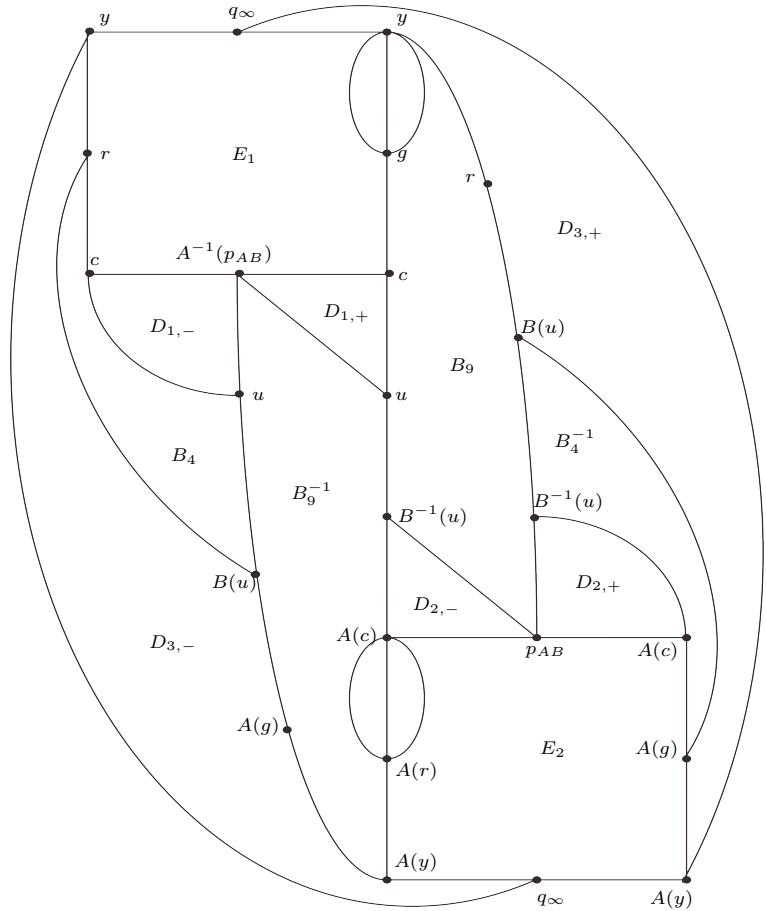
\begin{figure}
\begin{center}
\begin{tikzpicture}
\node at (0,0) {\includegraphics[width=10cm,height=12cm]{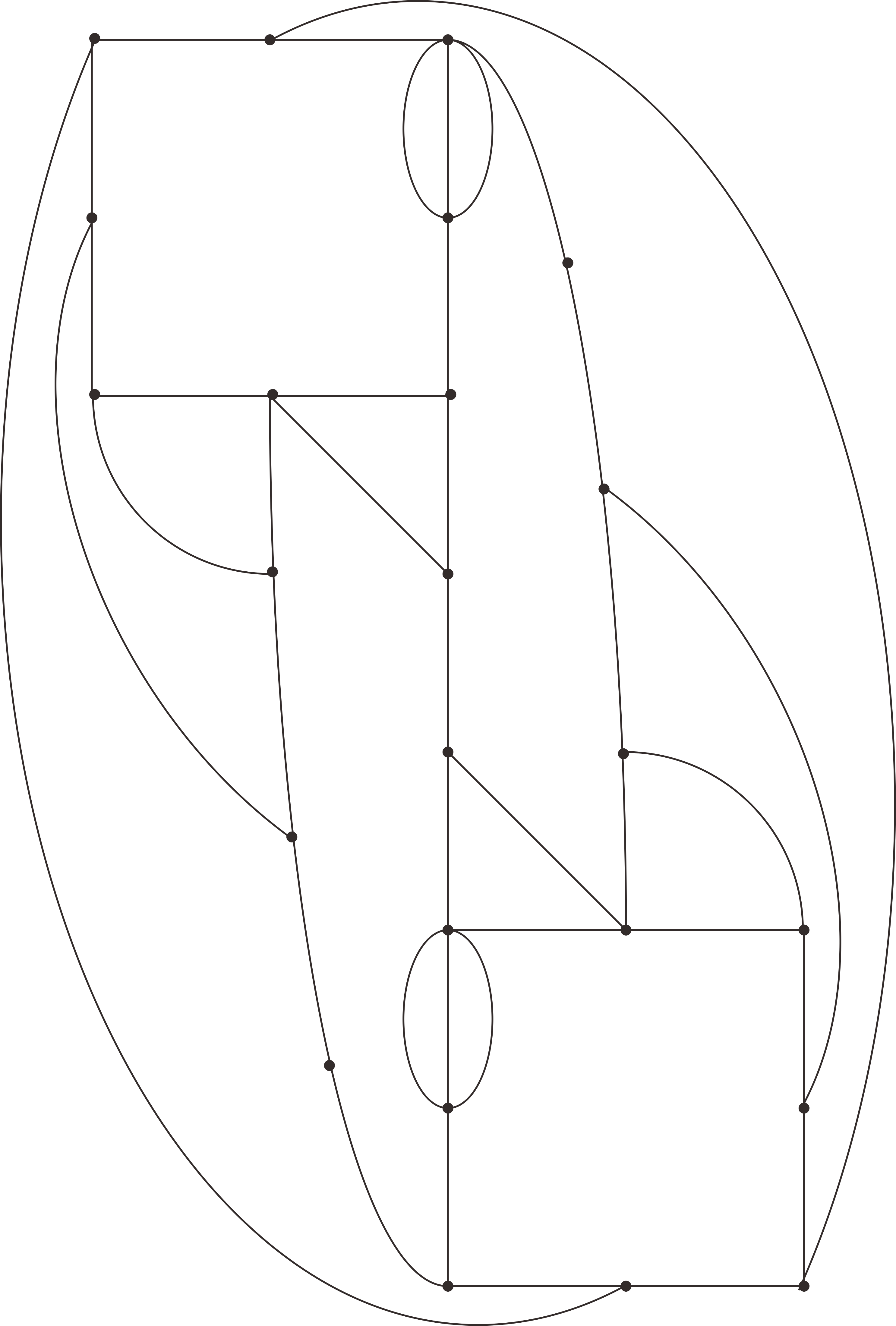}};
\coordinate [label=left:$\scriptstyle{u}$] (S) at (0.4,0.8);
\coordinate [label=left:$\scriptstyle{u}$] (S) at (-1.5,0.8);
\coordinate [label=left:$\scriptstyle{c}$] (S) at (0.4,2.4);
\coordinate [label=left:$\scriptstyle{A^{-1}(p_{AB})}$] (S) at (-1.4,2.7);
\coordinate [label=left:$\scriptstyle{c}$] (S) at (-3.7,2.6);
\coordinate [label=left:$\scriptstyle{D_{1,+}}$] (S) at (-0.1,1.9);

\coordinate [label=left:$\scriptstyle{g}$] (S) at (0.4,4);
\coordinate [label=left:$\scriptstyle{y}$] (S) at (0.4,5.8);
\coordinate [label=left:$\scriptstyle{B^{-1}(u)}$] (S) at (1.2,-0.8);
\coordinate [label=left:$\scriptstyle{D_{2,-}}$] (S) at (1.1,-2);
\coordinate [label=left:$\scriptstyle{ B^{-1}_9}$] (S) at (-0.6,-0.5);
\coordinate [label=left:$\scriptstyle{A(c)}$] (S) at (0,-2.4);

\coordinate [label=left:$\scriptstyle{A(r)}$] (S) at (0.8,-4.2);
\coordinate [label=left:$\scriptstyle{A(y)}$] (S) at (0.8,-5.4);
\coordinate [label=left:$\scriptstyle{q_{\infty}}$] (S) at (2.5,-5.9);
\coordinate [label=left:$\scriptstyle{E_{2}}$] (S) at (2.5,-3.9);
\coordinate [label=left:$\scriptstyle{A(y)}$] (S) at (4.2,-5.9);
\coordinate [label=left:$\scriptstyle{A(g)}$] (S) at (4,-3.9);
\coordinate [label=left:$\scriptstyle{A(c)}$] (S) at (4,-2.6);
\coordinate [label=left:$\scriptstyle{p_{AB}}$] (S) at (2.5,-2.6);
\coordinate [label=left:$\scriptstyle{D_{2,+}}$] (S) at (3.25,-1.7);
\coordinate [label=left:$\scriptstyle{B^{-1}(u)}$] (S) at (3,-0.6);
\coordinate [label=left:$\scriptstyle{ B^{-1}_4}$] (S) at (2.9,0.2);
\coordinate [label=left:$\scriptstyle{B_9}$] (S) at (1.3,1.2);
\coordinate [label=left:$\scriptstyle{B(u)}$] (S) at (2.5,1.7);
\coordinate [label=left:$\scriptstyle{r}$] (S) at (1.3,3.7);
\coordinate [label=left:$\scriptstyle{D_{3,+}}$] (S) at (3,3);
\coordinate [label=left:$\scriptstyle{E_{1}}$] (S) at (-1.6,4);
\coordinate [label=left:$\scriptstyle{q_{\infty}}$] (S) at (-1.6,5.9);

\coordinate [label=left:$\scriptstyle{D_{1,-}}$] (S) at (-2.4,1.7);
\coordinate [label=left:$\scriptstyle{B_{4}}$] (S) at (-2.4,0);
\coordinate [label=left:$\scriptstyle{D_{3,-}}$] (S) at (-2.4,-2.5);
\coordinate [label=left:$\scriptstyle{B(u)}$] (S) at (-1.6,-1.7);
\coordinate [label=left:$\scriptstyle{A(g)}$] (S) at (-1.3,-3.6);
\coordinate [label=left:$\scriptstyle{r}$] (S) at (-3.55,4);
\coordinate [label=left:$\scriptstyle{y}$] (S) at (-3.55,5.8);

\end{tikzpicture}
\end{center}
 \caption{The boundary of the polytope for  $\Delta_{3,4,\infty;\infty}$.}
	\label{figure:34ppolytope}
\end{figure}

If we  denote the 3-manifold at infinity of the even subgroup $\Gamma_2=\langle I_1I_2, I_2I_3 \rangle$ of  $\Delta_{3,4,\infty;\infty}$ by $M$, then $M$ is just the quotient space of the 3-ball obtained from $H'$ cutting alone$D_1 \cup D_2 \cup D_3$, then the side-pairings are
\begin{eqnarray*}
	\gamma_1&: &F_{1,-}\longrightarrow F_{1,+};\\
	\gamma_2&: &F_{2,-}\longrightarrow F_{2,+};\\
	\alpha&:&E_1 \longrightarrow E_2;\\
	\delta_1&:& D_{1,-}\longrightarrow D_{1,+};\\
	\delta_2&: &D_{2,-}\longrightarrow D_{2,+};\\
	\delta_3&: &D_{3,-}\longrightarrow D_{3,+};\\
		\beta_4&: &B_{4}\longrightarrow  B^{-1}_{4}; \\
	\beta_9&: &B_{9}\longrightarrow B^{-1}_{9}.
\end{eqnarray*}
 $\gamma_1$ is the homeomorphism  from  $F_{1,-}$ to $F_{1,+}$ such that $\alpha_1(y)=y$,  $\gamma_1(g)=g$; similarly we have $\gamma_2$;  $\alpha$ is the  orientation-reversing  homeomorphism  from $E_{1}$ to  $E_{2}$ such that $\alpha(g)=A(g)$,  $\alpha(r)=A(r)$,  $\alpha(q_{\infty})=q_{\infty}$,  $\alpha(y)=A(y)$ and  $\alpha(c)=A(c)$, here  we give  orientations of both  $E_{1}$  and  $E_{2}$ from an orientation of the 3-ball,  so $\alpha$ is uniquely determined ;  $\delta_
{i}$ is the homeomorphism  from  $D_{i,-}$ to $D_{i,+}$ which  preserves the labels on vertices for $i=1,2,3$; $\beta_4$  is the homeomorphism  from $B_{4}$ to $B^{-1}_{4}$  such that  $\beta_4(c)=A(g)=B(c)$,  $\beta_4(r)=A(c)=B(r)$,  $\beta_4(B(u))=B^2(u)$ and  $\beta_4(u)=B(u)$; $\beta_9$  is the homeomorphism  from $B_{9}$ to $B^{-1}_{9}$  such that  $\beta_9(p_{AB})=A^{-1}(p_{AB})$,  $\beta_9(B^2(u))=u$,  $\beta_9(B(u))=B^2(u)$, $\beta_9(r)=A(c)=B(r)$, $\beta_9(y)=A(r)=B(y)$,
$\beta_9(g)=A(y)=B(g)$, $\beta_9(c)=A(g)=B(c)$, $\beta_9(u)=B(u)$ and  $\beta_9(B^2(u))=u$.

We write down the  ridge circles for the 3-manifold $M$  in Table \ref{table:edgecircle34p}.

\begin{table}[!htbp]
\caption{Ridge Cycles of the 3-manifold at infinity of  of $\Delta_{3,4,\infty;\infty}$.}
  \centering
\begin{tabular}{c|c}
\toprule
\textbf{Ridge} & \textbf{Ridge Cycle  } \\
\midrule
$e_{1}$ & $E_1\cap F_{1,-}\xrightarrow{\gamma_1}  F_{1,+}\cap B_9 \xrightarrow{\beta_9}  B^{-1}_9\cap E_2 \xrightarrow{\alpha^{-1}}  E_1\cap D_{3,-} \xrightarrow{\delta_3}   D_{3,+}\cap B_9$ \\
		& $	\xrightarrow{\beta_9}   B^{-1}_9\cap  F_{2,-}  \xrightarrow{\gamma_2}  F_{2,+}\cap E_2  \xrightarrow{\alpha^{-1}}  E_1 \cap B_4  \xrightarrow{\beta_4}  B^{-1}_4\cap E_2$  \\
		& $ \xrightarrow{\alpha^{-1}}  E_1 \cap B_9     \xrightarrow{\beta_9}  B^{-1}_9 \cap D_{3,-}     \xrightarrow{\delta_3}  D_{3,+} \cap E_2  \xrightarrow{\alpha^{-1}}  E_1\cap F_{1,-} $  \\
\hline		
$e_{2}$ & $B_9\cap B^{-1}_4\xrightarrow{\beta_9^{-1}} B_4\cap B_9 \xrightarrow{\beta_4^{-1}} B_9\cap B^{-1}_9 \xrightarrow{\beta_9^{-1}} B_9\cap B^{-1}_4$\\
\hline	
$e_{3}$ & $D_{1,+}\cap B_9\xrightarrow{\delta_1} B^{-1}_9\cap D_{3,-} \xrightarrow{\beta_9} D_{3,+}\cap B^{-1}_4 \xrightarrow{\delta_3} B_4\cap D_{1,-} \xrightarrow{\beta_4^{-1}}  D_{1,+}\cap B_9$  \\
\hline		
$e_{4}$ & $D_{1,+}\cap E_{1}\xrightarrow{\delta_1} E_{2}\cap D_{2,+} \xrightarrow{\alpha} D_{2,-}\cap E_{2} \xrightarrow{\delta_2^{-1}} E_{1}\cap D_{1,-} \xrightarrow{\alpha^{-1}} D_{1,+}\cap E_{1}$ \\
\hline		
$e_{5}$ & $D_{3,+}\cap E_{1}\xrightarrow{\delta_3} E_{2}\cap D_{3,+} \xrightarrow{\alpha} D_{3,-}\cap E_{2} \xrightarrow{\delta_3^{-1}} E_1\cap D_{3,-} \xrightarrow{\alpha^{-1}} D_{3,+}\cap E_{1}$  \\
\hline		
$e_{6}$ & $B^{-1}_9\cap D_{1,+}\xrightarrow{\delta_1^{-1}} D_{1,-}\cap B^{-1}_{9} \xrightarrow{\beta_9^{-1}} B_{9}\cap D_{2,-} \xrightarrow{\delta_2} D_{2,+}\cap B_{9}\xrightarrow{\beta_9} B^{-1}_9\cap D_{1,+}$  \\
\hline		
$e_{7}$ & $B_9 \cap D_{3,+}\xrightarrow{\beta_9^{-1}} D_{3,-}\cap B_4 \xrightarrow{\delta_3^{-1}} B^{-1}_4 \cap D_{2,+} \xrightarrow{\beta_4} D_{2,-}\cap B^{-1}_9 \xrightarrow{\delta_2^{-1}} B_9 \cap D_{3,+}$ \\
\hline		
$e_{8}$ & $F_{1,-}\cap F_{1,+}\xrightarrow{\gamma_1} F_{1,-}\cap F_{1,+}$ \\
\hline		
$e_{9}$ & $F_{2,-}\cap F_{2,+}\xrightarrow{\gamma_2}F_{2,-}\cap F_{2,+}$ \\
\bottomrule
\end{tabular}
\label{table:edgecircle34p}
\end{table}

\subsection{The 3-manifold at infinity of $\Delta_{3,4, \infty;\infty}$} \label{subsection:34pmanifoldfinal}

From Table	\ref{table:edgecircle34p}, we get Table  \ref{table:relation34p}, then we get a presentation of the fundamental group of  the 3-manifold $M$,  which is a group $\pi_{1}(M)$ with eight generators $\alpha$,  $\gamma_1$, $\gamma_2$, $\delta_1$, $\delta_2$, $\delta_3$, $\beta_4$, $\beta_9$ and nine relations in Table \ref{table:relation34p}.

\begin{table}[!htbp]
\caption{Cycle relation  of the 3-manifold at infinity of  of $\Delta_{3,4,\infty;\infty}$.}
  \centering
\begin{tabular}{c|c}
\toprule
\textbf{Ridge} & \textbf{Ridge relation  } \\
\midrule
$e_{1}$  & $\alpha^{-1}  \delta_3 \beta_9  \alpha^{-1}   \beta_4 \alpha^{-1} \gamma_2 \beta_9 \delta_3 \alpha^{-1} \beta_9 \gamma_1$ \\	

$e_{2}$  & $\beta^{-1}_9 \beta^{-1}_4 \beta^{-1}_9$ \\
		
$e_{3}$  & $\beta^{-1}_4 \delta_3\beta_9 \delta_1$ \\
		
$e_{4}$  & $\alpha^{-1} \delta^{-1}_2 \alpha \delta_1$ \\
		
$e_{5}$  & $\alpha^{-1}  \delta^{-1}_3 \alpha  \delta_3$ \\
		
$e_{6}$  & $\beta_9 \delta_2 \beta^{-1}_9 \delta^{-1}_1$ \\
			
$e_{7}$ & $\delta^{-1}_2 \beta_4 \delta^{-1}_3 \beta^{-1}_9$ \\
			
$e_{8}$ & $\gamma_1 $ \\
			
$e_{9}$ & $\gamma_2$ \\
\bottomrule
\end{tabular}
\label{table:relation34p}
\end{table}

Consider the 3-manifold  $m295$ in Snappy  Census \cite{CullerDunfield:2014}, it is hyperbolic with volume 4.4153324775... and $$\pi_{1}(m295)=\langle s,t| st^{-4}s^2t^{-1}s^{-1}t^{4}s^{-2}t\rangle.$$
 \texttt{Magma} tells us  that $\pi_{1}(m295)$ and $\pi_{1}(M)$ above are isomorphic and finds an isomorphism $\Psi_2: \pi_1(m295)\longrightarrow \pi_1(M)$  given by
$$
\Psi_2(s)=\delta_3^{-1}\beta_9^{-1}, \quad \Psi_2(t)=\alpha\delta_3^{-1}\beta_9^{-1}.
$$

 As in the proof of Theorem \ref{thm:3pp},  we get that  $M$ is homeomorphic to $m295$. This finishes the proof of Theorem \ref{thm:34p}.

\end{document}